\documentclass[a4paper]{amsart}
\usepackage{amsmath,amsthm,amssymb}
\usepackage{mathrsfs}
\usepackage[T1]{fontenc}
\usepackage[shortlabels]{enumitem}
\usepackage{enumitem}
\usepackage{tikz} 
\usepackage{geometry}
\usepackage{mathtools}
\usepackage{fancyhdr}
\usepackage[pagewise]{lineno}\linenumbers
\usepackage{bbm}
\usepackage{hyperref}

\hypersetup{
    colorlinks=true,
    linkcolor=blue,
    citecolor=blue,
    filecolor=magenta
    }

\def\XXint#1#2#3{{\setbox0=\hbox{$#1{#2#3}{\int}$ }
\vcenter{\hbox{$#2#3$ }}\kern-.6\wd0}}

\newcommand{\RN}[1]{%
  \textup{\uppercase\expandafter{\romannumeral#1}}%
} 
\newcommand{\abs}[1]{\left|#1\right|}

\newcommand{\R}{\mathbb{R}}

\newcommand{\Ga}{\Gamma}   
\newcommand{\de}{\delta}

\newcommand{\bke}[1]{\left ( #1 \right )}

\newcommand{\bket}[1]{\left \{ #1 \right \}}

\newcommand{\calC}{{\mathcal C}}
\newcommand\Si{\Sigma}
\newcommand{\pd}{\partial}

\newcommand{\norm}[1]{\left\lVert#1\right\rVert}
\counterwithin*{equation}{section}
\usepackage{amsthm}
\newsavebox\myboxA
\newsavebox\myboxB
\newlength\mylenA

\newcommand*\xoverline[2][0.75]{%
    \sbox{\myboxA}{$\math#2$}%
    \setbox\myboxB\null
    \ht\myboxB=\ht\myboxA%
    \dp\myboxB=\dp\myboxA%
    \wd\myboxB=#1\wd\myboxA
    \sbox\myboxB{$\math\overline{\copy\myboxB}$}
    \setlength\mylenA{\the\wd\myboxA}
    \addtolength\mylenA{-\the\wd\myboxB}%
    \ifdim\wd\myboxB<\wd\myboxA%
       \rlap{\hskip 0.5\mylenA\usebox\myboxB}{\usebox\myboxA}%
    \else
        \hskip -0.5\mylenA\rlap{\usebox\myboxA}{\hskip 0.5\mylenA\usebox\myboxB}%
    \fi}

\newtheorem{theorem}{Theorem}
\newtheorem{lem}{Lemma}
\newtheorem*{lem*}{Lemma}

\newtheorem{prop}{Proposition}

\theoremstyle{definition}

\numberwithin{equation}{section}
\newtheorem*{innerthm}{\innerthmname}
\newcommand{\innerthmname}{}
\newenvironment{statement}[1]
 {\renewcommand{\innerthmname}{#1}\innerthm}
 {\endinnerthm}
\theoremstyle{definition}

\theoremstyle{remark}
\newtheorem*{remark}{Remark}
\title[Local and global regularity near boundary]{Local and global regularity for the Stokes and Navier-Stokes equations with the localized boundary data in the half-space}
\author{Kyungkeun Kang and Chanhong Min}

\begin{document}
\nolinenumbers
\maketitle

\begin{center}
Dedicated to Professor Vladim\'ir \u{S}ver\'ak on the occasion of his 65th birthday.
\end{center}
  

\begin{abstract}
We study the Stokes system with the localized boundary data in the half-space. We are concerned with the local regularity of its solution near the boundary away from the support of the given boundary data which are product forms of each spatial variable and the temporal variable.
We first show that if the boundary data are smooth in time, the corresponding solutions are also smooth in space and time near the boundary, even if the boundary data are only spatially integrable.
Secondly, if the normal component of the boundary data is absent, we are able to construct a solution such that its second normal derivatives of the tangential components become singular near the boundary. 
Perturbation argument enables us to construct solutions of the Navier-Stokes equations with similar singular behaviors near the boundary in the half-space as the case of Stokes system.
Lastly, we provide specific types of the localized boundary data to obtain the pointwise bounds of the solutions up to second derivatives. It turns out that such solutions are globally strong, and the second normal derivatives are, however, unbounded near the boundary. 
These results can be compared to the previous works in which only the normal component is present. In fact, the temporally non-smooth tangential boundary data can also cause spatially singular behaviors near the boundary, although such behaviors are milder than those caused by the normal boundary data of the same type.
\end{abstract}

  
\section{Introduction}
We consider the non-stationary Stokes system in the half-space:
\begin{equation}\label{stokeseq}
\begin{dcases}
\begin{aligned}
  \partial_t u-\Delta u+\nabla p&= 0\\
  \text{div } u&= 0
\end{aligned}
\end{dcases}
\quad\text{  in } \,\, \mathbb{R}_{+}^n \times (0,\infty),
\end{equation}
with zero initial data and non-zero boundary data:
\begin{equation}\label{stokesbdry}
    u\mid_{t=0}=0,\qquad u\mid_{x_n=0}=g.
\end{equation}
Here the boundary data \(g : \mathbb{R}^{n-1}\times \mathbb{R}_{+}\rightarrow \mathbb{R}^n\) are separable forms for the spatial and temporal variables, which are given by
\begin{equation}\label{bdry}
    g(y',s)=(g_1(y',s),\, g_2(y',s), \,\cdots, \,g_n(y',s))=\left(a_1g_1^{\mathcal{S}}(y')g_1^{\mathcal{T}}(s), \,a_2g_2^{\mathcal{S}}(y')g_2^{\mathcal{T}}(s), \,\cdots, \,a_ng_n^{\mathcal{S}}(y')g_n^{\mathcal{T}}(s)\right),
\end{equation}
where
\(a_i\in\mathbb{R}\), \(g_i^{\mathcal{S}}\geq0\) and \(g_i^{\mathcal{T}}\geq 0\) for \(i=1\), \(2\), \(\cdots\),  \(n\).
Our main concern is the local regularity of the solution near the boundary. 
Assuming that the boundary conditions are localized, we come up with the following simple situation away from the support of the boundary data:
\begin{equation}\label{SS-10}
\partial_t u-\Delta u+\nabla p=0, \quad
 {\rm div}\, u=0\qquad \mbox{in}\ \ B_{1}^+ \times (0, \infty)
\end{equation}
 together with the no-slip boundary condition given only on the flat part of the boundary, namely
\begin{equation}\label{SS-20}
u=0\qquad\mbox{ on }\,\,\Lambda'\times (0, \infty).
\end{equation}
Here we denote $B^+_1:=\bket{x\in\R^n: \abs{x}<1, x_n>0}$ (via translation, the center of the half ball can be assumed to be the origin) and $\Lambda':=B_1\cap\bket{x_n=0}$.
We emphasize that no condition is imposed on the rounded boundary, denoted by $\Lambda^+:=\partial B_1\cap\bket{x_n\ge 0}$.
One can imagine a similar situation to compare to the heat equation, i.e. $u_t-\Delta u=0$ for which the classical boundary regularity theory implies that
\begin{equation}\label{HE-30}
\norm{D^l_xD^m_t u}_{L^{\infty}(Q^+_{r})}\le C_{l, m, r, n}\norm{u}_{L^2(Q^+_1)},\qquad l, ~m\ge 0,
\end{equation}
where $Q_r^+=B^+_r\times(t-r^2, t)\subset B_1^+ \times (0, \infty)$ and \(r<1\).
Due to the non-local effect of the Stokes system \eqref{SS-10}-\eqref{SS-20}, such estimate \eqref{HE-30}  is, however, not clear and, as a matter of fact, it turned out that the estimate \eqref{HE-30}, in general, isn't true for the Stokes system (see \cite{Kang05}, \cite{CKca}, \cite{KangTsai22FE} and \cite{KangChang22}).

Indeed, in \cite{Kang05}, the first author constructed in three dimensions the examples of solutions such that their gradients are singular near the boundary, although the solutions themselves are bounded.
More specifically,
there exist solutions of the Stokes system \eqref{SS-10}-\eqref{SS-20} such that 
they are bounded and their derivatives are square integrable, but their normal derivatives are unbounded, i.e.,
\[
 \norm{u}_{L^{\infty}_{x,t}(Q_1^+)}+\norm{\nabla u}_{L^2_{x,t}(Q_1^+)}<\infty, \qquad\norm{\partial_{x_3} u}_{L^{\infty}_{x,t}(Q_{\frac{1}{2}}^+)}=\infty.
\]
Such results recently became refined in the sense that the 
boundary data causing a wider range of singularity are analyzed and the corresponding blow-up rates are calculated near the boundary in general dimensions, and furthermore, a construction of a solution to the Navier-Stokes equations is also made to show similar singular behaviors (see \cite{KangTsai22FE} for the details).
We remark that such constructed solution is an analog near the boundary to the well-known Serrin's example, namely $u(x,t)=\nabla h(x,t)$, where $h$ is harmonic in spatial variables, in the interior. We remark, however, that no non-trivial solution of the form $u(x,t)=\nabla h(x,t)$ near the boundary exists because of the homogeneous boundary condition \eqref{SS-20}.
It is worth referring to the a priori estimate  for \eqref{SS-10}-\eqref{SS-20} proved in \cite[Proposition 2]{Seregin00}, which is given as follows: For given $p,\, q\in (1,2]$ and any $r$ with $p\le r<\infty$, it holds that
\begin{equation}\label{Seregin-50}
\norm{\partial_t u}_{L^{q}_t L^r_{x}(Q_{\frac{1}{2}}^+)}+\norm{\nabla^2 u}_{L^{q}_t L^r_{x}(Q_{\frac{1}{2}}^+)}+\norm{\nabla \pi}_{L^{q}_t L^r_{x}(Q_{\frac{1}{2}}^+)}
\le C_{p,q,r}\bke{\norm{ u}_{L^{q}_t W^{1,p}_{x}(Q_{1}^+)}+\norm{\pi}_{L^{q}_t L^p_{x}(Q_{1}^+)}}.
\end{equation}
Therefore, it follows from the parabolic embedding that
\begin{equation}\label{Seregin-100}
\norm{u}_{\mathcal{C}^{\frac{\alpha}{2}}_t \mathcal{C}^{\alpha}_{x}(Q_{\frac{1}{2}}^+)}
< \infty,\qquad  0< \alpha <2\left(1-\frac{1}{q}\right).
\end{equation}
It was shown very recently in \cite{KangChang22} that the estimates \eqref{Seregin-50} and \eqref{Seregin-100} are  optimal. In fact, it turned out that the integrability in time is crucial to control $\nabla u$ (see \cite[Theorem 1.5]{KangChang22}). Indeed, it was also proved that if $q>2$, then $\nabla u$ is bounded.  More precisely, for given $p \in (1, \infty)$ and $q>2$, the estimate \eqref{Seregin-50} is valid and furthermore, it follows that
\begin{equation}\label{CK-150}
\nabla u\in\calC^{\frac{\alpha}{2}}_t \calC^{\alpha}_{x}(Q_{\frac{1}{4}}^+),\qquad  0< \alpha <1-\frac{2}{q}.
\end{equation}

As in the case of the Stokes system \eqref{SS-10}-\eqref{SS-20}, we also consider
 the following Navier-Stokes equations:
\begin{equation}\label{navierstokeseq-10}
 \partial_t u-\Delta u+(u\cdot \nabla)u+\nabla p=0, \quad
 {\rm div}\, u=0\qquad \mbox{in}\ \ B_{1}^+ \times (0, \infty),
\end{equation}
with the no-slip boundary condition \eqref{SS-20}.
Perturbation argument of the Stokes system enables a construction of solutions of \eqref{navierstokeseq-10} and \eqref{SS-20} with the singular gradients near the boundary as in the case of the Stokes system mentioned above (see \cite{KangTsai22} and \cite{KangChang22}). 

As a problem related to \eqref{HE-30} and \eqref{Seregin-50},
we also consider the Caccioppoli type inequality near the boundary for both the Stokes system and Navier-Stokes equations. Here we mean the Caccioppoli type inequality
for the Stokes system  \eqref{SS-10}-\eqref{SS-20} by
\begin{equation}\label{cacci-300}
\norm{\nabla u}^2_{L^2(Q_{\frac{1}{2}}^+)}\le C\norm{u}^2_{L^2(Q_{1}^+)}
\end{equation}
and similarly for the Navier-Stokes equations \eqref{navierstokeseq-10} and \eqref{SS-20} by
\begin{equation}\label{cacci-310}
\norm{\nabla u}^2_{L^2(Q_{\frac{1}{2}}^+)}\le C\bke{\norm{u}^2_{L^2(Q_{1}^+)}+\norm{u}^3_{L^3(Q_{1}^+)}},
\end{equation}
where $C$ in \eqref{cacci-300} and \eqref{cacci-310} are independent of respective $u$. The main concern for the above inequalities \eqref{cacci-300} and  \eqref{cacci-310} is whether or not the pressure appears in the right hand sides.
One can ask the same question for the interior case and it has been known that the Caccioppoli type inequalities are valid in the interior for both the Stokes system and Navier-Stokes equations (see \cite{Jin2013} and \cite{Wolf2015}). 
The answer for validation of the Caccioppoli type inequalities is, however, negative near the boundary.
Indeed, it was proved in \cite{CKca} that the Caccioppoli inequality \eqref{cacci-300} for the Stokes system and Caccioppoli type inequality \eqref{cacci-310} for the Navier-Stokes equations, in general, fail near the boundary.

On the other hand, instead of non-zero boundary data, we remark that a non-zero external force, may also cause similar singular behaviors to the gradients of the solutions near the boundary for the Stokes system and the Navier-Stokes equations (see \cite{Chang-Kang23}).  The constructed solution is an analog in the energy class to that of a shear flow type developed in \cite{Seregin-Sverak10}. We are not going to pursue this direction, since we deal with only the non-zero boundary data in this paper.

The main tool for our construction of such singular solutions is to use the explicit representation formula for the solution of the Stokes system in the half-space (see \cite{Gol60} and \cite{Sol68}). More precisely, using the Golovkin tensor
\begin{equation}\label{solformula-5}
K_{ij}(x,t)=-2\,\de_{ij}\,\pd_n\Ga(x,t)-4\,\pd_j\int_0^{x_n}\int_\Si\pd_n\Ga(z,t)\,\pd_iE(x-z)\,dz'\,dz_n-2\,\de_{nj}\pd_iE(x)\de(t),
\end{equation}
we recall that the solution \(w\) of (\ref{stokeseq})-(\ref{stokesbdry}) is given by 
\begin{equation}\label{solformula-10}
u_i(x,t)=\sum_{j=1}^n\int_{-\infty}^{\infty}\int_\Si K_{ij}(x-\xi',t-s)g_j(\xi',s)\,d\xi'\,ds.
\end{equation}

The singular solutions that have been formulated so far are based on the boundary data \eqref{bdry} with non-zero normal component but zero tangential components. 
One can expect singular behaviors of the solutions with only the tangential components of the localized boundary data, but this is not obvious.
The motivation of our study in this paper is to analyze behaviors of the solutions near the boundary when the boundary data are composed of the tangential components with zero normal component.
In addition, we would like to make distinct comparisons between singular behaviors caused by the tangential components and those of the normal component of the boundary data. 
By the way, we remark that all the examples of the boundary data constructed in the previous works were sufficiently regular in the spatial variables and some lack of regularity was assigned to the temporal variable, and thus, in this paper, we consider the opposite case as well.

We now state our main theorems.
Our first result states that if the temporal boundary data are smooth, then the solution of \eqref{stokeseq}-\eqref{stokesbdry} with pressure, are locally smooth away from the support of the boundary data. Therefore, we can say that a local smoothing effect is available if the boundary data are smooth in the temporal variable. 
For convenience, we first define the set
\begin{equation}\label{supp-g-10}
    A:=\left\{y'\in \mathbb{R}^{n-1}\left|~3<|y'|<4\sqrt{n},~-4<y_i<-3\right.\right\}.
\end{equation}
We are now ready to state the first main result.
\begin{theorem}\label{thm1}
Let \(j=1,\, 2,\, \cdots,\, n\) and \(A\) be the set defined in \eqref{supp-g-10}. Suppose that \(u\) is the solution of the Stokes system \eqref{stokeseq}-\eqref{stokesbdry} defined by \eqref{solformula} with the boundary data \(g\) given in \eqref{bdry} with \(a_i\in\mathbb{R}\) and \(\textup{ supp}(g_j^{\mathcal{T}})\subset \left(\frac{3}{4},\, \frac{7}{8}\right)\) such that
\[g_j^{\mathcal{S}}\in L^{1}(\mathbb{R}^{n-1}),\qquad  g_j^\mathcal{T}\in C_{c}^{\infty}(\mathbb{R}).\]
Then \(u, \,p\in C^{\infty}(Q_{1}^{+})\).
\end{theorem}
\begin{remark}
Singular behavior of the gradients of the solutions of the Stokes system \eqref{stokeseq}-\eqref{stokesbdry} was shown for the boundary data of the form in (1.3) with 
$a_n\neq 0$, $a_i=0$ for $i<n$, where $g_n^{\mathcal{S}}$ is sufficiently smooth and $g_n^{\mathcal{T}}$ is less regular (see \cite{Kang05}, \cite{CKca}, \cite{KangTsai22FE} and \cite{KangChang22} for more details). Theorem \ref{thm1} shows that the regularity of the temporal boundary data is crucial for the regularity of the solutions
away from the support of boundary data.
\end{remark}

Our second theorem discusses the local regularity of the solutions where now the spatial boundary data are smooth with the same assumptions on the support as described in the previous theorem.

\begin{theorem}\label{thm2}
Let \(j=1,~ 2,~\cdots, ~n-1\), \(1<p<\infty\) and $A$ be the set defined in \eqref{supp-g-10}. Suppose that the boundary data \(g\) with \(a_i\in\mathbb{R}\), \(\textup{supp}(g_j^{\mathcal{S}})\subset A\) and \(\textup{supp}(g_j^{\mathcal{T}})\subset \left(\frac{3}{4},~ 1\right)\) in \eqref{bdry} satisfy
\[
g_j^{\mathcal{S}}\in W^{2,1}(A),\qquad g_j^\mathcal{T}\in L^{\infty}(\mathbb{R}) \setminus \dot{B}_{pp}^{\frac{1}{2}-\frac{1}{2p}}(\mathbb{R}).
\]
Then the solution of the Stokes system \eqref{stokeseq} - \eqref{stokesbdry} defined by \eqref{solformula} with the boundary data \(g\) satisfy
 \begin{equation}\label{thm2estimate1}
        \left\|u\right\|_{L^{\infty}(Q_{2}^{+})}+\left\|\nabla u\right\|_{L^{\infty}(Q_{2}^{+})} + \sum_{j=1}^{n-1}\sum_{i=1}^{n}\left\|\partial_i \partial_j u\right\|_{L^{\infty}(Q_{2}^{+})}+\left\|\partial_n^2 u_n\right\|_{L^{\infty}(Q_{2}^{+})}<\infty,
\end{equation}
\begin{equation}\label{thm2estimate2}
    \left\|\partial_n^2 u_i\right\|_{L^p(Q_{\frac{1}{2}}^+)}=\infty,\qquad i=1, \,2,\,\cdots,\, n-1.
\end{equation}
Similar construction can be made for
the solution of the Navier-Stokes equations \eqref{navierstokeseq-10} with sufficiently small \(a_i\) and the no-slip boundary condition \eqref{SS-20} satisfying \eqref{thm2estimate1} with \(Q_1^+\) in place of \(Q_2^+\) and \eqref{thm2estimate2}.
\end{theorem}
\begin{remark}
The result of Theorem \ref{thm2} is similar to that of \cite[Theorem 1.1]{KangChang22}, where the tangential components of the boundary data are absent. One comparison is that the normal derivatives are singular in \cite[Theorem 1.1]{KangChang22} but it is not the case in Theorem \ref{thm2}, where the second normal derivatives are singular. This shows that the tangential components of the boundary data give rise to a milder singularity compared to that of the normal component of the boundary data of the same type.
\end{remark}

Our third theorem discusses the global pointwise estimates of the solutions and their \(L^p\) estimates. Also the pointwise estimate of the corresponding pressure is addressed.
\begin{theorem}\label{thm3}
Let \(n\geq 2\) and \(1\leq k \leq n-1\). Suppose that \(u\) is the solution of the Stokes equation \eqref{stokeseq}-\eqref{stokesbdry} defined by \eqref{solformula} with the boundary data \(g\) given in \eqref{bdry} with \(a_j={\delta_{kj}}\) and \(g_k^{\mathcal{S}}\in C_c^{\infty}(B_1^{'})\), \(g_k^\mathcal{T}\in C_c(\mathbb{R})\,\cap\, C^1(0,1)\), where 
\begin{align*}
    \textup{supp}(g_k^{\mathcal{T}})\subset\left[\frac{1}{4},~1\right],~~g_k^{\mathcal{T}}(s)=(1-s)^a~~\left(s\in \left[\frac{1}{2}, ~1\right],~ 0<a\leq \frac{1}{2}\right).
\end{align*}
Then \(u\) satisfies the following bounds:
\begin{align*}
    |u_i(x,t)| &\leq \frac{C}{\left< x\right>^{n}}\left\|g_k^{\mathcal{S}}\right\|_{W^{1,\infty}},\qquad 1\leq i \leq n,\\ 
    |\partial_j u_i(x,t)| &\leq C\left(\delta_{j<n}\frac{1}{\left< x\right>^{n+1}}+\left[\delta_{i<n}\frac{1}{\left<x\right>^n}+\delta_{in}\frac{1}{\left<x\right>^{n+1}}+\left(\frac{1}{\left<x\right>^{n+2}}+\mathbbm{1}_{|x|<2}\textup{LN}\right)\right]\delta_{jn}\right)\left\|g_k^{\mathcal{S}}\right\|_{W^{2,\infty}}, \quad1\leq i, ~j\leq n,\\ 
    |\partial_l\partial_j u_i(x,t)|&\leq C\left[ \delta_{l<n}\delta_{j<n}\frac{1}{\left<x\right>^{n+2}}+\left(\delta_{ln}\delta_{j<n}+\delta_{l<n}\delta_{jn}\right)\left(\frac{1}{\left<x\right>^{n+1}}+\mathbbm{1}_{|x|<2}\textup{LN}\right)\right.\\
    &\quad+\delta_{jn}\delta_{ln}\left(\mathbbm{1}_{|x|>2}\left(\frac{1}{\left<x\right>^{n+2}}+\frac{\textup{LN}}{\left<x\right>^{n}(x_n+1)^{2a}}\right)+\mathbbm{1}_{|x|<2}\frac{\log\left(1+\frac{t}{x_n}\right)}{(x_n^2+|t-1|)^{1-a}}\right)\\ 
    &\quad\left.+\delta_{n2}\sigma\mathbbm{1}_{|x|<2}\textup{LN} \log\left(2+\frac{1}{x_2^2+|t-1|}\right)\right]\left\|g_k^{\mathcal{S}}\right\|_{W^{3,\infty}},\qquad 1\leq i,~ j,~l\leq n,  
\end{align*}
for \(x\in \mathbb{R}_{+}^{n}\) and \(t\in [0,2]\), where \(\sigma=\delta_{i<n}\delta_{jn}\delta_{ln}\),
and
\begin{align*}
    \textup{LN}:=(x_n^2+|t-1|)^{a-\frac{1}{2}}+\delta_{a=\frac{1}{2}}\log \left(2+\frac{1}{x_n^2+|t-1|}\right).
\end{align*}
and \(C\) is independent \(x\), \(t\), and \(g_k\).
We have that
\begin{enumerate}
    \item \(u\in L_{x,t}^{p,q}(\mathbb{R}_+^n\times(0,2))\) for \(p\in\left(1,\infty\right]\) and \(q \in [1,\infty]\),
    \item \(\nabla u\in L_{x,t}^{p,p}(\mathbb{R}_+^n\times(0,2))\) for \(p\in\left(1,\frac{3}{1-2a}\right)\) if \(a<\frac{1}{2}\) and \(p\in(1,\infty)\) if \(a=\frac{1}{2}\),
    \item\(\nabla^2 u\in L_{x,t}^{p,p}(\mathbb{R}_+^n\times(0,2))\) for \(p \in \left[1, \frac{3}{2(1-a)}\right)\) if \(a<\frac{1}{2}\) and \(p\in [1,3)\) if \(a=\frac{1}{2}\).
\end{enumerate}
In particular, \(u\) belongs to the energy class \(L_t^\infty L_x^2 \cap L_{t}^2 \dot{H}_x^1(\mathbb{R}_+^n\times (0,2))\).\\

Moreover, for \(n\geq 2\), we have the following pointwise estimate for the pressure \(p\): for any multiindex \(\alpha\in \mathbb{Z}_0^n\) and \(\epsilon>0\),
\begin{align*}
    |\nabla^\alpha p(x,t)|&\leq C\mathbbm{1}_{\frac{1}{4}\leq t\leq 1}\frac{1}{\left<x\right>^{n+|\alpha|}}+C\mathbbm{1}_{t\leq \frac{1}{2}}\frac{1}{\left<x\right>^{n-1+|\alpha|}}+C\mathbbm{1}_{|x|<2}\left(\mathbbm{1}_{\frac{1}{2}\leq t<1}\frac{1}{(1-t)^{\frac{1+\epsilon}{2}-a}}+\mathbbm{1}_{t> 1}\frac{\textup{LN}^*}{(t-1)^{\frac{\epsilon}{2}}}\right)\\
    &+C\mathbbm{1}_{|x|>2}\left(\mathbbm{1}_{\frac{1}{2}\leq t<1}+\mathbbm{1}_{t>1}\right)\frac{\textup{LN}^*}{|x|^{n-1+|\alpha|}},
\end{align*}
where 
\begin{equation*}
    \textup{LN}^*:=|t-1|^{a-\frac{1}{2}}\mathbbm{1}_{a<\frac{1}{2}}+\log\left(2+\frac{1}{|t-1|}\right)\mathbbm{1}_{a=\frac{1}{2}}
\end{equation*}
and \(C=C(\alpha,\,\epsilon, \,g_k)\) is a positive constant.
\end{theorem}

\begin{remark}
\begin{enumerate}
\item  The pointwise estimates in Theorem \ref{thm3} can be compared to \cite[Proposition 3.1]{KangTsai22FE}, where the pointwise estimates caused by the normal component of the boundary data of the same type were obtained.

 \item The above pointwise estimate for the pressure \(p\) shows that
 \begin{enumerate}
     \item  \(p\in L_{x,t}^{q,r}(\mathbb{R}_+^{n}\times(0,2))\) for any \(q> \frac{n}{n-1}\) and \(1\leq r <\frac{2}{1-2a}\),
     \item\(\nabla p\in L_{x,t}^{q,r}(\mathbb{R}_+^{n}\times(0,2))\) for any \(q>1\) and \(1\leq r<\frac{2}{1-2a}\),
     \item \(\nabla^k p\in L_{x,t}^{q,r}(\mathbb{R}_+^{n}\times(0,2))\) for any \(k\geq  2\), \(q\geq 1\) and \(1\leq  r<\frac{2}{1-2a}\).
 \end{enumerate}
\end{enumerate}
\end{remark}

We now discuss the lower bounds of the second normal derivatives of the tangential components of the solutions.
\begin{theorem}\label{thm4}
Let \(n\geq 3\) and \(1\leq k\leq n-1\). Suppose that \(u\) is the solution of the Stokes equation \eqref{stokeseq}-\eqref{stokesbdry} defined by \eqref{solformula} with the boundary data \(g\) given in \eqref{bdry} with \(a_j={\delta_{kj}}\), and \(g_k^{\mathcal{S}}\in C_c^{3}(B_1^{'})\), \(g_k^\mathcal{T}\in C_c(\mathbb{R})\cap C^1(0,1)\), where 
\begin{align*}
    \textup{supp}(g_k^{\mathcal{T}})\subset\left[\frac{1}{4},~1\right],~~g_k^{\mathcal{T}}(s)=(1-s)^a~~\left(s\in \left[\frac{1}{2},~ 1\right],~ 0<a\leq \frac{1}{2}\right).
\end{align*}
Furthermore, choose \(g_k^{\mathcal{S}}\) to be the following product form
\begin{equation*}
    g_k^{\mathcal{S}}(y')=\prod_{j=1}^{n-1}\mathcal{G}(y_j),
\end{equation*}
where \(\mathcal{G}:\mathbb{R}\rightarrow\mathbb{R}\) is smooth, even, supported in \(\left(-\frac{4}{5\sqrt{n-1}}, \frac{4}{5\sqrt{n-1}}\right)=:(-r,r)\), \(\mathcal{G}(x)=1\) for \(|x|<\frac{1}{2\sqrt{n-1}}=:p\) and \(\mathcal{G}'(x)\leq 0\) for \(x>0\) (see \textup{Appendix 8.3}). Then the followings hold:\\
1) If \(i<n, ~i\neq k\), then for any \(\alpha\), \(\beta>0\), and \(|x'|\geq 3\) with \(x_i\geq \alpha\), \(x_k\geq \beta\), \(x_n\leq 1\) (see \textup{Appendix 8.4}), we have
\begin{equation*}
    \left|\partial_{x_n}^2 u_i(x,1)\right|\geq
    \begin{cases}
\displaystyle\frac{C_1 \alpha\beta}{|x'|^{n+2}}\log\frac{2}{x_n}-\frac{C_2}{|x'|^{n-2}} &\textup{if  } a=\frac{1}{2}, \\
\displaystyle\frac{C_1\alpha\beta}{|x'|^{n+2}x_n^{1-2a}}-\frac{C_2}{|x'|^{n-2}} &\textup{if  } 0<a<\frac{1}{2},
\end{cases}
\end{equation*}
where \(C_1\) and \(C_2\) are positive constants independent of \(x
\), \(\alpha\), \(\beta\).\\
2) If \(i=k\), and we further assume that \(\mathcal{G}''(x)=0\) at \(x=\frac{13}{20\sqrt{n-1}}=:q\), \(\mathcal{G}''(2q-x)=-\mathcal{G}''(x)\) for \(p\leq x\leq r\) and \(\mathcal{G}''(x)=0\) on \(\left(p,p+\frac{1}{20\sqrt{n-1}}\right)\cup\left(q-\frac{1}{20\sqrt{n-1}},q\right)\) (see \textup{Appendix 8.3}). Then for \(|x'|\geq 3\) and either \(x_k<p\), or \(x_k>r\) and \(\left(\sum_{i\neq k,n}x_i^2\right)^\frac{1}{2}<\frac{3}{2}\) (see \textup{Appendix 8.4}), we have
\begin{equation*}
    \left|\partial_{x_n}^2 u_i(x,1)\right|\geq
\begin{cases}
        \displaystyle\frac{C_1}{|x'|^{n+2}}\log\frac{2}{x_n}-\frac{C_2}{|x'|^{n-2}}& \textup{if  } a=\frac{1}{2},\\
        \displaystyle\frac{C_1}{|x'|^{n+2} x_n^{1-2a}}-\frac{C_2}{|x'|^{n-2}} & \textup{if  } 0<a<\frac{1}{2},
    \end{cases}
\end{equation*}
where \(C_1\) and \(C_2\) are positive constants independent of \(x\).
\end{theorem}
\begin{remark}
The blow-up estimates in Theorem \ref{thm4} can be also compared to \cite[Theorem 1.1]{KangTsai22FE}. Different comparison is also due to opposing contributions via the tangential components and the normal components of the boundary data, respectively.
\end{remark}

We organize this paper as follows: In Section \ref{sec2}, we remind some known formulas and results and introduce
useful lemmas to use later. Section \ref{sec3}, Section \ref{sec5} and Section \ref{sec6} are devoted to presenting the proofs of 
Theorem \ref{thm1}, Theorem \ref{thm3} and Theorem \ref{thm4}, respectively. The proof of Theorem \ref{thm2} is provided separately in Section \ref{sec4} for the Stokes system and in Section \ref{sec7} for the Navier-Stokes equations. In Appendix, details are given for the proofs of some lemmas introduced in Section \ref{sec2} and some figures are drawn for the localized boundary data and the regions of singularity for the second derivatives of the solution presented in Theorem \ref{thm4}.  

\section{Preliminaries}\label{sec2}

We recall that the \textit{heat kernel} \(\Gamma\) and the \textit{fundamental solution} \(E\) of \(-\Delta\) are given by
\begin{equation*}
    \Gamma(x,t)=\left\{\begin{matrix}
(4\pi t)^{-\frac{n}{2}}e^{-\frac{|x|^2}{4t}} & \text{ for } t>0, \\
 0& \text{ for } t\leq 0, \\
\end{matrix}\right.
~~~\text{ and } ~~
E(x)=\begin{cases}
\frac{1}{n(n-2)|B_1|}\frac{1}{|x|^{n-2}} & \text{ for } n\geq 3, \\
 -\frac{1}{2\pi}\log |x|& \text{ for } n=2. \\
\end{cases}
\end{equation*}

We also remind the following functions: for \(x\in\mathbb{R}^{n}\) and \(t\in\mathbb{R}\),
\begin{equation}\label{def-A}
    A(x,t):=\int_{\Sigma}\Gamma(z',0,t)E(x-z')dz'=\int_{\Sigma}\Gamma(x'-z',0,t)E(z',x_n)dz',
\end{equation}
\begin{equation}\label{def-B}  
     B(x,t):=\int_{\Sigma}\Gamma(x-z',t)E(z',0)dz'=\int_{\Sigma}\Gamma(z',x_n,t)E(x'-z',0)dz',
\end{equation}
where $\Sigma:=\mathbb{R}^{n-1}$
and for \(x\in\mathbb{R}_{+}^{n}\) and \(t\in\mathbb{R}\),
\begin{equation*}\label{def-C}
    C_i(x,t):=\int_{0}^{x_n}\int_{\Sigma}\partial_n \Gamma(x-z,t)\partial_i E(z)dz,\qquad 1\leq i\leq n.
\end{equation*}

The following relations are known between above functions and their proofs are given in \cite{KangTsai22FE}:
\begin{align}
    \partial_{n}C_i(x,t)&=\partial_{i}C_n(x,t)+\partial_{i}\partial_{n}B(x,t),\qquad i\neq n,\label{rel-C1}\\ 
    \partial_{n}C_n(x,t)&=-\sum_{k=1}^{n-1}\partial_{k}C_k(x,t)-\frac{1}{2}\partial_n \Gamma(x,t). \label{rel-C2}
\end{align}

We consider the following nonstationary Stokes system in \(\mathbb{R}_+^n\times(0,\infty)\):
\begin{equation}\label{stokeseq2}
\begin{dcases}
\begin{aligned}
  \partial_t w-\Delta w+\nabla \pi&= f\\
  \text{div } w&= 0
\end{aligned}
\end{dcases}
\qquad\text{in   }  \mathbb{R}_{+}^n \times (0,\infty),
\end{equation}
with zero initial data and non-zero boundary data:
\begin{equation}\label{stokesbdry2}
    w\mid_{t=0}=0,\qquad w\mid_{x_n=0}=g.
\end{equation}

We collect some estimates of the solution \(w\) which will be used later.
\begin{prop}
Let \(1<p,\, q<\infty\).
Suppose that \(w\) is the solution of \eqref{stokeseq2}-\eqref{stokesbdry2} with  \(g=0\). Then, \(w\) 
satisfies the following:
\begin{enumerate}
\renewcommand{\labelenumi}{\arabic{enumi}.}
\renewcommand{\theenumi}{1.\arabic{enumi}}
\item\textup{(\cite{KangChang22}, Proposition 2.6)}\label{1} If \(f=\nabla\cdot F\) with \(F_{in}\mid_{x_n=0}=0\), then
\begin{equation*}
    \left\|w(t)\right\|_{L^p(\mathbb{R}_+^n)}\leq C\int_{0}^{t}(t-s)^{-\frac{1}{2}}\left\|F(s)\right\|_{L^{p}(\mathbb{R}_+^n)}ds,\qquad 0<t<\infty.
\end{equation*}
\item\textup{(\cite{CKca}, Proposition 2.3)}\label{2} If \(f=\nabla\cdot F\) with \(F\in L^{q}(0,\infty;L^{p}(\mathbb{R}_+^n))\), \(F\mid_{x_n=0}\in L^q(0,\infty; \dot{B}_{pp}^{-\frac{1}{p}}(\mathbb{R}^{n-1}))\), then 
\begin{equation*}
    \left\|\nabla w\right\|_{L^{q}(0,\infty;L^{p}(\mathbb{R}_+^n))}\leq C\left(\left\|F\right\|_{L^{q}(0,\infty;L^{p}(\mathbb{R}_+^n))}+\left\|F\mid_{x_n=0}\right\|_{L^q(0,\infty; \dot{B}_{pp}^{-\frac{1}{p}}(\mathbb{R}^{n-1}))} \right).
\end{equation*}
\item\textup{(\cite{ChangJin}, Theorem 1.2)}\label{3} If \(f=\nabla\cdot F\) with \(F\in L^q(\mathbb{R}^n\times\mathbb{R}_+)\), then
\begin{equation*}
    \left\|w\right\|_{L^{q}(\mathbb{R}_+^n\times(0,T))}\leq CT^{\frac{1}{2}}\left\|F\right\|_{L^q(\mathbb{R}^{n}\times (0, T))}.
\end{equation*}
\end{enumerate}
\end{prop}

The \textit{Golovkin tensor} \(K_{ij}(x,t) : \mathbb{R}_{+}^n\times \mathbb{R}\rightarrow \mathbb{R}\) is the Poisson kernel of the nonstationary Stokes system (\ref{stokeseq2}) in the half-space \(\mathbb{R}_+^n\) with zero external force. A solution \(w\) of (\ref{stokeseq2})-(\ref{stokesbdry2}) with \(f=0\) is given by 
\begin{equation}\label{solformula}
    w_i(x,t)=\sum_{j=1}^{n}\int_{-\infty}^{\infty}\int_{\Sigma} K_{ij}(x-y',t-s)g_j(y',s)dy'ds,
\end{equation}
where we extend \(g_j(x',t)=0\) for \(t<0\). Here the Golovkin tensor \(K_{ij}(x,t)\) is
\begin{equation}\label{golovkin}
    K_{ij}(x,t)=-2\delta_{ij}\partial_n\Gamma(x,t)-4\partial_j C_i(x,t)-2\delta_{nj}\partial_i E(x)\delta(t),\qquad i, j=1,\,2,\,\cdots,\, n
\end{equation}
and the associated pressure tensor \(k_j\) is given by
\begin{equation}
    k_j(x,t)=2\partial_j\partial_n E(x)\delta(t)+2\delta_{nj}E(x)\delta '(t)+\frac{2}{t}\partial_j A(x,t),\qquad j=1,\,2,\,\cdots,\, n.
\end{equation}

We remark that the last term of the above formula is not integrable and thus the formula has to be understood in either one of the following senses:
\begin{align*}
    p(x,t)&=2\sum_{i=1}^{n}\partial_i\partial_n\int_{\Sigma}E(x-y')g_i(y',t)dy'+2\int_{\Sigma}E(x-y')\partial_t g_n(y',t)dy'\\
    &+\sum_{j=1}^{n}\partial_j\int_{-\infty}^{\infty}\int_{\Sigma}\frac{2}{t-s}A(x-y',t-s)\left[g_j(y',s)-g_j(y',t)\right]dy'ds,
\end{align*}

or
\begin{align}
    \nonumber p(x,t)&=2\sum_{i=1}^{n}\partial_i\partial_n\int_{\Sigma}E(x-y')g_i(y',t)dy'+2\int_{\Sigma}E(x-y')\partial_t g_n(y',t)dy'\\
    &-4\sum_{j=1}^{n}(\partial_t-\Delta_{x'})\int_{-\infty}^{\infty}\int_{\Sigma}\partial_j A(x-y',t-s)g_j(y',s)dy'ds. \label{secondformulapressure}
\end{align}

The proof of the above representations for \(n=3\), is given in \cite{Sol153}.

We now recall some useful estimates of the previous functions and their proofs can also be found in \cite{Sol153}. 
\begin{align*}
    \left|\partial_x^l\partial_t^m \Gamma(x,t)\right|&\lesssim \frac{1}{(|x|^2+t)^{\frac{l+n}{2}+m}},\\
    \left|\partial_x^l\partial_t^m A(x,t)\right|&\lesssim\frac{1}{t^{m+\frac{1}{2}}(|x|^2+t)^{\frac{l+n-2}{2}}}\qquad l+n\geq 3,\\
    \left|\partial_{x'}^l\partial_{x_n}^k\partial_t^m B(x,t)\right|& \lesssim  \left\{\begin{matrix}
\displaystyle\frac{1}{ (|x|^2+t)^{\frac{l+n-2}{2}}(x_n^2+t)^{\frac{k+1}{2}+m} } &  m \geq 0,\\
\displaystyle\frac{e^{-\frac{x_n^2}{10t}}}{(|x|^2+t)^{\frac{l+n-2}{2}}t^{\frac{k+1}{2}}} & m =0, \\
\end{matrix}\right. \\
\left|\partial_{x'}^l\partial_{x_n}^k\partial_t^m C_i(x,t)\right|&\lesssim \frac{1}{ t^{m+\frac{1}{2}}(|x|^2+t)^{\frac{l+n-1}{2}}(x_n^2+t)^{\frac{k}{2}} } \qquad 1\leq i \leq n.
\end{align*}

We finally list some useful lemmas.

\begin{lem} \textup{(\cite[Lemma 3.2]{KangChang22})}\label{lemma1}                                                                                      
Let \(\Gamma_1(x,t):=\frac{1}{(4\pi t)^\frac{1}{2}}e^{-\frac{x^2}{4t}}\) be the 1-dimensional heat kernel and $B$ be the function defined in \eqref{def-B}.
Then, for \(1\leq j \leq n-1\) and \(|x'|\geq 1\) we have
\begin{align*}
\partial_j\partial_n B(x,t)&=-\partial_n \Gamma_1(x_n, t)\left(\partial_j E(x',0)+J(x',t)\right)    
\end{align*}
and \(|J(x',t)|\leq c_nt^{\frac{1}{2}}\).
\end{lem}

\begin{lem} \textup{(\cite[Lemma 2.1]{KangTsai22FE})}\label{lemma2}
For positive \(L,~a,~d\) and \(k\) we have
\begin{align*}
    \int_{0}^{L}\frac{r^{d-1}}{(r+a)^k}dr\lesssim
    \begin{cases}
 L^d(a+L)^{-k}&  \textup{if  }k<d,\\
 L^{d}(a+L)^{-d}(1+\log_{+}\frac{L}{a})&  \textup{if  } k=d,\\
 L^{d}(a+L)^{-d}a^{-(k-d)}&  \textup{if  } k>d.\\
\end{cases}
\end{align*}
\end{lem}

\begin{lem}\textup{(\cite[Lemma 2.2]{KangTsai22})}\label{lemma3}
Let \(a>0\), \(b>0\), \(k>0\), \(m>0\) and \(k+m>d\). Let \(0\neq x\in \mathbb{R}^{d}\) and
\begin{align*}
    I:=\int_{\mathbb{R}^d}\frac{dz}{(|z|+a)^{k}(|z-x|+b)^m}.
\end{align*}
Then, with \(R=\max \left\{|x|,\,a,\,b\right\}\sim |x|+a+b\),
\begin{align*}
    I\lesssim R^{d-k-m}+\delta_{kd}R^{-m}\log\frac{R}{a}+\delta_{md}R^{-k}\log\frac{R}{b}+\mathbbm{1}_{k>d}R^{-m}a^{d-k}+\mathbbm{1}_{m>d}R^{-k}b^{d-m}.
\end{align*}
\end{lem}

\begin{lem}\textup{(\cite[Lemma 3.3]{KangTsai22FE})}\label{Kestimate}
Let \(n\geq 3\) and \(\displaystyle K(x',t)=\int_{\Sigma}\frac{e^{-\frac{|x'-z'|^2}{4t}}}{|z'|^{n-2}}dz'\) for \(x'\in \Sigma=\mathbb{R}^{n-1}\). For any \(m\in[2,\infty)\), we have
\begin{align*}
    K(x',s)&\geq \left(\frac{m}{(m+1)|x'|}\right)^{n-2}s^{\frac{n-1}{2}}(4\pi)^{\frac{n-1}{2}}\left(1-2^{\frac{n-1}{4}}e^{-\frac{|x'|^2}{8m^2s}}\right),\\
    K(x',s)&\leq \left(\frac{m}{(m-1)|x'|}\right)^{n-2}s^{\frac{n-1}{2}}(4\pi)^{\frac{n-1}{2}}+\frac{C}{|x'|^{n-2}}s^{\frac{n-1}{2}}e^{-\frac{|x'|^2}{8m^2s}},
\end{align*}
where \(C=C(n)>0\) is a constant.
\end{lem}

\begin{lem}\label{alg}
If \(0<s<t<1\), \(0<a<1\) and \(c>0\), then
\begin{align*}
    \frac{e^{-\frac{x_n^2}{cs}}}{(1-t+s)^{1-a}}\lesssim \frac{1}{(x_n^2+1-t+s)^{1-a}}.
\end{align*}
\end{lem}
\begin{proof}
The proof will be provided in the Appendix \ref{apdx81}.
\end{proof}
Finally we will frequently use the following integral estimates.
\begin{lem}\label{integralestimates} For any \(c>0\),
\begin{equation*}
    \int_{0}^{t}\frac{1}{s^k}e^{-\frac{x_n^2}{cs}}ds\lesssim\begin{cases}
\frac{e^{-\frac{x_n^2}{2ct}}}{x_n^{2k-2}} & \textup{if }k>1,\\ 
 e^{-\frac{x_n^2}{ct}}\log\left(1+\frac{ct}{x_n^2} \right )& \textup{if } k=1,\\ 
e^{-\frac{x_n^2}{2ct}} & \textup{if } 0<k<1.
\end{cases}
\end{equation*}
\end{lem}
\begin{proof}
The case \(k=1\) follows from the inequality (5) in \cite{Gau}.
The case \(k>1\) follows from the integration by parts after a suitable change of variables with the aid of the estimate \(u^{k-1}e^{-u}\lesssim e^{-\frac{u}{2}}\). Lastly, the case \(0<k<1\) can be similarly obtained as the case \(k>1\).
\end{proof}

\begin{lem}\label{integralidentity}
For \(b,\, c>0\) and \(\alpha\in \left[\frac{1}{2},\,1\right)\), \(\beta \in \left[\frac{1}{2},\,1\right]\),
\begin{equation*}
    \int_{0}^{1}\frac{1}{(u+c)^\alpha (u+b)^{\beta}}du\lesssim 
    \begin{cases}
\log\left(2+\frac{1}{\sqrt{b+c}}\right) & \textup{if } \alpha=\beta=\frac{1}{2},\\ 
\displaystyle\frac{1}{(c+b)^{\alpha+\beta-1}} & \textup{if } \frac{1}{2}\leq \alpha <1,\quad \frac{1}{2}\leq \beta <1,\\ 

\displaystyle\frac{1}{(b+c)^{\alpha}}\log\left(\frac{c}{b}+1\right) & \textup{if } \frac{1}{2}\leq \alpha <1,\quad \beta=1.
\end{cases}
\end{equation*}

\end{lem}
\begin{proof}
    The proof will be provided in the Appendix 8.2.
\end{proof}

\section{Proof of Theorem \ref{thm1} for the Stokes equations}\label{sec3}

Since the Stokes system is linear, we treat the cases of the normal and tangential components of the boundary data separately.

$\bullet$ \textbf{(Step 1)} (Case that \(g_j=0\) for \(j\neq n\) and $g_n \neq 0$) We first consider the case that only the normal component of the boundary data is not zero, i.e. \(j=n\). \\

1) We first estimate the tangential components of the velocity \(w\).
Let \(1\leq i \leq n-1\). Reminding that
\begin{align*}
    K_{in}(x,t)=-4\partial_n C_i(x,t)-2\partial_i E(x)\delta(t)
    =-4\partial_i C_n(x,t)-4\partial_i\partial_n B(x,t)-2\partial_i E(x)\delta(t),
\end{align*}
we note that \(w_i\) is decomposed as follows.
\begin{align*}
    w_i(x,t)=&\int_{0}^t\int_{\Sigma}K_{in}(x-y',t-s)g_n(y',s)dy'ds\\
    =&-4\int_{0}^{t}\int_{\Sigma}\partial_i C_n(x-y',t-s)g_n(y',s)dy'ds
    -4\int_{0}^{t}\int_{\Sigma}\partial_i\partial_n B(x-y',t-s)g_n(y',s)dy'ds\\
    &-2\int_{\Sigma}\partial_i E(x-y')g_n(y',t)dy'\\
    :=&w_i^{\mathcal{C}}(x,t)+w_i^{\mathcal{B}}(x,t)+w_i^{\mathcal{E}}(x,t).
\end{align*}
We control $w_i^{\mathcal{C}}$, $w_i^{\mathcal{B}}$ and $w_i^{\mathcal{E}}$ separately.\\
\textbf{(Estimate of \(w_i^{\mathcal{E}}\))} We note that
\begin{equation*}
    w_i^{\mathcal{E}}(x,t)=-2g_n^{\mathcal{T}}(t)\int_{\Sigma}\partial_i E(x-y')g_n^{\mathcal{S}}(y')dy'.
\end{equation*}
Thus we find that
\begin{equation*}
    \partial_x^{l}\partial_t^{m}w_i^{\mathcal{E}}(x,t)=-2\partial_t^{m}g_n^{\mathcal{T}}(t)\int_{\Sigma}\partial_x^{l}\partial_i E(x-y')g_n^{\mathcal{S}}(y')dy',
\end{equation*}
and we get that using \(|x'-y'|\gtrsim 1\) for any \(x'\in B'(0,1)\) and \(y'\in A\), 
\begin{align*}
    \left|\partial_x^l\partial_t^m w_i^{\mathcal{E}}(x,t)\right|
    \lesssim \left|\partial_t^m g_n^{\mathcal{T}}(t)\right|\int_{A}\left|\partial_x^{l}\partial_i E(x-y')\right|\left|g_n^{\mathcal{S}}(y')\right|dy'ds
    \lesssim \left\|\partial_t^m g_n^{\mathcal{T}}\right\|_{L^{\infty}}\left\|g_n^{\mathcal{S}}\right\|_{L^{1}}.
\end{align*}
Thus we obtain that
\begin{equation}\label{thm1-10}
    \left\|\partial_x^l\partial_t^m w_i^{\mathcal{E}}\right\|_{L^{\infty}(Q_{+}^{1})}<\infty.
\end{equation}
\textbf{(Estimate of \(w_i^{\mathcal{C}}\))} We note that
\begin{equation*}
    w_i^{\mathcal{C}}(x,t)=-4\int_0^{t}g_n^{\mathcal{T}}(t-s)\int_{\Sigma}\partial_i C_n(x-y',s)g_n^{\mathcal{S}}(y')dy'ds
\end{equation*}
and since \(\partial_t^k g_n^{\mathcal{T}}(0)=0\) for any \(k\geq 0\), we find that 
\begin{equation*}
    \partial_{x'}^l\partial_t^m w_i^{\mathcal{C}}(x,t)=-4\int_0^t\partial_t^{m}g_n^{\mathcal{T}}(t-s)\int_{\Sigma}\partial_{x'}^l\partial_i C_n(x-y',s)g_n^\mathcal{S}(y')dy'ds.
\end{equation*}

Thus we find that 
\begin{align}\label{est-wic}
 \nonumber  \left|\partial_{x'}^l\partial_t^m w_i^{\mathcal{C}}(x,t)\right|
   &\lesssim \int_0^t\left|\partial_t^{m}g_n^{\mathcal{T}}(t-s)\right|\int_{A}\left|\partial_{x'}^l\partial_i C_n(x-y',s)\right|\left|g_n^\mathcal{S}(y')\right|dy'ds.\\
\nonumber   &\lesssim \int_0^t\left|\partial_t^{m}g_n^{\mathcal{T}}(t-s)\right|\int_{A}\frac{1}{s^{\frac{1}{2}}(|x-y'|^2+s)^{\frac{l+n}{2}}}\left|g_n^\mathcal{S}(y')\right|dy'ds\\
   &\lesssim \sqrt{t}\left\|\partial_t^m g_n^{\mathcal{T}}\right\|_{L^{\infty}}\left\|g_n^{\mathcal{S}}\right\|_{L^{1}},
\end{align}
where we used that \((|x-y'|^2+s)^{\frac{\alpha}{2}}\geq |x'-y'|^{\alpha}\gtrsim 1\) for any \(\alpha>0\).\\

Next, the normal derivatives of $w_i^{\mathcal{C}}$ are calculated as follows using \eqref{rel-C2}:

\begin{align*}
    \partial_{x_n}\partial_{x'}^l\partial_t^m w_i^{\mathcal{C}}(x,t)=&-4\int_0^t\partial_t^{m}g_n^{\mathcal{T}}(t-s)\int_{\Sigma}\partial_{x_n}\partial_{x'}^l\partial_i C_n(x-y',s)g_n^\mathcal{S}(y')dy'ds\\
   =&4\sum_{k=1}^{n-1}\int_0^t\partial_t^{m}g_n^{\mathcal{T}}(t-s)\int_{\Sigma}\partial_{x_k}\partial_{x'}^l\partial_i C_k(x-y',s)g_n^\mathcal{S}(y')dy'ds\\
    &+2\int_0^t\partial_t^{m}g_n^{\mathcal{T}}(t-s)\int_{\Sigma}\partial_{x'}^l\partial_i \Gamma(x-y',s)g_n^\mathcal{S}(y')dy'ds.
\end{align*}
As shown above in the case of tangential derivatives, similar computations show that 
\begin{align*}
    \left|\partial_{x_n}\partial_{x'}^l\partial_t^m w_i^{\mathcal{C}}(x,t)\right|
    \lesssim (t+\sqrt{t})\left\|\partial_t^m g_n^{\mathcal{T}}\right\|_{L^{\infty}}\left\|g_n^{\mathcal{S}}\right\|_{L^{1}}.
\end{align*}
Hence, we obtain
\begin{equation}\label{1-normal}
    \left\|\partial_{x_n}\partial_{x'}^l\partial_t^m w_i^{\mathcal{C}}\right\|_{L^{\infty}(Q_{+}^{1})}\lesssim \left\|\partial_t^m g_n^{\mathcal{T}}\right\|_{L^{\infty}}\left\|g_n^{\mathcal{S}}\right\|_{L^{1}}.
\end{equation}

Next, the second normal derivatives of $w_i^{\mathcal{C}}$ is estimated as follows using \eqref{rel-C1} and \eqref{rel-C2}:
\begin{align*}
    \partial_{x_n}^2\partial_{x'}^l\partial_t^m w_i^{\mathcal{C}}(x,t)
    &=4\sum_{k=1}^{n-1}\int_{0}^{t}\partial_t^m g_n^{\mathcal{T}}(t-s)\int_{\Sigma}\partial_{x_k}\partial_{x'}^l\partial_i\partial_k C_n(x-y',s)g_n^{\mathcal{S}}(y')dy'ds\\
    &+4\sum_{k=1}^{n-1}\int_{0}^{t}\partial_{t}^m g_n^{\mathcal{T}}(t-s)\int_{\Sigma}\partial_{x_k}\partial_{x'}^l\partial_i\partial_n\partial_k B(x-y',s)g_n^\mathcal{S}(y')dy'ds\\
    & +2\int_{0}^{t}\partial_{t}^{m}g_n^{\mathcal{T}}(t-s)\int_{\Sigma}\partial_n\partial_{x'}^{l}\partial_i \Gamma(x-y',s)g_n^{\mathcal{S}}(y')dy'ds\\
    &=I_1+I_2+I_3.
\end{align*}

Here each \(I_k\) is estimated as follows:
\begin{align*}
    |I_1|&\lesssim \sum_{k=1}^{n-1}\int_{0}^{t}\left|\partial_t^{m}g_n^{\mathcal{T}}(t-s)\right|\int_{A}\left|\partial_{x_k}\partial_{x'}^l\partial_i\partial_k C_n(x-y',s)\right|\left|g_n^{\mathcal{S}}(y')\right|dy'ds\\
    &\lesssim\int_{0}^{t}\left|\partial_t^{m}g_n^{\mathcal{T}}(t-s)\right|\int_{A}\frac{1}{s^{\frac{1}{2}}(|x-y'|^2+s)^{\frac{l+n+2}{2}}}\left|g_n^{\mathcal{S}}(y')\right|dy'ds\\
    &\lesssim \sqrt{t}\left\|\partial_t^m g_n^{\mathcal{T}}\right\|_{L^{\infty}}\left\|g_n^{\mathcal{S}}\right\|_{L^{1}},
\end{align*}

\begin{align}
   \nonumber |I_2|&\lesssim \sum_{k=1}^{n-1}\int_{0}^{t}\left|\partial_t^{m}g_n^{\mathcal{T}}(t-s)\right|\int_{A}\left|\partial_{x_k}\partial_{x'}^l\partial_i\partial_n\partial_k B(x-y',s)\right|\left|g_n^\mathcal{S}(y')\right|dy'ds\\
   \nonumber &\lesssim \int_{0}^{t}\left|\partial_t^{m}g_n^{\mathcal{T}}(t-s)\right|\int_{A}\frac{e^{-\frac{x_n^2}{10s}}}{s(|x-y'|^2+s)^{\frac{l+n+1}{2}}}\left|g_n^{\mathcal{S}}(y')\right|dy'ds\\
   \nonumber &\lesssim \left\|\partial_t^m g_n^{\mathcal{T}}\right\|_{L^{\infty}}\left\|g_n^{\mathcal{S}}\right\|_{L^{1}} \int_{0}^t\frac{e^{-\frac{x_n^2}{10s}}}{s}ds\\
    &\lesssim \left\|\partial_t^m g_n^{\mathcal{T}}\right\|_{L^{\infty}}\left\|g_n^{\mathcal{S}}\right\|_{L^{1}}\log \left(1+\frac{10t}{x_n^2}\right), \label{E}
\end{align}
where the last inequality follows from Lemma \ref{integralestimates}.
Thus we conclude that 
\begin{equation*}
    |I_2|\lesssim \left\|\partial_t^m g_n^{\mathcal{T}}\right\|_{L^{\infty}}\left\|g_n^{\mathcal{S}}\right\|_{L^{1}}\log \left(1+\frac{10t}{x_n^2}\right).
\end{equation*}

Unfortunately, we cannot directly estimate the \(L^\infty\) norm of \(I_2\). However using the above estimate for \(I_2\) and the estimate \(\int_{0}^{1}\abs{\log \left(1+\frac{10t}{x_n^2}\right)}^{p}dx_n\leq c(p)\) for any \(p\in [1,\infty)\) and \(t\in (0,1)\), we see that for any \(p\in[1,\infty)\) we have the bound
\begin{equation*}
    \left\|I_2\right\|_{L^p(Q_{+}^{1})}<\infty.
\end{equation*}

Finally,
\begin{align*}
    |I_3|&\lesssim 
    \int_{0}^{t}\left|\partial_t^{m}g_n^{\mathcal{T}}(t-s)\right|\int_{A}\left|\partial_n\partial_{x'}^{l}\partial_i \Gamma(x-y',s)\right|\left|g_n^{\mathcal{S}}(y')\right|dy'ds\\
    &\lesssim
    \int_{0}^{t}\left|\partial_t^{m}g_n^{\mathcal{T}}(t-s)\right|\int_{A}\frac{1}{(|x-y'|^2+s)^{\frac{l+n+2}{2}}}\left|g_n^{\mathcal{S}}(y')\right|dy'ds \\
    &\lesssim t\left\|\partial_t^m g_n^{\mathcal{T}}\right\|_{L^{\infty}}\left\|g_n^{\mathcal{S}}\right\|_{L^{1}}.
\end{align*}

Hence we conclude that for \(p \in (1,\infty)\),
\begin{equation}\label{2ndwic}
    \left\|\partial_{x_n}^2\partial_{x'}^l\partial_t^m w_i^{\mathcal{C}}\right\|_{L^{p}(Q_{+}^{1})}\lesssim \left\|\partial_t^m g_n^{\mathcal{T}}\right\|_{L^{\infty}}\left\|g_n^{\mathcal{S}}\right\|_{L^{1}}.
\end{equation}

Next, the we estimate the third normal derivatives of $w_i^{\mathcal{C}}$.  Using \eqref{rel-C1}, \eqref{rel-C2} and the fact that the function \(B(x,t)\) satisfies the heat equation,
\begin{align*}
    \partial_{x_n}^3\partial_{x'}^l\partial_t^m w_i^{\mathcal{C}}(x,t)
    &=-4\sum_{k=1}^{n-1}\sum_{j=1}^{n-1}\int_{0}^{t}\partial_t^m g_n^{\mathcal{T}}(t-s)\int_{\Sigma}\partial_{x_k}\partial_{x'}^l\partial_i\partial_k\partial_j C_j(x-y',s)g_n^{\mathcal{S}}(y')dy'ds\\
    &-2\sum_{k=1}^{n-1}\int_{0}^{t}\partial_t^m g_n^{\mathcal{T}}(t-s)\int_{\Sigma}\partial_{x_k}\partial_{x'}^{l}\partial_i\partial_k\partial_n \Gamma(x-y',s)g_n^{\mathcal{S}}(y')dy'ds\\
    &-4\sum_{k=1}^{n-1}\sum_{j=1}^{n-1}\int_{0}^{t}\partial_t^m g_n^{\mathcal{T}}(t-s)\int_{\Sigma}\partial_{x_k}^2\partial_{x'}^l\partial_i\partial_j^2B(x-y',s)g_n^{\mathcal{S}}(y')dy'ds\\
    &+4\sum_{k=1}^{n-1}\int_{0}^{t}\partial_t^{m+1}g_n^{\mathcal{T}}(t-s)\int_{\Sigma}\partial_k^2\partial_{x'}^{l}\partial_i B(x-y',s)g_n^{\mathcal{S}}(y')dy'ds\\
    &+2\int_{0}^{t}\partial_{t}^{m}g_n^{\mathcal{T}}(t-s)\int_{\Sigma}\partial_n^2\partial_{x'}^{l}\partial_i \Gamma(x-y',s)g_n^{\mathcal{S}}(y')dy'ds\\
    &=J_1+J_2+J_3+J_4+J_5.
\end{align*}

Since there are only the tangential derivatives in the functions \(B\) and \(C_j\) (\(1\leq j\leq n-1\)), using the estimates of the functions given in Preliminaries, we see that
\begin{equation}\label{3rdwic}
    \left\|\partial_{x_n}^3\partial_{x'}^l\partial_t^m w_i^{\mathcal{C}}\right\|_{L^{\infty}(Q_{+}^{1})}\leq \sum_{k=1}^5 |J_k|\lesssim \left(\left\|\partial_t^{m} g_n^{\mathcal{T}}\right\|_{L^{\infty}}+\left\|\partial_t^{m+1} g_n^{\mathcal{T}}\right\|_{L^{\infty}}\right)\left\|g_n^{\mathcal{S}}\right\|_{L^{1}}.
\end{equation}

For arbitrary \(k\)th normal derivative, we divide the case where \(k\) is even or odd. If \(k\) is odd, after removing all the normal derivatives applied to the functions \(C_j\) (\(1\leq j \leq n\)) and \(B\), we obtain that 
\begin{equation*}
    \left\|\partial_{x_n}^{k}\partial_{x'}^l\partial_t^m w_i^{\mathcal{C}}\right\|_{L^{\infty}(Q_{+}^{1})}\lesssim \sum_{s=0}^{\frac{k-1}{2}}\left\|\partial_t^{m+s} g_n^{\mathcal{T}}\right\|_{L^{\infty}}\left\|g_n^{\mathcal{S}}\right\|_{L^{1}}.
\end{equation*}

On the other hand, if \(k\) is even, we cannot remove the first normal derivative applied to the function \(B\) and this leads to the time singularity at \(0\) as shown in \eqref{E}. Thus we cannot obtain the \(L^{\infty}\) estimates directly from the pointwise estimates of the functions \(C_k\) and \(B\) given in the Preliminaries. However we still have the following \(L^p\) estimates for any \(p\in [1,\infty)\):
\begin{equation*}
    \left\|\partial_{x_n}^{k}\partial_{x'}^l\partial_t^m w_i^{\mathcal{C}}\right\|_{L^{p}(Q_{+}^{1})} \lesssim \begin{cases}
\displaystyle\sum_{s=0}^{\frac{k-2}{2}}\left\|\partial_t^{m+s} g_n^{\mathcal{T}}\right\|_{L^{\infty}}\left\|g_n^{\mathcal{S}}\right\|_{L^{1}} &  \textup{if } k \geq 2,\\
\left\|\partial_t^{m} g_n^{\mathcal{T}}\right\|_{L^{\infty}}\left\|g_n^{\mathcal{S}}\right\|_{L^{1}} & \textup{if } k=0. 
\end{cases}
\end{equation*}

Then from the interpolation inequality
\begin{equation*}
    \left\|f\right\|_{L^{\infty}(Q_{+}^{1})}\leq C\left(\left\|\nabla f\right\|_{L^{\infty}(Q_{+}^{1})}^{\frac{n+1}{p+n+1}}\left\|f\right\|_{L^p(Q_{+}^{1})}^{\frac{p}{p+n+1}}+\left\|f\right\|_{L^p(Q_{+}^{1})}\right),\qquad 1\leq p <\infty,
\end{equation*} 
we obtain that
\begin{equation}\label{kthwic}
    \left\|\partial_{x_n}^{k}\partial_{x'}^l\partial_t^m w_i^{\mathcal{C}}\right\|_{L^{\infty}(Q_{+}^{1})} <\infty.
\end{equation}

\textbf{(Estimate of \(w_i^{\mathcal{B}}\))} We note that
\begin{equation*}
    w_i^{\mathcal{B}}(x.t)=-4\int_{0}^{t}g_n^{\mathcal{T}}(t-s)\int_{\Sigma}\partial_i\partial_n B(x-y', s)g_n^{\mathcal{S}}(y')dy'ds.
\end{equation*}

Thus
\begin{equation*}
    \partial_{x_n}^k \partial_{x'}^l \partial_t^{m}w_i^{\mathcal{B}}(x.t)=-4\int_{0}^{t}\partial_t^m g_n^{\mathcal{T}}(t-s)\int_{\Sigma}\partial_{x_n}^k \partial_{x'}^l\partial_i\partial_n B(x-y', s)g_n^{\mathcal{S}}(y')dy'ds.
\end{equation*}
We already have discussed how to estimate the above integral when we estimated \(w_i^{\mathcal{C}}\) and thus we obtain
\begin{equation}\label{wib}
    \left\|\partial_{x_n}^{k}\partial_{x'}^l\partial_t^m w_i^{\mathcal{B}}\right\|_{L^{\infty}(Q_{+}^{1})} <\infty.
\end{equation}\\
2) We now estimate the normal component of the velocity \(w\). Recalling that
\begin{align*}
    K_{nn}(x,t)
    =-2\partial_n \Gamma(x,t)-4\partial_n C_n(x,t)-2\partial_n E(x)\delta(t)
    =4\sum_{k=1}^{n-1}\partial_k C_k(x,t)-2\partial_n E(x)\delta(t),
\end{align*}
we note that \(w_n\) is decomposed as follows
\begin{align*}
    w_n(x,t)&=\int_{0}^{t}\int_{\Sigma}K_{nn}(x-y',t-s)g_n(y',s)dy'ds\\
    &=4\sum_{k=1}^{n-1}\int_{0}^t\int_{\Sigma}\partial_k C_k(x-y',t-s)g_n(y',s)dy'ds-2\int_{\Sigma}\partial_n E(x-y')g_n(y',t)dy'\\
    &=w_n^{\mathcal{C}}(x,t)+w_n^{\mathcal{E}}(x,t).
\end{align*}
\textbf{(Estimates of \(w_n^{\mathcal{E}}\) and \(w_n^{\mathcal{C}}\)})   \(w_n^{\mathcal{E}}\) and \(w_n^{\mathcal{C}}\) enjoy a similar estimate as those of \(w_i^{\mathcal{E}}\) and \(w_i^{\mathcal{C}}\) for $i<n$ respectively, and therefore, it follows that
\begin{equation}\label{est-wn-wnc}
    \left\|\partial_x^l\partial_t^m w_n^\mathcal{E}\right\|_{L^{\infty}(Q_{+}^{1})}+\left\|\partial_{x_n}^{k}\partial_{x'}^l\partial_t^m w_n^{\mathcal{C}}\right\|_{L^{\infty}(Q_{+}^{1})} <\infty.
\end{equation}
\newline
$\bullet$ \textbf{(Step 2)}  (Case that $g_j \neq 0$ for $j\neq n$ and \(g_n=0\)) Secondly, we treat the case where
only the tangential component of the boundary data is not zero, i.e., \(1\leq j \leq n-1\).\\ 

1) We first estimate the tangential components of the velocity \(w\). Let \(1\leq i \leq n-1\). Since
\begin{align*}
    K_{ij}(x,t)=-2\delta_{ij}\partial_n \Gamma(x,t)-4\partial_j C_i(x,t),
\end{align*}
\(w_i\) is expressed as follows:
\begin{align*}
    w_i(x,t)&=-2\delta_{ij}\int_{0}^{t}\int_{\Sigma}\partial_n \Gamma(x-y',t-s)g_j(y',s)dy'ds-4\int_{0}^{t}\int_{\Sigma}\partial_j C_i(x-y',t-s)g_j(y',s)dy'ds\\
    &=w_i^{\mathcal{H}}(x,t)+w_i^{\mathcal{C}}(x,t).
\end{align*}
\textbf{(Estimate of \(w_i^{\mathcal{H}}\) and \(w_i^{\mathcal{C}}\))} For \(w_i^{\mathcal{H}}\) we note that
\begin{align*}
    w_i^{\mathcal{H}}(x,t)=-2\delta_{ij}\int_{0}^{t}g_{j}^{\mathcal{T}}(t-s)\int_{\Sigma}\partial_n\Gamma(x-y',s)g_j^{\mathcal{S}}(y')dy'ds
\end{align*}
and since \(\partial_t^k g_j^{\mathcal{T}}(0)=0\) for any \(k\geq 0\), we find that
\begin{align*}
\partial_{x_n}^{k}\partial_{x'}^{l}\partial_t^mw_i^{\mathcal{H}}(x,t)=-2\delta_{ij}\int_{0}^{t}\partial_t^m g_j^{\mathcal{T}}(t-s)\int_{\Sigma}\partial_{x_n}^{k}\partial_{x'}^{l}\partial_n \Gamma(x-y',s)g_j^{\mathcal{S}}(y')dy'ds.
\end{align*}

Thus we find that
\begin{align*}
     \left|\partial_{x_n}^{k}\partial_{x'}^{l}\partial_t^m w_i^{\mathcal{H}}(x,t)\right|&\lesssim\int_{0}^{t}\left|\partial_t^m g_j^{\mathcal{T}}(t-s)\right|\int_{A}\left|\partial_{x_n}^{k}\partial_{x'}^{l}\partial_n \Gamma(x-y',s)\right|\left|g_j^{\mathcal{S}}(y')\right|dy'ds\\
     &\lesssim\int_{0}^{t}\left|\partial_t^m g_j^{\mathcal{T}}(t-s)\right|\int_{A}\frac{1}{(|x-y'|^2+s)^{\frac{n+k+l+1}{2}}}\left|g_j^{\mathcal{S}}(y')\right|dy'ds\\
     &\lesssim t\left\|\partial_t^m g_j^{\mathcal{T}}\right\|_{L^{\infty}}\left\|g_j^{\mathcal{S}}\right\|_{L^{1}}.
\end{align*}

Since \(w_i^{\mathcal{C}}\) enjoys the same estimates as the \textbf{Step 1}, we find that
\begin{equation}\label{tangentialwicwih}
 \left\|\partial_{x_n}^{k}\partial_{x'}^{l}\partial_t^m w_i^{\mathcal{H}}\right\|_{L^{\infty}(Q_+^1)}+ \left\|\partial_{x_n}^{k}\partial_{x'}^{l}\partial_t^m w_i^{\mathcal{C}}\right\|_{L^{\infty}(Q_+^1)}<\infty.
\end{equation}

2) We now estimate the normal component of the velocity \(w\).\\
Recalling that $K_{nj}(x,t)=-4\partial_j C_n(x,t)$, we note that \(w_n\) is written as follows
\begin{align*}
    w_n(x,t)=-4\int_0^t\int_{\Sigma}\partial_jC_n(x-y',t-s)g_j(y',s)dy'ds.
\end{align*}
By the calculations given in \textbf{Step 1}, we find that \(w_n\) satisfies the estimate 
\begin{equation}\label{tangentialwnc}
    \left\|\partial_x^l\partial_t^m w_n^\mathcal{C}\right\|_{L^{\infty}(Q_{+}^{1})}<\infty.
\end{equation}

$\bullet$ \textbf{(Step 3)} We now estimate the pressure \(p\).\\
We use the second formula \eqref{secondformulapressure} for the pressure to obtain that
\begin{align*}
    p(x,t)&=2\partial_n^2\int_{\Sigma}E(x-y')g_n(y',t)dy'+2\int_{\Sigma}E(x-y')\partial_t g_n(y',t)dy'\\
    &-4\partial_t\int_{0}^{t}\int_{\Sigma}\partial_n A(x-y',t-s)g_n(y',s)dy'ds+4\Delta_{x'}\int_0^t\int_{\Sigma}\partial_n A(x-y',t-s)g_n(y',s)dy'ds\\
    &=I_1+I_2+I_3+I_4.
\end{align*}

The estimates of \(I_1\) and \(I_2\) are easy and the result is 
\begin{equation*}
  |I_1|+|I_2|\lesssim \left\|g_n^{\mathcal{S}}\right\|_{L^1}(\left\|g_n^{\mathcal{T}}\right\|_{L^\infty}+\left\|\partial_t g_n^{\mathcal{T}}\right\|_{L^\infty}).
\end{equation*}

For \(I_3\) we have that using \(g_n^{\mathcal{T}}(0)=0\),
\begin{align*}
    I_3=-4\int_{0}^{t}\int_{A}\partial_n A(x-y',s)\partial_t g_n(y',t-s)dy'ds.
\end{align*}

Thus we have that
\begin{align*}
    |I_3|&\lesssim \int_{0}^{t}\int_{A}\left|\partial_n A(x-y',s)\right|\left|\partial_t g_n(y',t-s)\right|dy'ds\\
    &\lesssim \int_{0}^{t}\int_{A}\frac{1}{s^{\frac{1}{2}}(|x-y'|^2+s)^{\frac{n-1}{2}}}\left|\partial_t g_n(y',t-s)\right|dy'ds\\
    &\lesssim \sqrt{t} \left\|g_n^{\mathcal{S}}\right\|_{L^1}\left\|\partial_t g_n^{\mathcal{T}}\right\|_{L^{\infty}}.
\end{align*}

Finally for \(I_4\), we have that
\begin{align*}
    I_4=4\int_{0}^{t}\int_{\Sigma}\Delta_{x'}\partial_n A(x-y',s)g_n(y',t-s)dy'ds.
\end{align*}

Thus we have that
\begin{align*}
    |I_4|&\lesssim \int_{0}^{t}\int_{\Sigma}\left|\Delta_{x'}\partial_n A(x-y',s)\right|\left|g_n(y',t-s)\right|dy'ds.\\
    &\lesssim \int_{0}^{t}\int_{\Sigma}\frac{1}{s^{\frac{1}{2}}(|x-y'|^2+s)^{\frac{n+1}{2}}}\left|g_n(y',t-s)\right|dy'ds\\
    &\lesssim \sqrt{t}\left\| g_n^{\mathcal{T}}\right\|_{L^{\infty}} \left\|g_n^{\mathcal{S}}\right\|_{L^1}.
\end{align*}

Hence we conclude that
\begin{equation}\label{p}
    |p(x,t)|\lesssim
   \left\|g_n^{\mathcal{S}}\right\|_{L^1}(\left\|g_n^{\mathcal{T}}\right\|_{L^\infty}+\left\|\partial_t g_n^{\mathcal{T}}\right\|_{L^\infty}).
\end{equation}

For the time derivatives of \(p\), we find that since \(\partial_t^k g_n^{\mathcal{T}}(0)=0\) for any \( k\geq 0\),
\begin{align*}
    \partial_t^m p(x,t)
    &=2\partial_n^2\int_{\Sigma}E(x-y')\partial_t^m g_n(y',t)dy'+2\int_{\Sigma}E(x-y')\partial_t^{m+1}g_n(y',t)dy'\\
    &-4\int_{0}^t\int_{\Sigma}\partial_n A(x-y',t-s)\partial_t^{m+1}g_n(y',t-s)dy'ds +4\Delta_{x'}\int_{0}^{t}\int_{\Sigma}\partial_n A(x-y',s)\partial_t^m g_n(y',s)dy'ds.
\end{align*}

And we obtain the estimates
\begin{equation}\label{p_t}
    |\partial_t^m p(x,t)|\lesssim  \left\|g_n^{\mathcal{S}}\right\|_{L^1}\left(\left\|\partial_t^m g_n^{\mathcal{T}}\right\|_{L^\infty}+\left\|\partial_t^{m+1} g_n^{\mathcal{T}}\right\|_{L^\infty}\right),
\end{equation}
by following the same proof for the estimate of \(p\).\\
The higher mixed derivatives of \(p\) follows from the previous estimates of \(w\) and the Stokes equations.
\qedsymbol


\section{Proof of Theorem \ref{thm2} for the Stokes equations}\label{sec4}
$\bullet$\,\,\textbf{(Step 1) Estimate of \(w_i\) and \(\nabla w_i\).} \,\,We recall that for \(1\leq i \leq n\),
\begin{align*}
    w_i(x,t)&=-2\delta_{ij}\int_{0}^t\int_{\Sigma}\partial_n \Gamma(x-y',t-s)g_j(y',s)dy'ds-4\int_{0}^t\int_{\Sigma}\partial_j C_i(x-y',t-s)g_j(y',s)dy'ds\\
    &=I_1+I_2.
\end{align*}
It is rather straightforward that
\begin{align*}
    |I_1|&\lesssim\int_{0}^{t}\int_{\Sigma}\frac{1}{(|x-y'|^2+s)^{\frac{n+1}{2}}}\left|g_j(y',t-s)\right|dy'ds\lesssim t\left\|g_j^\mathcal{T}\right\|_{L^{\infty}}\left\|g_j^\mathcal{S}\right\|_{L^{1}},\\
    |I_2|&\lesssim\int_{0}^{t}\int_{\Sigma}\frac{1}{s^{\frac{1}{2}}(|x-y'|^2+s)^{\frac{n}{2}}}\left|g_j(y',t-s)\right|dy'ds\lesssim \sqrt{t}\left\|g_j^\mathcal{T}\right\|_{L^{\infty}}\left\|g_j^\mathcal{S}\right\|_{L^{1}}.
\end{align*}
Thus we find that
\begin{equation}\label{thm2w}
    \left\|w_i\right\|_{L^{\infty}(Q_2^{+})}\lesssim\left\|g_j^\mathcal{T}\right\|_{L^{\infty}}\left\|g_j^\mathcal{S}\right\|_{L^{1}}.
\end{equation}
Next, we estimate the tangential derivatives.
For \(1\leq k\leq n-1\), we note that
\begin{align*}
    \partial_k w_i(x,t)&=-2\delta_{ij}\int_{0}^t\int_{\Sigma}\partial_k\partial_n \Gamma(x-y',s)g_j(y',t-s)dy'ds-4\int_{0}^t\int_{\Sigma}\partial_k\partial_j C_i(x-y',s)g_j(y',t-s)dy'ds\\
    &=J_1+J_2.
\end{align*}
Similarly as in the above computations, it follows that 
\begin{align*}
    |J_1|&\lesssim\int_{0}^{t}\int_{\Sigma}\frac{1}{(|x-y'|^2+s)^{\frac{n+2}{2}}}\left|g_j(y',t-s)\right|dy'ds\lesssim t\left\|g_j^\mathcal{T}\right\|_{L^{\infty}}\left\|g_j^\mathcal{S}\right\|_{L^{1}},\\
    |J_2|&\lesssim\int_{0}^{t}\int_{\Sigma}\frac{1}{s^{\frac{1}{2}}(|x-y'|^2+s)^{\frac{n+1}{2}}}\left|g_j(y',t-s)\right|dy'ds\lesssim \sqrt{t}\left\|g_j^\mathcal{T}\right\|_{L^{\infty}}\left\|g_j^\mathcal{S}\right\|_{L^{1}}.
\end{align*}
Therefore, we obtain for \(1\leq k \leq n-1\),
\begin{equation}\label{thm2w_k}
    \left\|\partial_k w_i\right\|_{L^{\infty}(Q_2^{+})}\lesssim\left\|g_j^\mathcal{T}\right\|_{L^{\infty}}\left\|g_j^\mathcal{S}\right\|_{L^{1}}.
\end{equation}
Next, we estimate the normal derivatives. We compute
\begin{align*}
    \partial_n w_i(x,t)&=-2\delta_{ij}\int_{0}^{t}\int_{\Sigma}\partial_n^2 \Gamma(x-y',s)g_j(y',t-s)ds-4\int_{0}^{t}\int_{\Sigma}\partial_n\partial_j C_j(x-y',s)g_j(y',t-s)dy'ds\\
    &=2\delta_{ij}\int_{0}^{t}\int_{\Sigma}\partial_n^2 \Gamma(x-y',s)g_j(y',t-s)ds-4\int_{0}^{t}\int_{\Sigma}\partial_j^2 C_n(x-y',s)g_j(y',t-s)dy'ds\\
    &-4\int_{0}^{t}\int_{\Sigma}\partial_j^2\partial_n B(x-y',s)g_j(y',t-s)dy'ds\\
    &=K_1+K_2+K_3.
\end{align*}
The term $K_1$ and $K_2$ are controlled as follows:
\begin{align*}
    |K_1|&\lesssim\int_{0}^{t}\int_{\Sigma}\frac{1}{(|x-y'|^2+s)^{\frac{n+2}{2}}}\left|g_j(y',t-s)\right|dy'ds\lesssim t\left\|g_j^\mathcal{T}\right\|_{L^{\infty}}\left\|g_j^\mathcal{S}\right\|_{L^{1}},\\
    |K_2|&\lesssim\int_{0}^{t}\int_{\Sigma}\frac{1}{s^{\frac{1}{2}}(|x-y'|^2+s)^{\frac{n+1}{2}}}\left|g_j(y',t-s)\right|dy'ds\lesssim \sqrt{t}\left\|g_j^\mathcal{T}\right\|_{L^{\infty}}\left\|g_j^\mathcal{S}\right\|_{L^{1}}.
\end{align*}
For \(K_3\), with the aid of Lemma \ref{lemma1}, we have
\begin{align*}
    K_3&=-4\int_{0}^{t}\int_{\Sigma}\partial_n\Gamma_1(x_n,s)(\partial_j E(x'-y',0)+J(x'-y',s))\partial_j g_j(y',t-s)dy'ds
\end{align*}
and thus, it follows that by Lemma \ref{integralestimates},
\begin{align*}
    |K_3|&\lesssim\int_{0}^{t}\int_{\Sigma}\left|\partial_n \Gamma_1(x_n,s)\right|\left(\left|\partial_j E(x'-y',0)\right|+|J(x'-y',s)|\right)\left|\partial_j g_j(y',t-s)\right| dy'ds\\
    &\lesssim \left\|g_j^\mathcal{T}\right\|_{L^{\infty}}\int_{0}^{t}\int_{A}\frac{x_n}{s}e^{-\frac{x_n^2}{4s}}\left(\frac{1}{|x'-y'|^{n-1}}+s^{\frac{1}{2}}\right)\left|\partial_j g_j^{\mathcal{S}}(y')\right|dy'ds\\
    &\lesssim \left\|g_j^\mathcal{T}\right\|_{L^{\infty}}\left\|\nabla g_j^\mathcal{S}\right\|_{L^{1}}\int_{0}^{t}\left(\frac{x_n}{s^{\frac{1}{2}}}e^{-\frac{x_n^2}{4s}}+\frac{x_n}{s}e^{-\frac{x_n^2}{4s}}\right)ds\\
    &\lesssim \left\|g_j^\mathcal{T}\right\|_{L^{\infty}}\left\|\nabla g_j^\mathcal{S}\right\|_{L^{1}}\left(1+x_n\log\left(1+\frac{t}{4x_n^2}\right)\right) \lesssim \left\|g_j^\mathcal{T}\right\|_{L^{\infty}}\left\|\nabla g_j^\mathcal{S}\right\|_{L^{1}}.
\end{align*}
This leads to conclude that
\begin{equation}\label{thm2w_n}
    \left\|\partial_n w_i\right\|_{L^{\infty}(Q_2^{+})}\lesssim \left\|g_j^\mathcal{T}\right\|_{L^{\infty}}\left\|g_j^\mathcal{S}\right\|_{W^{1,1}}.
\end{equation}
\\
$\bullet$\,\,\textbf{(Step 2) Estimate of \(\nabla^2 w_i\).}\,\,
We now consider the second derivatives of \(w_i\). We begin with two the tangential derivatives. If \(1\leq l,\, k \leq n-1\) we have that
\begin{align*}
    \partial_l\partial_k w_i&=-2\delta_{ij}\int_{0}^{t}\int_{\Sigma}\partial_l\partial_k\partial_n \Gamma(x-y',t-s)g_j(y',s)dy'ds-4\int_{0}^{t}\int_{\Sigma}\partial_l\partial_k\partial_j C_i(x-y',t-s)g_j(y',s)dy'ds\\
    &=I_1+I_2.
\end{align*}
We seperately estimate $I_1$ and $I_2$.
\begin{align*}
    |I_1|&\lesssim \int_{0}^{t}\int_{\Sigma}\frac{1}{(|x-y'|^2+s)^{\frac{n+3}{2}}}|g_j(y',t-s)|dy'ds\lesssim t\left\|g_j^\mathcal{T}\right\|_{L^{\infty}}\left\|g_j^\mathcal{S}\right\|_{L^{1}},\\
     |I_2|&\lesssim \int_{0}^{t}\int_{\Sigma}\frac{1}{s^{\frac{1}{2}}(|x-y'|^2+s)^{\frac{n+2}{2}}}|g_j(y',t-s)|dy'ds\lesssim \sqrt{t}\left\|g_j^\mathcal{T}\right\|_{L^{\infty}}\left\|g_j^\mathcal{S}\right\|_{L^{1}}.
\end{align*}
Thus we find that
\begin{equation}\label{thm2w_lk}
	\left\| \partial_l\partial_k w_i\right\|_{L^{\infty}(Q_2^{+})}\lesssim\left\|g_j^\mathcal{T}\right\|_{L^{\infty}}\left\|g_j^\mathcal{S}\right\|_{L^{1}}.
\end{equation}
We now estimate the second derivatives with one tangential derivative and one normal derivative.  If \(1\leq k\leq n-1\), we have that
\begin{align*}
    \partial_n\partial_k w_i&=-2\delta_{ij}\int_{0}^{t}\int_{\Sigma}\partial_k\partial_n^2 \Gamma(x-y',s)g_j(y',t-s)dy'ds-4\int_{0}^{t}\int_{\Sigma}\partial_k\partial_j\partial_n C_i(x-y',s)g_j(y',t-s)dy'ds\\
    &=J_1+J_2.
\end{align*}
For $J_1$, we note that
\begin{equation*}
    |J_1|\lesssim \int_{0}^{t}\int_{\Sigma}\frac{1}{(|x-y'|^2+s)^{\frac{n+3}{2}}}|g_j(y',t-s)|dy'ds \lesssim t\left\|g_j^\mathcal{T}\right\|_{L^{\infty}}\left\|g_j^\mathcal{S}\right\|_{L^{1}}.
\end{equation*}
For \(J_2\), we divide into the following cases: \(1\leq i\leq n-1\) and \(i=n\).\\
If \(1\leq i \leq n-1\), then
\begin{align*}
    J_2&=-4\int_{0}^{t}\int_{\Sigma}\partial_k\partial_j \left(\partial_i C_n(x-y',s)+\partial_i\partial_n B(x-y',s)\right)g_j(y',t-s)dy'ds\\
    &=-4\int_{0}^{t}\int_{\Sigma}\partial_k\partial_j\partial_i C_n(x-y',s)g_j(y',t-s)dy'ds-4\int_{0}^{t}\int_{\Sigma}\partial_i\partial_n B(x-y',s)\partial_k\partial_j g_j(y',t-s)dy'ds\\
    &=J_{21}^{(1)}+J_{22}^{(1)}.
\end{align*}
Continuing computations, we obtain
\begin{align*}
    |J_{21}^{(1)}|&\lesssim \int_{0}^{t}\int_{\Sigma}\frac{1}{s^{\frac{1}{2}}(|x-y'|^2+s)^{\frac{n+2}{2}}}|g_j(y',t-s)|dy'ds\lesssim \sqrt{t}\left\|g_j^\mathcal{T}\right\|_{L^{\infty}}\left\|g_j^\mathcal{S}\right\|_{L^{1}},\\
    |J_{22}^{(1)}|&\lesssim \left\|g_j^\mathcal{T}\right\|_{L^{\infty}}\left\|\nabla^2 g_j^\mathcal{S}\right\|_{L^{1}},
\end{align*}
where we last estimate can be obtained using the same argument for estimating \(K_3\) in the previous step.\\
If \(i=n\), then
\begin{align*}
    J_2&=-4\int_{0}^{t}\int_{\Sigma}\partial_k\partial_j\partial_n C_n(x-y',s)g_j(y',t-s)dy'ds\\
    &=4\sum_{l=1}^{n-1}\int_{0}^{t}\int_{\Sigma}\partial_k\partial_j\partial_l C_l(x-y',s)g_j(y',t-s)dy'ds+2\int_{0}^{t}\int_{\Sigma}\partial_k\partial_j\partial_n \Gamma(x-y',s)g_j(y',t-s)dy'ds\\
    &=J_{21}^{(2)}+J_{22}^{(2)}.
\end{align*}
We compute that
\begin{align*}
    |J_{21}^{(2)}|&\lesssim \int_{0}^{t}\int_{\Sigma}\frac{1}{s^{\frac{1}{2}}(|x-y'|^2+s)^{\frac{n+2}{2}}}|g_j(y,t-s)|dy'ds\lesssim\sqrt{t}\left\|g_j^\mathcal{T}\right\|_{L^{\infty}}\left\|g_j^\mathcal{S}\right\|_{L^{1}},\\
    |J_{22}^{(2)}|&\lesssim \int_{0}^{t}\int_{\Sigma}\frac{1}{(|x-y'|^2+s)^{\frac{n+3}{2}}}|g_j(y,t-s)|dy'ds\lesssim t\left\|g_j^\mathcal{T}\right\|_{L^{\infty}}\left\|g_j^\mathcal{S}\right\|_{L^{1}}.\\
\end{align*}
Thus we find that
\begin{equation}\label{w_nk}
\left\| \partial_n\partial_k w_i\right\|_{L^{\infty}(Q_2^{+})}\lesssim\left\|g_j^\mathcal{T}\right\|_{L^{\infty}}\left\|g_j^\mathcal{S}\right\|_{W^{2,1}}.
\end{equation}\\
Finally we now estimate the second derivatives with two normal derivatives.\\
If \(i=n\), using \(\nabla\cdot w=0\), we have that \(\partial_n^2 w_n=-\sum_{k=1}^{n-1}\partial_{n}\partial_{k}w_k\) and thus we have the \(L^\infty\)-boundedness of \(\partial_n^2 w_n\) from the previous case.\\
If \(i\neq n\), then
\begin{align*}
    \partial_n^2 w_i&=-2\delta_{ij}\int_{0}^{t}\int_{\Sigma}\partial_n^3 \Gamma(x-y',s)g_j(y',t-s)dy'ds-4\int_{0}^{t}\int_{\Sigma}\partial_n^2\partial_j C_i(x-y',s)g_j(y',t-s)dy'ds\\
    &=K_1+K_2.
\end{align*}
Continuing computations, we get
\begin{align*}
    |K_1|\lesssim \int_{0}^{t}\int_{\Sigma}\frac{1}{(|x-y'|^2+s)^{\frac{n+3}{2}}}|g_j(y',t-s)|dy'ds\lesssim t\left\|g_j^\mathcal{T}\right\|_{L^{\infty}}\left\|g_j^\mathcal{S}\right\|_{L^{1}}.
\end{align*}
For \(K_2\), we have
\begin{align*}
    K_2&=-4\int_{0}^{t}\int_{\Sigma}\partial_n\partial_j(\partial_i C_n(x-y',s)+\partial_i\partial_n B(x-y',s))g_j(y',t-s)dy'ds\\
    &=4\sum_{k=1}^{n-1}\int_{0}^{t}\int_{\Sigma}\partial_i\partial_j\partial_k C_k(x-y',s)g_j(y',t-s)dy'ds+2\int_{0}^{t}\int_{\Sigma}\partial_i\partial_j\partial_n\Gamma(x-y',s)g_j(y',t-s)dy'ds\\
    &-4\int_{0}^{t}\int_{\Sigma}\partial_n^2 B(x-y',s)\partial_i\partial_j g_j(y',t-s)dy'ds\\
    &=K_{21}^{(2)}+K_{22}^{(2)}+K_{23}^{(2)}.
\end{align*}
First two terms are estimated as follows:
\begin{align*}
    |K_{21}^{(2)}|&\lesssim \int_{0}^{t}\int_{\Sigma}\frac{1}{s^{\frac{1}{2}}(|x-y'|^2+s)^{\frac{n+2}{2}}}|g_j(y',t-s)|dy'ds\lesssim \sqrt{t}\left\|g_j^\mathcal{T}\right\|_{L^{\infty}}\left\|g_j^\mathcal{S}\right\|_{L^{1}},\\
    |K_{22}^{(2)}|&\lesssim \int_{0}^{t}\int_{\Sigma}\frac{1}{(|x-y'|^2+s)^{\frac{n+3}{2}}}|g_j(y',t-s)|dy'ds\lesssim t\left\|g_j^\mathcal{T}\right\|_{L^{\infty}}\left\|g_j^\mathcal{S}\right\|_{L^{1}}.
\end{align*}

For \(K_{23}^{(2)}\) we find that Lemma \ref{lemma1} gives
\begin{align*}
    K_{23}^{(2)}&=-4\int_{0}^{t}\int_{\Sigma}\partial_n\partial_i\partial_nB(x-y',s)\partial_jg_j(y',t-s)dy'ds\\
    &=4\int_{0}^{t}\int_{\Sigma}\partial_n^2\Gamma(x_n,s)(\partial_i E(x'-y',0)+J(x'-y',s))\partial_jg_j(y',t-s)dy'ds\\
    &=K_{231}^{(2)}+K_{232}^{(2)}.
\end{align*}
Here using calculations similar to one for the integral appeared in the estimate of first \(I_2\) in the proof of Theorem \ref{thm1} and Lemma \ref{integralestimates}, we find that
\begin{align*}
    \left|K_{232}^{(2)}\right|\lesssim \left(1+\log\left(1+\frac{t}{x_n^2}\right)\right)\left\|g_j^\mathcal{T}\right\|_{L^{\infty}}\left\|g_j^\mathcal{S}\right\|_{L^{1}}.
\end{align*}
It is shown in \cite{KangChang22}, that for any \(p\in(1,\infty)\), if \(g_j^{\mathcal{T}}\) satisfies the hypothesis given in Theorem \ref{thm2},  \(\left\|K_{231}^{(2)}\right\|_{L^{p}(Q_{\frac{1}{2}}^{+})}\) is unbounded. 
Our construction of such an example for \(g_j^\mathcal{T}\) satisfying the hypothesis is rather direct. We let \(g_j^\mathcal{T}\) to be the sum of all characteristic functions of the interval \(((2k+1)^{-a}, (2k)^{-a})\) where \(k\in \mathbb{Z}_+\) and \(a\in(0,1)\) is to be determined. Then since each interval above is pairwise disjoint, \(g_k^\mathcal{T}\in L^{\infty}(\mathbb{R})\) is immediate. 
To show that this function is not in the stated Besov space, we will use its standard integral characterization. Then we can see that the integral corresponding to \(g_k^{\mathcal{T}}\) is divergent for all \(a\in (0,1)\) for \(p \geq 3\) and for some \(a\) for \(p <3\). (see \cite[Theorem 1.1]{KangChang22} for more details).
Since we have \(\partial_n^2 w_i=K_1+K_{21}^{(2)}+K_{22}^{(2)}+K_{231}^{(2)}+K_{231}^{(2)}\), the required blow-up follows.
\qedsymbol


\section{Proof of Theorem \ref{thm3} for the Stokes equations}\label{sec5}
$\bullet$ \textbf{(Step 1) Estimate of \(w_i\).} 
We have that for \(1\leq k\leq n-1\),
\begin{align*}
    K_{ik}(x,t)=-2\delta_{ik}\partial_n \Gamma(x,t)-4\partial_k C_i(x,t).
\end{align*}
Thus, we remind that 
\begin{align*}
    w_i(x,t)&=-2\delta_{ik}\int_{0}^{t}\int_{|y'|\leq 1} \partial_n \Gamma(x-y',s) g_k(y',t-s)dy'ds-4\int_{0}^{t}\int_{|y'|\leq 1}\partial_k C_i(x-y',s)g_k(y',t-s)dy'ds\\
    &=I_1 +I_2.
\end{align*}
Assume first that \(1\leq i \leq n-1\). We first estimate \(I_1\). Since \(|x-y'|\geq \frac{|x|}{2}\) for \(|x|>2\), we find that for \(|x|>2\),
\begin{align*}
    |I_1|&\lesssim \int_{0}^{t}\int_{|y'|\leq 1}\frac{1}{(|x-y'|^2+s)^{\frac{n+1}{2}}}dy'ds\left\|g_k\right\|_{L^{\infty}}\\
    &\lesssim \frac{t}{|x|^{n+1}}\left\|g_k\right\|_{L^{\infty}}.
\end{align*}
For \(|x|<2\), we find that from Lemma \ref{integralestimates},
\begin{align*}
    |I_1|&\lesssim \int_{0}^{t}\int_{|y'|\leq 1}|\partial_n \Gamma(x-y',s)||g_k(y',t-s)|dy'ds\\
    &\lesssim \int_{0}^{t}|\partial_n \Gamma_1(x_n,s)|\int_{|y'|\leq 1}|\Gamma'(x'-y',s)|dy'ds\left\| g_k\right\|_{L^{\infty}}\\
    &\lesssim\int_{0}^{t}\frac{x_n}{s^\frac{3}{2}}e^{-\frac{x_n^2}{4s}}ds\left\| g_k\right\|_{L^{\infty}}\\
    &\lesssim \left\| g_k\right\|_{L^{\infty}}.
\end{align*}
Thus we obtain that
\begin{align*}
    |I_1|\lesssim \frac{1}{\left<x\right>^{n+1}}\left\|g_k\right\|_{L^{\infty}}.
\end{align*}
We now estimate \(I_2\). Since \(|x-y'|\leq 3\) for \(|x|<2\), we find that for \(|x|<2\),
\begin{align*}
    |I_2|&\lesssim\int_{0}^{t}\int_{|y'|\leq 1}\left| C_i(x-y',s)\right|\left|\partial_kg_k(y',t-s)\right|dy'ds\\
    &\lesssim \int_{0}^{t}\int_{|y'|\leq 1}\frac{1}{s^{\frac{1}{2}}(|x-y'|^2+s)^{\frac{n-1}{2}}}dy'ds \left\|\nabla_{x'} g_k\right\|_{L^{\infty}}\\
    &\lesssim \int_{|y'|\leq 1}\frac{1}{|x-y'|^{n-2}}\int_{0}^{\frac{t}{|x-y'|^2}}\frac{1}{u^\frac{1}{2}(u+1)^{\frac{n-1}{2}}} \left\|\nabla_{x'} g_k\right\|_{L^{\infty}}\\
    &\lesssim \int_{|y'|\leq 1}\frac{1}{|x-y'|^{n-2}} \left\|\nabla_{x'} g_k\right\|_{L^{\infty}}\\
    &\lesssim \int_{0}^{3}\frac{r^{n-2}}{(r+x_n)^{n-2}}dr\left\|\nabla_{x'} g_k\right\|_{L^{\infty}} \lesssim\left\|\nabla_{x'} g_k\right\|_{L^{\infty}}.
\end{align*}
For \(|x|>2\), we find that
\begin{align*}
    |I_2|&\lesssim \int_{0}^{t}\int_{|y'|\leq 1}\frac{1}{s^{\frac{1}{2}}(|x-y'|^2+s)^\frac{n}{2}}dy'ds\left\| g_k\right\|_{L^{\infty}}\\
    &\lesssim \int_{0}^{t}\frac{1}{s^{\frac{1}{2}}}\int_{|y'|\leq 1}\frac{1}{(|x|^{2}+s)^{\frac{n}{2}}}dy'ds\left\|g_k\right\|_{L^{\infty}}\lesssim\frac{\sqrt{t}}{|x|^{n}}\left\| g_k\right\|_{L^{\infty}}.
\end{align*}
Thus we obtain that
\begin{align*}
    |I_2|\lesssim \frac{\left\|g_k\right\|_{L^{\infty}}+\left\|\nabla_{x'} g_k\right\|_{L^{\infty}}}{\left<x\right>^{n}}.
\end{align*}
Hence we conclude that
\begin{equation}\label{w_i}
    |w_i(x,t)|\leq |I_1|+|I_2|\lesssim \frac{1}{\left<x\right>^{n}}\left(\left\| g_k\right\|_{L^{\infty}}+\left\|\nabla_{x'} g_k\right\|_{L^{\infty}}\right)
\end{equation}
for \(1\leq i \leq n-1\).\\
Now let \(i=n\). Then \(I_1=0\) and thus we only need to estimate \(I_2\), which enjoys the same estimate as that of \(I_2\) in the previous case.
Hence we find that
\begin{equation}\label{thm3w}
\left\|w\right\|_{L^{\infty}(Q_1^+)}\lesssim \frac{N_1}{\left<x\right>^n}.
\end{equation}
$\bullet$ \textbf{(Step 2) Estimate of \(\partial_j w_i\).}
We have that
\begin{align*}
    \partial_j w_i(x,t)&=-2\delta_{ik}\int_{0}^{t}\int_{|y'|\leq 1}\partial_j\partial_n \Gamma(x-y',s)g_k(y',t-s)dy'ds\\
    &-4\int_{0}^{t}\int_{|y'|\leq 1}\partial_j\partial_k C_i(x-y',s)g_k(y',t-s)dy'ds\\
    &=J_1+J_2.
\end{align*}
1. We first estimate \(J_1\). For \(j\leq n-1\) reminding that 
\begin{align*}
    J_1=-2\delta_{ik}\int_{0}^{t}\int_{|y'|\leq 1}\partial_n \Gamma(x-y',s)\partial_j g_k(y',t-s)dy'ds,
\end{align*}
and  using the same method for estimating \(I_1\), we get 
\begin{align}\label{thm3j1}
    |J_1|\lesssim \frac{t}{\left<x\right>^{n+2}}\left\|\nabla_{x'} g_k\right\|_{L^{\infty}}.
\end{align}
For \(j=n\), we have that for \(|x|<2\),
\begin{align*}
    J_1&=-2\delta_{ik}\int_{0}^{t}\int_{|y'|\leq 1}\partial_n^2 \Gamma(x-y',t)g_k(y',t-s)dy'ds\\
    &\lesssim -2\delta_{ik}\int_{0}^{t}\int_{|y'|\leq 1}\Gamma(x-y',s)\Delta_{y'}g_k(y',t-s)ds+\int_{0}^{t}\Gamma(x-y',s)\partial_t g_k(y',t-s)dy'ds.
\end{align*}
Here the first integral is bounded by \(\left\|\nabla^2 g_k\right\|_{L^{\infty}}\) and the second integral is bounded by \(\textup{LN}\left\|g_k\right\|_{L^{\infty}}\).

For \(|x|>2\), as done for \(I_1\), we find that
\begin{equation*}
    |J_1|\lesssim \frac{t\left\|g_k\right\|_{L^{\infty}}}{|x|^{n+2}}.
\end{equation*}
Thus we obtain
\begin{equation}\label{bigestimate}
    |J_1|\lesssim \frac{1}{\left<x\right>^{n+2}}\left[1+\mathbbm{1}_{|x|<2}\left\{\frac{1}{(x_n^2+|t-1|)^{\frac{1}{2}-a}}\mathbbm{1}_{0<a<\frac{1}{2}} + \left(1+\log \left(2+\frac{1}{x_n^2+|t-1|}\right)\right)\mathbbm{1}_{a=\frac{1}{2}}\right\}\right]\left\|g_k\right\|_{W^{2,\infty}}.
\end{equation}
2. We now estimate \(J_2\). We divide into the following cases:
\begin{enumerate}
    \item \(1\leq i, ~j \leq n-1\),
    \item \(i=n,~1\leq j\leq n-1\),
    \item \(1\leq i \leq n-1, j=n\),
    \item \(i=j=n\).
\end{enumerate}
(1) If \(1\leq i,~j \leq n-1\), then for \(|x|\geq2\),
\begin{align*}
    |J_2|&\lesssim\int_{0}^{t}\int_{|y'|\leq 1}|\partial_j\partial_k C_i(x-y',s)||g_k(y',t-s)|dy'ds\\
    &\lesssim \int_{0}^{t}\int_{|y'|\leq 1}\frac{1}{s^{\frac{1}{2}}(|x-y'|^2+s)^{\frac{n+1}{2}}}dy'ds\left\|g_k\right\|_{L^{\infty}}\lesssim \frac{\sqrt{t}}{|x|^{n+1}}\left\|g_k\right\|_{L^{\infty}},
\end{align*}

and for \(|x|<2\), using the same method for estimating \(I_2\), we get
\begin{align*}
    |J_2|\lesssim \left\|\nabla_{x'} g_k\right\|_{L^{\infty}}.
\end{align*}

Thus we see that for \(1\leq i,\, j \leq n-1\)
\begin{align*}
    |\partial_j w_i(x,t)|\leq |J_1|+|J_2|\lesssim \frac{\left\| g_k\right\|_{W^{1,\infty}}}{\left<x\right>^{n+1}}.
\end{align*}
(2) If \(i=n, ~1\leq j \leq n-1\), since \(C_i\) and \(C_n\) share the same estimates we obtain that by the same method of estimates in the previous case, we get that
\begin{align*}
    |\partial_j w_n(x,t)|\lesssim \frac{\left\| g_k\right\|_{W^{1,\infty}}}{\left<x\right>^{n+1}}.
\end{align*}
\\
(3) If \(1\leq i \leq n-1,~ j=n\), then 
\begin{align*}
    J_2&=-4\int_{0}^{t}\int_{|y'|\leq 1}\partial_n\partial_k C_i(x-y',s)g_k(y',t-s)dy'ds\\
    &=-4\int_{0}^{t}\int_{|y'|\leq 1}\partial_i\partial_k C_n(x-y',s)g_k(y',t-s)dy'ds-4\int_{0}^{t}\int_{|y'|\leq 1}\partial_i\partial_k\partial_n B(x-y',s)g_k(y',t-s)dy'ds\\
    &=J_{21}+J_{22}.
\end{align*}
For \(J_{21}\), using the same method of estimates in the case (1), we get that
\begin{align}\label{thm3j21}
    |J_{21}|\lesssim \frac{1}{\left<x\right>^{n+1}}(\left\| g_k\right\|_{L^{\infty}}+\left\|\nabla_{x'} g_k\right\|_{L^{\infty}}).
\end{align}
For \(J_{22}\), we find that if \(|x|\geq 2\), then using \(\partial_n B(x,t)=-\frac{x_n}{2t}B(x,t)\),
\begin{align*}
    |J_{22}|&\lesssim\int_{0}^{t}\int_{|y'|\leq 1}\partial_i\partial_k\partial_n B(x-y',s)g_k(y',t-s)dy'ds\\
    &=2\int_{0}^{t}\int_{|y'|\leq 1}\frac{x_n}{s}\partial_{i}\partial_j B(x-y',s)g_k(y',t-s)dy'ds.
\end{align*}
Thus we see that by Lemma \ref{integralestimates},
\begin{align*}
   |J_{22}|&\lesssim \int_{0}^{t}\int_{|y'|\leq 1}\frac{x_n}{s}\frac{e^{-\frac{x_n^2}{10s}}}{(|x-y'|^2+s)^{\frac{n}{2}}s^{\frac{1}{2}}}dy'ds\left\|g_k\right\|_{L^{\infty}}\\
    &=\int_{0}^{t}\frac{x_n e^{-\frac{x_n^2}{10s}}}{s^{\frac{3}{2}}}\int_{|y'|\leq 1}\frac{1}{(|x-y'|^2+s)^{\frac{n}{2}}}dy'ds \left\|g_k\right\|_{L^{\infty}}\\
   &\lesssim \frac{1}{|x|^n}\left\|g_k\right\|_{L^{\infty}}.
\end{align*}
If \(|x|<2\), using the same method as above and integration by parts, we find that
\begin{align*}
    |J_{22}|&\lesssim\left\|\nabla_{x'}g_k\right\|_{L^{\infty}}.
\end{align*}
Thus we get
\begin{align}\label{thm3j22}
    |J_{22}|\lesssim \frac{\left\|g_k\right\|_{W^{1,\infty}}}{\left<x\right>^{n}}.
\end{align}
Hence we conclude that
\begin{align}\label{thm31stderivfinal}
    |\partial_j w_n(x,t)|&\leq |J_1|+|J_{21}|+|J_{22}|\lesssim \left(\frac
    {1}{\left<x\right>^{n+1}}+\frac{1}{\left<x\right>^n}\right)\left\| g_k\right\|_{W^{1,\infty}}.
\end{align}
(4) If \(i=j=n\), then \(\nabla\cdot w=0\) gives the same estimate as in the case (1).\\
\\
\textbf{(Step 3) Estimate of \(\partial_l \partial_j w_i\).} \,\,We recall that
\begin{align*}
    \partial_l\partial_j w_i(x,t)&=-2\delta_{ik}\int_{0}^{t}\int_{|y'|\leq 1}\partial_l\partial_j\partial_n \Gamma(x-y',s)g_k(y',t-s)dy'ds\\
    &-4\int_{0}^{t}\int_{|y'|\leq 1} \partial_l\partial_j\partial_k C_i(x-y',s)g_k(y',t-s)dy'ds\\
    &=K_1+K_2.
\end{align*}
1. We first estimate \(K_1\). For \(|x|\geq 2\), we get
\begin{align*}
    |K_1|&\lesssim\int_{0}^{t}\int_{|y'|\leq 1}\frac{1}{(|x-y'|^2+s)^{\frac{n+3}{2}}}dy'ds\left\|g_k\right\|_{L^{\infty}} \lesssim \frac{\left\|g_k\right\|_{L^{\infty}}}{|x|^{n+3}}.
\end{align*}
For \(|x|<2\), if \(l,~j<n\), then
\begin{align*}
    |K_1|\lesssim\int_{0}^{t}\int_{|y'|\leq 1}\left|\partial_n \Gamma(x-y',s)\right|\left|\partial_l\partial_j g_k(y',t-s)\right|dy'ds\lesssim\left\|\nabla_{x'}^2 g_k\right\|_{L^{\infty}},
\end{align*}
where the last estimate follows from the same method for \(I_1\). \\
If \(l=n, ~j\neq n\) or \(l\neq n, ~j=n\), then using integration by part and performing the similar estimates leading \eqref{bigestimate}, we get
\begin{equation}
    |K_1|\lesssim \frac{1}{\left<x\right>^{n+2}}\left[1+\mathbbm{1}_{|x|<2}\left\{\frac{1}{(x_n^2+t-1)^{\frac{1}{2}-a}}\mathbbm{1}_{0<a<\frac{1}{2}} + \left(1+\log \left(2+\frac{1}{x_n^2+|t-1|}\right)\right)\mathbbm{1}_{a=\frac{1}{2}}\right\}\right]\left\|\nabla g_k\right\|_{W^{2,\infty}}.
\end{equation}
If \(l=j=n\), then using integration by parts,
\begin{align*}
    K_1&=-\int_{0}^{t}\int_{|y'|\leq 1}\partial_n\Gamma(x-y',s)\Delta_{y'}g_k(y',t-s)dy'ds+\int_{0}^{t}\int_{|y'|\leq 1}\partial_n \Gamma(x-y',s)\partial_t g_k(y',t-s)dy'ds\\
    &=-\int_{0}^{t}\partial_n\Gamma_1(x_n,s)(1-t+s)_+^{a}\int_{|y'|\leq 1}\Gamma'(x'-y',s)\Delta_{y'}g_k^{\mathcal{S}}(y')dy'ds\\
    &+a\int_{0}^{t}\partial_n \Gamma_1(x_n,s)(1-t+s)_+^{a-1}\int_{|y'|\leq 1}\Gamma'(x'-y',s)g_k^{\mathcal{S}}(y')dy' ds.
\end{align*}
Here the first integral is bounded by \(\left\|\nabla^2 g_k^\mathcal{S}\right\|_{L^{\infty}}\) and the second integral, if \(t<1\), it is bounded by
\begin{equation*}
    \frac{\log\left(1+\frac{16t}{x_n}\right)}{(x_n^2+1-t)^{1-a}}\left\|g_k^S\right\|_{L^{\infty}},
\end{equation*}
where we have used Lemma \ref{alg} and Lemma \ref{integralestimates} to get the above bounds. And if \(t>1\), then it is bounded by
\begin{align*}
    \int_{t-1}^{t}\frac{1}{s}e^{-\frac{x_n^2}{8s}}(1-t+s)^{a-1}ds&\lesssim\int_{t-1}^{t}\frac{1}{(x_n^2+s)(1-t+s)^{1-a}}ds\\
    &=\int_{0}^{1}\frac{du}{(x_n+t-1+u)u^{1-a}}\\
    &\lesssim \frac{1}{(x_n^2+t)^{a}(x_n^2+t-1)^{1-a}},
\end{align*}
where we have used the inequality \(e^{-\theta}\lesssim\frac{1}{1+\theta}\) and Lemma \ref{lemma2}.
Thus 
\begin{align}\label{thm32ndderivk1}
    |K_1|\lesssim \left(1+\frac{\log\left(1+\frac{16t}{x_n}\right)}{(x_n^2+1-t)^{1-a}}\mathbbm{1}_{t<1}+\frac{1}{(x_n^2+t)^a (x_n^2+t-1)^{1-a}}\mathbbm{1}_{t>1}\right)\left\|g_k\right\|_{W^{2,\infty}}.
\end{align}

2. We now estimate \(K_2\). We divide into the following cases:
\begin{enumerate}
    \item \( 1\leq i,~j,~l \leq n-1\),
    \item \(1\leq j,~l \leq n-1,~ i=n\),
    \item \(1\leq i,~j\leq n-1, l=n, \text{ or } 1\leq i,~l\leq n-1,~ j=n\)
    \item \(1\leq j\leq n-1, ~i=l=n, \text{ or } 1\leq l\leq n-1,~ i=j=n\),
    \item \(1\leq i\leq n-1,~j=l=n\),
    \item \(i=j=l=n\).
\end{enumerate}
Note that only the roles of \(j\) and \(l\) are switched for the respective subcases of the case (3) and (4). Thus the subcases of (3) and (4) will give the same estimate and therefore we only focus one of them respectively.

(1) If \(1\leq i,\,j,\,l \leq n-1\), using the same method for estimating \(I_2\) in \textbf{Step 1}, we obtain
\begin{align*}
    |K_2|\lesssim \frac{\left\|g_k\right\|_{W^{2,\infty}}}{\left<x\right>^{n+2}}.
\end{align*}

(2) If \(1\leq j,\,l \leq n-1,\, i=n\), since \(C_i\) and \(C_n\) share the same estimates we obtain that from the previous case,
\begin{align*}
    |K_2|\lesssim \frac{\left\|g_k\right\|_{W^{2,\infty}}}{\left<x\right>^{n+2}}.
\end{align*}

(3) If \(1\leq i,~j\leq n-1, l=n\), we have that
\begin{align*}
    K_2&=-4\int_{0}^{t}\int_{|y'|\leq 1}\partial_n \partial_j\partial_k C_i(x-y',s)g_k(y',t-s)dy'ds\\
    &=-4\int_{0}^{t}\int_{|y'|\leq 1}\partial_j\partial_k \left(\partial_i C_n(x-y',s)+\partial_i\partial_n B(x-y',s)\right)g_k(y',t-s)dy'ds\\
    &=K_{21}+K_{22}.
\end{align*}
Here \(K_{21}\) is estimated similarly as in the case (1), and we get that
\begin{align*}
    |K_{21}|\lesssim \frac{\left\|g_k\right\|_{W^{2,\infty}}}{\left<x\right>^{n+2}}.
\end{align*}
For \(K_{22}\) we find that using \(\partial_n B(x,t)=-\frac{x_n}{2t}B(x,t)\)
\begin{align*}
    K_{22}=2\int_{0}^{t}\int_{|y'|\leq 1}\frac{x_n}{s}\partial_j\partial_k\partial_i B(x-y',s)g_k(y',t-s)dy'ds.
\end{align*}
Then using the same method as estimating \(J_{22}\), we get that
\begin{align*}
    |K_{22}|\lesssim\frac{\left\|g_k\right\|_{W^{1,\infty}}}{\left<x\right>^{n+1}}.
\end{align*}

(4) If \(1\leq j\leq n-1, ~i=l=n\), we have that
\begin{align*}
    K_2&=4\int_{0}^{t}\int_{|y'|\leq 1}\sum_{p=1}^{n-1}\partial_j\partial_k\partial_pC_p(x-y',s)g_k(y',t-s)dy'ds\\
    &+2\int_{0}^{t}\int_{|y'|\leq 1}\partial_j\partial_k\partial_n \Gamma(x-y',s)g_k(y',t-s)dy'ds\\
    &=K_{23}+K_{24}.
\end{align*}
Here \(K_{24}\) shares the same estimate as that of \(K_1\) and we get
\begin{align*}
    |K_{24}|\lesssim \frac{\left\|g_k\right\|_{W^{2,\infty}}}{\left<x\right>^{n+3}}.
\end{align*}
And for \(K_{23}\), we see that it shares the same estimate as that of \(K_2\) in the case (1), and we thus have 
\begin{align*}
    |K_{23}|\lesssim \frac{\left\|g_k\right\|_{W^{2,\infty}}}{\left<x\right>^{n+2}}.
\end{align*}

(5) If \(1\leq i\leq n-1,~j=l=n\), we have that,
\begin{align*}
    K_{2}&=-4\int_{0}^{t}\int_{|y'|\leq 1}\partial_k\partial_n^2 C_i(x-y',s)g_k(y',t-s)dy'ds\\
    &=-4\int_{0}^{t}\int_{|y'|\leq 1}\partial_k\partial_n \left(\partial_i C_n(x-y',s)+\partial_i \partial_nB(x-y',s)\right)dy'ds\\
    &=-4\int_{0}^{t}\int_{|y'|\leq 1}\partial_k\partial_i\left(\sum_{p=1}^{n-1}\partial_p C_p(x-y',s)-\frac{1}{2}\partial_n \Gamma(x-y',s)\right)g_k(y',t-s)dy'ds\\
    &-4\int_{0}^{t}\int_{|y'|\leq 1}\partial_i\partial_k\partial_n^2 B(x-y',s)g_k(y',t-s)dy'ds\\
    &=K_{25}+K_{26}-4\int_{0}^{t}\int_{|y'|\leq 1}\partial_i\partial_k\partial_n^2 B(x-y',s)g_k(y',t-s)dy'ds\\
    &=K_{25}+K_{26}+4\int_{0}^{t}\int_{|y'|\leq 1}\partial_i\partial_k\left(-\sum_{p=1}^{n-1}\partial_p^2 B(x-y',s)+\partial_s B(x-y',s)\right)g_k(y',t-s)dy'ds\\
    &=K_{25}+K_{26}+K_{27}+K_{28}.
\end{align*}
Here \(K_{25}\) enjoys the same estimate as that of \(K_{23}\) and \(K_{26}\) enjoys the same estimate as that of \(K_{24}\). Thus we get that
\begin{align}\label{k_25,26}
    |K_{25}|+|K_{26}|\lesssim\frac{\left\|g_k\right\|_{W^{2,\infty}}}{\left<x\right>^{n+2}}.
\end{align}
Also \(K_{27}\) has the similar estimate as that of \(K_{22}\) and we get
\begin{align}\label{k_27}
    |K_{27}|\lesssim \frac{\left\|g_k\right\|_{W^{2,\infty}}}{\left<x\right>^{n+2}}.
\end{align}
We finally estimate \(K_{28}\). We have that
\begin{align*}
    K_{28}&=4\int_{0}^{t}\int_{|y'|\leq 1} \partial_i\partial_k B(x-y',s)\partial_tg_k(y',t-s)dy'ds.
\end{align*}

First assume that \(|x|\geq 2\) and \(t<1\). 

We then have that using Lemma \ref{alg}, and Lemma \ref{lemma2},
\begin{align*}
    |K_{28}|&\lesssim \int_{0}^{t}\int_{|y'|\leq 1}\frac{e^{-\frac{x_n^2}{10s}}}{s^{\frac{1}{2}}(|x-y'|^2+s)^{\frac{n}{2}}}\left|\partial_t g_k(y',t-s)\right|dy'ds\\
    &\lesssim \int_{0}^{t}\int_{|y'|\leq 1}\frac{e^{-\frac{x_n^2}{10s}}}{s^{\frac{1}{2}}(|x-y'|^2+s)^{\frac{n}{2}}}(1-t+s)^{a-1}dy'ds\left\| g_k^{\mathcal{S}}\right\|_{L^{\infty}}\\
    &\lesssim\frac{1}{|x|^n}\int_{0}^{t}\frac{e^{-\frac{x_n^2}{10s}}}{s^{\frac{1}{2}}(1-t+s)^{1-a}}ds\left\| g_k^{\mathcal{S}}\right\|_{L^{\infty}}\lesssim \frac{1}{|x|^n}\int_{0}^{t}\frac{1}{s^{\frac{1}{2}}(x_n^2+1-t+s)^{1-a}}ds\left\| g_k^{\mathcal{S}}\right\|_{L^{\infty}}\\
    &\lesssim\frac{1}{|x|^n}\frac{1}{x_n+1}\left(\log\left(2+\frac{1}{x_n^2+|t-1|}\right)\mathbbm{1}_{a=\frac{1}{2}}+\frac{1}{(x_n^2+|1-t|)^{\frac{1}{2}-a}}\mathbbm{1}_{a<\frac{1}{2}}\right)\left\| g_k^{\mathcal{S}}\right\|_{L^{\infty}}\\
    &=\frac{\textup{LN}\left\| g_k^{\mathcal{S}}\right\|_{L^{\infty}}}{|x|^n(x_n+1)}\lesssim\frac{\textup{LN}\left\| g_k^{\mathcal{S}}\right\|_{L^{\infty}}}{|x|^n(x_n+1)^{2a}}, \qquad(\because 2a\leq 1).
\end{align*}
Assuming \(|x|\geq 2\), \(t>1\), since \(g_k(y',t-s)\) is supported in \(t-1<s<t\), it follows that
\begin{align*}
    |K_{28}|&\lesssim \int_{t-1}^{t}\int_{|y'|\leq 1}\frac{e^{-\frac{x_n^2}{10s}}}{s^{\frac{1}{2}}(|x-y'|^2+s)^{\frac{n}{2}}}\left|\partial_t g_k(y',t-s)\right|dy'ds\\
    &\lesssim \int_{t-1}^{t}\int_{|y'|\leq 1}\frac{e^{-\frac{x_n^2}{10s}}}{s^{\frac{1}{2}}(|x-y'|^2+s)^{\frac{n}{2}}}(1-t+s)^{a-1}dy'ds\left\| g_k^{\mathcal{S}}\right\|_{L^{\infty}}\\
    &\lesssim\frac{1}{|x|^n}\int_{t-1}^{t}\frac{e^{-\frac{x_n^2}{10s}}}{s^{\frac{1}{2}}(1-t+s)^{1-a}}ds\left\| g_k^{\mathcal{S}}\right\|_{L^{\infty}}\lesssim \frac{1}{|x|^n}\int_{t-1}^{t}\frac{1}{(x_n^2+1-t+s)^{1-a}(x_n^2+s)^\frac{1}{2}}ds\left\| g_k^{\mathcal{S}}\right\|_{L^{\infty}}\\
    &\leq\frac{1}{|x|^n}\int_{0}^{1}\frac{u^{a-1}}{(x_n^2+u+t-1)^{\frac{1}{2}}}du\left\| g_k^{\mathcal{S}}\right\|_{L^{\infty}}\lesssim \frac{\textup{LN}}{|x|^n(x_n+1)^{2a}}\left\| g_k^{\mathcal{S}}\right\|_{L^{\infty}},
\end{align*}
where we have used the estimate \(e^{-\frac{x_n^2}{10s}}\lesssim \left(\frac{s}{x_n^2+s}\right)^\frac{1}{2}\).\\
Assume \(|x|<2\) and \(n\geq 3\). By performing similar calculations as above,
\begin{align*}
    |K_{28}|&\lesssim\int_{0}^{t}\int_{|y'|\leq 1}|B(x-y',s)||\partial_i\partial_k\partial_t g_k(y',t-s)|dy'ds\\
    &\lesssim\int_{(t-1)_{+}}^{t}\int_{|y'|\leq 1}\frac{e^{-\frac{x_n^2}{10s}}}{(|x-y'|^2+s)^{\frac{n-2}{2}}s^{\frac{1}{2}}}(1-t+s)^{a-1}dy'ds\left\|\nabla^2g_k^{\mathcal{S}}\right\|_{L^{\infty}}\\
    &\lesssim\int_{(t-1)_{+}}^{t}\frac{e^{-\frac{x_n^2}{10s}}}{\sqrt{s}}(1-t+s)^{a-1}ds\left\|\nabla^2g_k^{\mathcal{S}}\right\|_{L^{\infty}}\\
    &\lesssim \textup{LN} \left\|\nabla^2g_k^{\mathcal{S}}\right\|_{L^{\infty}}.
\end{align*}
Finally assume \(|x|<2\) and \(n=2\). Then, this is essentially the same integral as \(I_{22}\) in (3.26) of \cite{KangTsai22FE} and the result follows by following the proof of Proposition 3.1 in \cite{KangTsai22FE}.\\

(6) If \(i=j=l=n\) then \(\nabla\cdot w=0\) gives that it reduces to the case (5).\\

\textbf{(Step 4) \(L^p\) estimates.}\,\, 1. The estimate of \(u\) is immediate from its pointwise bound.\\

2. For the estimate of \(\nabla u\), we only need to consider the term \(\mathbbm{1}_{|x|<2}\textup{LN}\) from the pointwise bound of \(\nabla u\) as the other terms can be bounded by the bounds of \(u\), which belong to \(L^p(\mathbb{R}_+^n\times (0,2))\) if and only if \(p\in (1,\infty]\).\\
If \(a<\frac{1}{2}\), then
\begin{align*}
    \int_{0}^2\int_{|x|<2}|\textup{LN}|^p dxdt&\lesssim \int_{0}^1\int_{0}^2 (x_n^2+t)^{p\left(a-\frac{1}{2}\right)}dtdx_n\\
    &=\frac{1}{1+p\left(a-\frac{1}{2}\right)}\int_{0}^2\left((x_n^2+1)^{p\left(a-\frac{1}{2}\right)+1}-x_n^{2p\left(a-\frac{1}{2}\right)+2}\right)dx_n.
\end{align*}
If \(1+p\left(a-\frac{1}{2}\right)\geq 0\), then the integral is obviously finite and if \(1+p\left(a-\frac{1}{2}\right)<0\), then the integral converges if and only if \(2p\left(a-\frac{1}{2}\right)+2>-1\), i.e., \(p<\frac{3}{2a-1}\). Thus the integral is convergent if and only if \(p<\frac{3}{2a-1}\).\\
If \(a=\frac{1}{2}\), then
\begin{align*}
    \int_{0}^2\int_{|x|<2}|\textup{LN}|^p dxdt&\lesssim \int_{0}^2\int_{0}^1 \left|\log\left(2+\frac{1}{x_n^2+t}\right)\right|^p dtdx_n\\
    &\lesssim \int_{0}^2\left|\log\left(2+\frac{1}{x_n^2}\right)\right|^p dx_n,
\end{align*}
which is finite for all \(p\geq 1\).\\
3. For the estimate of \(\nabla^2 u\), we only need to consider the estimates of the followings
\begin{equation}\label{tobeestimated}
    \frac{\textup{LN}}{\left<x\right>^{n}(x_n+1)^{2a}}+\mathbbm{1}_{|x|<2}\frac{\log\left(1+\frac{t}{x_n}\right)}{(x_n^2+|t-1|)^{1-a}},\qquad\textup{for } n\geq 3,
\end{equation}
and 
\begin{equation}\label{tobeestimated2}
    \textup{LN}\log\left(2+\frac{1}{x_2^2+|t-1|}\right),\qquad\textup{for } n=2,
\end{equation}
since the other terms appearing in the pointwise estimate are previously estimated.\\
In this proof, we shall estimate the second term in \eqref{tobeestimated} only, which is locally the most singular. We have that
\begin{align*}
    \int_{0}^2\int_{0}^2\frac{\log^p\left(1+\frac{t}{x_n}\right)}{(x_n^2+|t-1|)^{(1-a)p}}dtdx_n &\leq \int_0^2\int_0^2\frac{\log^p\left(1+\frac{2}{x_n}\right)}{(x_n^2+t)^{(1-a)p}} dtdx_n\\
    &\lesssim \frac{1}{1-(1-a)p}\int_0^2 \log^p\left(1+\frac{2}{x_n}\right)\left((x_n^2+2)^{1-(1-a)p}-x_n^{2-2(1-a)p}\right)dx_n.
\end{align*}
Similarly as the above calculations, we find that the integral above is finite if and only if \(p<\frac{3}{2(1-a)}\).\\

\textbf{(Step 5) Pressure estimate.} \,\,
We next estimate the pressure. We remind that
\begin{align*}
    p(x,t)&=2\partial_k\partial_n \int_{\Sigma}E(x-y')g_k(y',t)dy'-4(\partial_t-\Delta_{x'})\int_{-\infty}^{\infty}\int_{\Sigma}\partial_k A(x-y',t-s)g_k(y',s)dy'ds\\
    &=I_1+I_2.
\end{align*}
For \(I_1\) we note that
\begin{align*}
    I_1=\int_{\Sigma}2\partial_n E(x-y')\partial_kg_k^{\mathcal{S}}(y')dy'g_{k}^{\mathcal{T}}(t)
\end{align*}
and the above integral is the solution of the following Dirichlet problem:

\begin{equation*}
\begin{dcases}
\begin{aligned}
  \Delta u&=0& &\text{in  } \mathbb{R}_{+}^{n},\\
  u(x',0)&= \partial_k g_k^{\mathcal{S}}(x') & &\text{on    } \Sigma, 
\end{aligned}
\end{dcases}
\end{equation*}
and thus we have the estimate
\begin{align*}
    |I_1|\lesssim \left\|\nabla_{x'}g_k\right\|_{L^{\infty}}\mathbbm{1}_{\frac{1}{4}\leq t\leq 1}(t).
\end{align*}
For \(|x|\geq 2\), we find that since \(|x-y'|\leq \frac{|x|}{2}\) for \(|y'|\leq 1\),
\begin{align*}
    |I_1|\lesssim \int_{|y'|\leq 1}\frac{1}{|x-y'|^{n}}dy'\left\|g_k\right\|_{L^{\infty}}\mathbbm{1}_{\frac{1}{4}\leq t\leq 1}(t)\lesssim\mathbbm{1}_{\frac{1}{4}\leq t\leq 1} \frac{1}{|x|^n}\left\|g_k\right\|_{L^{\infty}}
\end{align*}
and thus we have that
\begin{align}\label{pressureI_1}
    |I_1|\lesssim\mathbbm{1}_{\frac{1}{4}\leq t\leq 1}\frac{1}{\left<x\right>^{n}}\left\|\nabla_{x'}g_k\right\|_{L^{\infty}}.
\end{align}

For \(I_2\), we divide into the following cases:
\begin{enumerate}
    \item \( 0<t<\frac{1}{4}\),
    \item \(\frac{1}{4}\leq t<1\),
    \item \(t\geq 1\).
\end{enumerate}

(1) If \(0<t<\frac{1}{4}\) then since \(\text{supp}(g_k^\mathcal{T})\subset\left[\frac{1}{4}, 1\right]\) and \(A(x,t)=0\) for \(t<0\), we find that
\begin{align*}
    \int_{-\infty}^{\infty}\int_{\Sigma}\partial_k A(x-y',t-s)g_k(y',s)dy'ds
    =\int_{\frac{1}{4}}^{\infty}\int_{\Sigma}\partial_k A(x-y',t-s)g_k^{\mathcal{S}}(y')g_k^{\mathcal{T}}(s)dy'ds=0,
\end{align*}
and thus \(I_2=0\).\\

(2) If \(\frac{1}{4}\leq t <1\), then
\begin{align}\label{t<1start}
   \nonumber |I_2|&=\left|\int_{-\infty}^{\infty}\int_{\Sigma}\partial_k A(x-y',t-s)\left(g_k^{\mathcal{S}}(y')\partial_s g_k^{\mathcal{T}}(s)-\Delta_{y'}g_k^{\mathcal{S}}(y')g_k^{\mathcal{T}}(s)\right)dy'ds\right|\\
   \nonumber &\leq \int_{-\infty}^{\infty}\int_{\Sigma}\left|\partial_k A(x-y',t-s)\right|\left(\left|g_k^{\mathcal{S}}(y')\right|\left|\partial_s g_k^{\mathcal{T}}(s)\right|+\left|\Delta_{y'}g_k^{\mathcal{S}}(y')\right|\left|g_k^{\mathcal{T}}(s)\right|\right)dy'ds\\
   \nonumber &\lesssim \int_{\frac{1}{4}}^{t}\int_{|y'|\leq 1}\frac{1}{(t-s)^{\frac{1}{2}}(|x-y'|^2+(t-s))^{\frac{n-1}{2}}}dy'ds\times \left\|g_k^{\mathcal{T}}\right\|_{L^{\infty}} \left\|\nabla_{x'}^2g_k^{\mathcal{S}}\right\|_{L^{\infty}}\\
    &+\int_{\frac{1}{4}}^{t}\int_{|y'|\leq 1}\frac{1}{(t-s)^{\frac{1}{2}}(|x-y'|^2+(t-s))^{\frac{n-1}{2}}}\left|\partial_s g_k^{\mathcal{T}}(s)\right|dy'ds\times\left\|g_k^{\mathcal{S}}\right\|_{L^{\infty}}.
\end{align}

We first claim the following estimate, which will be used throughout this proof.
\begin{equation}\label{claimineq}
    \int_{|y'|\leq 1}\frac{1}{(|x-y'|^2+(t-s))^{\frac{n-1}{2}}}dy'\lesssim \mathbbm{1}_{|x|<2}\frac{\log\left(2+\frac{1}{\sqrt{t-s}}\right)}{(|x|+\sqrt{t-s}+1)^{n-1}}+\mathbbm{1}_{|x|>2}\frac{1}{|x|^{n-1}}.
\end{equation}
Indeed, if \(|x|<2\), then using that \(|y'|\leq 1\) implies \(|x-y'|\leq |x|+1\) and Lemma \ref{lemma2}, we see that
\begin{align*}
    \int_{|y'|\leq 1}\frac{1}{(|x-y'|^2+(t-s))^{\frac{n-1}{2}}}dy'&\lesssim\int_{|x-y'|\leq |x|+1}\frac{1}{(|x-y'|^2+(t-s))^{\frac{n-1}{2}}}dy'\\
    &\lesssim\int_{0}^{|x|+1}\frac{r^{n-2}}{(r+\sqrt{t-s})^{n-1}}dr\\
    &\lesssim \frac{(|x|+1)^{n-1}}{(|x|+\sqrt{t-s}+1)^{n-1}}\left(1+\log_{+}\frac{|x|+1}{\sqrt{t-s}}\right)\\
    &\lesssim\frac{\log\left(2+\frac{1}{\sqrt{t-s}}\right)}{(|x|+\sqrt{t-s}+1)^{n-1}},
\end{align*}
where in the last inequality we used the simple inequality \(1+\log_{+}(3a)\leq 4\log(a+2)\) for \(a>0\).

If \(|x|\geq 2\), then using \(|x-y'|\geq \frac{|x|}{2}\), we find that
\begin{align*}
    \int_{|y'|\leq 1}\frac{1}{(|x-y'|^2+(t-s))^{\frac{n-1}{2}}}dy'&\lesssim \frac{1}{|x|^{n-1}}.
\end{align*}
This proves the claim \eqref{claimineq}.\\

i) If \(t\leq \frac{1}{2}\), then the term \(|\partial_s g_k^{\mathcal{T}}(s)|\) is bounded and thus using (\ref{claimineq}), we have that for any \(\epsilon>0\), both integrals of the RHS of \eqref{t<1start} are bounded by
\begin{align*}
    \int_{\frac{1}{4}}^{t}\int_{|y'|\leq 1}\frac{1}{(t-s)^{\frac{1}{2}}(|x-y'|^2+(t-s))^{\frac{n-1}{2}}}&dy'ds
    \lesssim\mathbbm{1}_{|x|<2}\int_{\frac{1}{4}}^t\frac{1}{(t-s)^{\frac{1}{2}}}\frac{\log\left(2+\frac{1}{\sqrt{t-s}}\right)}{(|x|+1+\sqrt{t-s})^{n-1}}ds\\
    &+\mathbbm{1}_{|x|>2}\int_{\frac{1}{4}}^{t}\frac{1}{(t-s)^{\frac{1}{2}}}\frac{1}{|x|^{n-1}}ds\\
    &\lesssim\mathbbm{1}_{|x|<2}\int_{0}^{\sqrt{t-\frac{1}{4}}}\frac{\log\left(2+\frac{1}{u}\right)}{(|x|+1+u)^{n-1}}du+\mathbbm{1}_{|x|>2}\frac{1}{|x|^{n-1}}\\
    &\lesssim \mathbbm{1}_{|x|<2}\int_{0}^{\sqrt{t-\frac{1}{4}}}\log\left(2+\frac{1}{u}\right)du+\mathbbm{1}_{|x|>2}\frac{1}{|x|^{n-1}}\\
    &\lesssim \mathbbm{1}_{|x|<2}\left(t-\frac{1}{4}\right)^{\frac{1-\epsilon}{2}}+\mathbbm{1}_{|x|>2}\frac{1}{|x|^{n-1}}.
\end{align*}
Thus we have that for \(t\leq \frac{1}{2}\),
\begin{equation}\label{pressurei21}
    |I_2|\lesssim\mathbbm{1}_{|x|<2}+\mathbbm{1}_{|x|>2}\frac{1}{|x|^{n-1}}.
\end{equation}

ii) We assume \(t\geq \frac{1}{2}\). Then, we have that the first integral in the RHS of \eqref{t<1start} is bounded by
\begin{align*}
    \mathbbm{1}_{|x|<2}\int_{\frac{1}{4}}^{t}\frac{\log\left(2+\frac{1}{\sqrt{t-s}}\right)}{(t-s)^{\frac{1}{2}}(|x|+1+\sqrt{t-s})^{n-1}}ds&+\mathbbm{1}_{|x|>2}\frac{1}{|x|^{n-1}}\int_{\frac{1}{4}}^t \frac{1}{\sqrt{t-s}}ds\\
    &\lesssim\mathbbm{1}_{|x|<2} \int_{0}^{\sqrt{t-\frac{1}{4}}}\log\left(2+\frac{1}{u}\right)du+\mathbbm{1}_{|x|>2}\frac{1}{|x|^{n-1}}\\
    &\lesssim \mathbbm{1}_{|x|<2} +\mathbbm{1}_{|x|>2}\frac{1}{|x|^{n-1}}.
\end{align*}
Also the second integral in the RHS of \eqref{t<1start} is bounded by
\begin{align}\label{interm}
    \nonumber \int_{\frac{1}{4}}^{t}\int_{|y'|\leq 1}&\frac{1}{(t-s)^{\frac{1}{2}}(|x-y'|^2+(t-s))^{\frac{n-1}{2}}} \left|\partial_s g_k^{\mathcal{T}}(s)\right|dy'ds\\
    \nonumber &=\int_{\frac{1}{4}}^{\frac{1}{2}}\int_{|y'|\leq 1}\frac{1}{(t-s)^{\frac{1}{2}}(|x-y'|^2+(t-s))^{\frac{n-1}{2}}} \left|\partial_s g_k^{\mathcal{T}}(s)\right|dy'ds\\
    \nonumber\nonumber &+\int_{\frac{1}{2}}^{t}\int_{|y'|\leq 1}\frac{1}{(t-s)^{\frac{1}{2}}(|x-y'|^2+(t-s))^{\frac{n-1}{2}}} \left|\partial_s g_k^{\mathcal{T}}(s)\right|dy'ds\\
    \nonumber &\lesssim \mathbbm{1}_{|x|<2}\frac{1}{(1-t)^{\frac{1+\epsilon}{2}-a}}+\mathbbm{1}_{|x|>2}\frac{1}{|x|^{n-1}}\\
     &+\mathbbm{1}_{|x|<2}\int_{\frac{1}{2}}^{t}\frac{1}{(t-s)^{\frac{1}{2}}}\frac{\log\left(2+\frac{1}{\sqrt{t-s}}\right)}{(|x|+1+\sqrt{t-s})^{n-1}}\frac{1}{(1-s)^{1-a}} ds+\mathbbm{1}_{|x|>2}\int_{\frac{1}{2}}^{t}\frac{1}{(t-s)^{\frac{1}{2}}(1-s)^{1-a}}ds\frac{1}{|x|^{n-1}},
\end{align}
where the last inequality follows since the integral with the range \(s\in \left[\frac{1}{4},\frac{1}{2}\right]\) is bounded by the integral estimated in the case \(t\leq \frac{1}{2}\).
To estimate the first integral on the RHS of \eqref{interm} (thus we may assume that \(|x|<2\) here), we use the estimate \(\log(1+x)\leq C_\epsilon x^\epsilon\) for any \(0<\epsilon<1\) to get
\begin{align*}
    \int_{\frac{1}{2}}^{t}\frac{1}{(t-s)^{\frac{1}{2}}}&\frac{\log\left(2+\frac{1}{\sqrt{t-s}}\right)}{(|x|+1+\sqrt{t-s})^{n-1}} \frac{1}{(1-s)^{1-a}} ds\\
    &= \int_{0}^{\sqrt{t-\frac{1}{2}}}\frac{\log\left(2+\frac{1}{u}\right)}{(|x|+1+u)^{n-1}}\frac{1}{(u^2+1-t)^{1-a}}du\\
    &\lesssim \int_{0}^{\sqrt{t-\frac{1}{2}}}\frac{\left(1+\frac{1}{u}\right)^\epsilon}{(u+\sqrt{1-t})^{2-2a}} du\lesssim\frac{1}{(1-t)^{\frac{1+\epsilon}{2}-a}}
\end{align*}
and the second integral on the RHS of \eqref{interm} is bounded by \(\textup{LN}^*\).

Thus we conclude that
\begin{align}\label{pressurei22}
    |I_2|\lesssim \mathbbm{1}_{|x|<2}\left(1+\frac{1}{(1-t)^{\frac{1+\epsilon}{2}-a}}\right) +\mathbbm{1}_{|x|>2}\left(\frac{1}{|x|^{n-1}}+\frac{\textup{LN}^*}{|x|^{n-1}}\right).
\end{align}
~3) Finally if \(t\geq 1\), then as in the previous case,
\begin{align}\label{pressurei2t>1}
    |I_2|&\lesssim \int_{\frac{1}{4}}^{1}\int_{|y'|\leq 1}\frac{1}{(t-s)^{\frac{1}{2}}(|x-y'|^2+(t-s))^{\frac{n-1}{2}}}dy'ds\times \left\|g_k^{\mathcal{T}}\right\|_{L^{\infty}} \left\|\nabla_{x'}^2g_k^{\mathcal{S}}\right\|_{L^{\infty}}\\
    \nonumber&+\int_{\frac{1}{4}}^{1}\int_{|y'|\leq 1}\frac{1}{(t-s)^{\frac{1}{2}}(|x-y'|^2+(t-s))^{\frac{n-1}{2}}}\left|\partial_s g_k^{\mathcal{T}}(s)\right|dy'ds\times\left\|g_k^{\mathcal{S}}\right\|_{L^{\infty}}.
\end{align}
Here the first integral in the RHS of (\ref{pressurei2t>1}) can be estimated as 
\begin{align*}
    \int_{\frac{1}{4}}^{1}\int_{|y'|\leq 1}&\frac{1}{(t-s)^{\frac{1}{2}}(|x-y'|^2+(t-s))^{\frac{n-1}{2}}}dy'ds\lesssim\int_{\frac{1}{4}}^1\frac{1}{(t-s)^{\frac{1}{2}}}\frac{\log\left(2+\frac{1}{\sqrt{t-s}}\right)}{(|x|+1+\sqrt{t-s})^{n-1}}ds\\
    &\lesssim \int_{\sqrt{t-1}}^{\sqrt{t-\frac{1}{4}}}\frac{\log\left(2+\frac{1}{u}\right)}{(|x|+u+1)^{n-1}}du\lesssim \frac{1}{\left<x\right>^{n-1}}\int_{\sqrt{t-1}}^{\sqrt{t-\frac{1}{4}}}\log\left(2+\frac{1}{u}\right)du\\
    &\lesssim \frac{1}{\left<x\right>^{n-1}}\int_{\sqrt{t-1}}^{\sqrt{t-\frac{1}{4}}}\left(1+\frac{1}{u}\right)^{\epsilon}du\lesssim \frac{1}{\left<x\right>^{n-1}}\int_{\sqrt{t-1}}^{\sqrt{t-\frac{1}{4}}}\frac{1}{u^{\epsilon}}du\\
    &\lesssim \frac{\left(t-\frac{1}{4}\right)^{\frac{1-\epsilon}{2}}}{\left<x\right>^{n-1}}.
\end{align*}
For the second integral in the RHS of (\ref{pressurei2t>1}), we have that
\begin{equation}\label{secondof5.14}
\begin{split}
        \int_{\frac{1}{4}}^{1}\int_{|y'|\leq 1}&\frac{1}{(t-s)^{\frac{1}{2}}(|x-y'|^2+(t-s))^{\frac{n-1}{2}}}\left|\partial_s g_k^{\mathcal{T}}(s)\right|dy'ds\\
    &\lesssim\int_{\frac{1}{4}}^{\frac{1}{2}}\int_{|y'|\leq 1}\frac{1}{(t-s)^{\frac{1}{2}}(|x-y'|^2+(t-s))^{\frac{n-1}{2}}}\left|\partial_s g_k^{\mathcal{T}}(s)\right|dy'ds\\
    &+\int_{\frac{1}{2}}^{1}\int_{|y'|\leq 1}\frac{1}{(t-s)^{\frac{1}{2}}(|x-y'|^2+(t-s))^{\frac{n-1}{2}}}\left|\partial_s g_k^{\mathcal{T}}(s)\right|dy'ds.
\end{split}
\end{equation}
The first integral of \eqref{secondof5.14} can be estimated as follows: for any \(\epsilon>0\),
\begin{align*}
    \int_{\frac{1}{4}}^{\frac{1}{2}}\int_{|y'|\leq 1}&\frac{1}{(t-s)^{\frac{1}{2}}(|x-y'|^2+(t-s))^{\frac{n-1}{2}}}\left|\partial_s g_k^{\mathcal{T}}(s)\right|dy'ds\\
    &\lesssim \int_{\frac{1}{4}}^{\frac{1}{2}}\frac{1}{(t-s)^{\frac{1}{2}}}\frac{\log\left(2+\frac{1}{\sqrt{t-s}}\right)}{(|x|+1+\sqrt{t-s})^{n-1}}ds\\
    &\lesssim \frac{1}{\left<x\right>^{n-1}}\int_{\sqrt{t-\frac{1}{2}}}^{\sqrt{t-\frac{1}{4}}}\log \left(2+\frac{1}{u}\right)du\lesssim  \frac{1}{\left<x\right>^{n-1}}\int_{\sqrt{t-\frac{1}{2}}}^{\sqrt{t-\frac{1}{4}}}\frac{1}{v^{\epsilon}}dv\\
    &\lesssim \frac{(t-\frac{1}{2})^{\frac
    {1-\epsilon}{2}}}{\left<x\right>^{n-1}}.
\end{align*}\\
The second integral of \eqref{secondof5.14} can be estimated as follows:
\begin{align*}
    \int_{\frac{1}{2}}^{1}\int_{|y'|\leq 1}&\frac{1}{(t-s)^{\frac{1}{2}}(|x-y'|^2+(t-s))^{\frac{n-1}{2}}}\left|\partial_sg_k^{\mathcal{T}}(s)\right|dy'ds\\
    &\lesssim\int_{\frac{1}{2}}^{1}\int_{|y'|\leq 1}\frac{1}{(t-s)^{\frac{1}{2}}(|x-y'|^2+(t-s))^{\frac{n-1}{2}}(1-s)^{1-a}}dy'ds\\
    &\lesssim\int_{\frac{1}{2}}^{1}\frac{1}{(t-s)^{\frac{1}{2}}}\frac{\log\left(2+\frac{1}{\sqrt{t-s}}\right)}{(|x|+\sqrt{t-s}+1)^{n-1}}\frac{1}{(1-s)^{1-a}}ds\mathbbm{1}_{|x|<2}+\int_{\frac{1}{2}}^{1}\frac{1}{(t-s)^{\frac{1}{2}}(1-s)^{1-a}}ds\frac{1}{|x|^{n-1}}\mathbbm{1}_{|x|>2}\\
    &\lesssim\int_{\sqrt{t-1}}^{\sqrt{t-\frac{1}{2}}}\frac{\log\left(2+\frac{1}{u}\right)}{(1-t+u^2)^{1-a}}du\mathbbm{1}_{|x|<2}+\frac{\textup{LN}^*}{|x|^{n-1}}\mathbbm{1}_{|x|>2}\\
    &\lesssim\int_{\sqrt{t-1}}^{\sqrt{t-\frac{1}{2}}}\frac{(1+\frac{1}{u})^{\epsilon}}{(1-t+u^2)^{1-a}}du\mathbbm{1}_{|x|<2}+\frac{\textup{LN}^*}{|x|^{n-1}}\mathbbm{1}_{|x|>2}\\
    &\lesssim\int_{\sqrt{t-1}}^{\sqrt{t-\frac{1}{2}}}\frac{1}{u^{\epsilon}(1-t+u^2)^{1-a}}du\mathbbm{1}_{|x|<2}+\frac{\textup{LN}^*}{|x|^{n-1}}\mathbbm{1}_{|x|>2}\\
    &\lesssim\frac{1}{(t-1)^{\frac{\epsilon}{2}}}\int_{0}^{\frac
    {1}{2}}\frac{dv}{v^{1-a}(t-1+v)^{\frac{1}{2}}}\mathbbm{1}_{|x|<2}+\frac{\textup{LN}^*}{|x|^{n-1}}\mathbbm{1}_{|x|>2}\\
    &\lesssim\frac{\textup{LN}^*}{(t-1)^{\frac{\epsilon}{2}}}\mathbbm{1}_{|x|<2}+\frac{\textup{LN}^*}{|x|^{n-1}}\mathbbm{1}_{|x|>2}.
\end{align*}

Thus we conclude that
\begin{equation}\label{pressurei23}
    |I_2|\lesssim \frac{\left(t-\frac{1}{4}\right)^{\frac{1-\epsilon}{2}}}{\left<x\right>^{n-1}}+\mathbbm{1}_{|x|<2}\frac{\textup{LN}^*}{(t-1)^{\frac{\epsilon}{2}}}+\mathbbm{1}_{|x|>2}\frac{\textup{LN}^*}{|x|^{n-1}}.
\end{equation}

Hence \eqref{pressurei21}, \eqref{pressurei22} and \eqref{pressurei23} give
\begin{align*}
    |I_2|&\lesssim \mathbbm{1}_{|x|<2}\left(\mathbbm{1}_{t\leq\frac{1}{2}}\left(t-\frac{1}{4}\right)^{\frac{1-\epsilon}{2}}+\mathbbm{1}_{\frac{1}{2}\leq t<1}\left(1+\frac{1}{(1-t)^{\frac{1+\epsilon}{2}-a}}\right)+\mathbbm{1}_{t> 1}\frac{\textup{LN}^*}{(t-1)^{\frac{\epsilon}{2}}}\right)\\
    &+\mathbbm{1}_{|x|>2}\left(\mathbbm{1}_{t\leq \frac{1}{2}}\frac{1}{|x|^{n-1}}+\left(\mathbbm{1}_{\frac{1}{2}\leq t<1}+\mathbbm{1}_{t>1}\right)\left(\frac{1}{|x|^{n-1}}+\frac{\textup{LN}^*}{|x|^{n-1}}\right)\right).
 \end{align*}
For the estimates of the higher spatial derivatives, the similar argument as above gives the result,
and this proves the pressure estimate.
\qedsymbol


\section{Proof of Theorem \ref{thm4}}\label{sec6}
From the proof of the previous theorem, we find that
\begin{align*}
    \left|\partial_l\partial_j w_i(x,t)-\sigma K_{28}\right|\lesssim \frac{CN_3}{\left<x\right>^{n}} 
\end{align*}
for \(x\in\mathbb{R}_{+}^{n}\) and \(t\in[0, 2]\) and thus assuming \(\sigma=1\), we find that
\begin{equation}\label{w_iupper}
    \left|\partial_n^2 w_i(x,t)\right|\geq \left|K_{28}\right|-\left|\partial_n^2 w_i(x,t)-K_{28}\right|\gtrsim |K_{28}|-\frac{CN_3}{\left<x\right>^n}.
\end{equation}
We now show that \(\left|K_{28}(x,1)\right|\) is unbounded as \(x_n\rightarrow 0\). Note that using \(4B(x,s)=e^{-\frac{x_n^2}{4s}}\frac{C_n}{s^{\frac{n}{2}}}K(x',s)\) where \(C_n=4(4\pi)^{-\frac{n}{2}}\frac{1}{n(n-2)|B_1|}\),

\begin{align*}
    K_{28}(x,1)&=4\int_{0}^{1}\int_{|y'|\leq 1}\partial_i\partial_kB(x-y',s)\partial_s g_k(y',1-s)dy'ds\\
    &=4\int_{0}^{1}\int_{|y'|\leq 1}B(x-y',s)\partial_i\partial_k\partial_s g_k(y',1-s)dy'ds \\
    &=4\int_{0}^{1}\int_{|y'|\leq 1}B(x-y',s)\partial_i\partial_kg_k^{\mathcal{S}}(y')\partial_s g_k^{\mathcal{T}}(1-s)dy'ds\\
    &=\int_{0}^{1}\int_{|y'|\leq 1}e^{-\frac{x_n^2}{4s}}\frac{C_n}{s^{\frac{n}{2}}}K(x'-y',s)\partial_i\partial_k g_k^{\mathcal{S}}(y')\partial_s g_k^{\mathcal{T}}(1-s)dy'ds\\
    &=-a\int_{0}^{\frac{1}{2}}\int_{|y'|\leq 1}e^{-\frac{x_n^2}{4s}}\frac{C_n}{s^{\frac{n}{2}}}K(x'-y',s)\partial_i\partial_k g_k^{\mathcal{S}}(y')s^{a-1}dy'ds\\
    &+\int_{\frac{1}{2}}^{\frac{3}{4}}\int_{|y'|\leq 1}e^{-\frac{x_n^2}{4s}}\frac{C_n}{s^{\frac{n}{2}}}K(x'-y',s)\partial_i\partial_k g_k^{\mathcal{S}}(y')\partial_s g_k^{\mathcal{T}}(1-s)dy'ds\\
    &=K_{281}(x,1)+K_{282}(x,1).
\end{align*}\\

1) First assume that \(i\neq k\). 
Note that \(K_{28}(x, 1)\) is odd in \(x_i\) and \(x_k\) and thus we may assume that \(x_i<0\) and \(x_k>0\). The other cases then follow with a sign change.
We also note that since \(\partial_s g_k^{\mathcal{T}}(1-s)\) is bounded on \([\frac{1}{2}, \frac{3}{4}]\) and using Lemma \ref{Kestimate},
\begin{align*}
    \left|K_{282}(x,1)\right|&\lesssim \frac{1}{|x'|^{n-2}}.
\end{align*}

We now estimate \(K_{281}(x,1)\). We have, with \(d:=|x'-y'|\), that
\begin{align*}
    K_{281}(x,1)&=-a\int_{0}^{\frac{1}{2}}\int_{|y'|\leq 1}e^{-\frac{x_n^2}{4s}}\frac{C_n}{s^{\frac{n}{2}-a+1}}K(x'-y',s)\left((\partial_i\partial_k g_k^{\mathcal{S}})_{+}(y')-(\partial_i\partial_k g_k^{\mathcal{S}})_{-}(y')\right)dy'ds\\
    &\geq -a\int_{0}^{\frac{1}{2}}\int_{|y'|\leq 1}e^{-\frac{x_n^2}{4s}}\frac{C_n}{s^{\frac{n}{2}-a+1}}\left(\left(\frac{m}{(m-1)d}\right)^{n-2}s^{\frac{n-1}{2}}(4\pi)^{\frac{n-1}{2}}+\frac{C}{d^{n-2}}s^{\frac{n-1}{2}}e^{-\frac{d^2}{8m^2s}}\right)(\partial_i\partial_k g_k^{\mathcal{S}})_{+}(y')dy'ds\\
    &+a\int_{0}^{\frac{1}{2}}\int_{|y'|\leq 1}e^{-\frac{x_n^2}{4s}}\frac{C_n}{s^{\frac{n}{2}-a+1}}\left(\frac{m}{(m+1)d}\right)^{n-2}s^{\frac{n-1}{2}}(4\pi)^{\frac{n-1}{2}}\left(1-2^{\frac{n-1}{4}}e^{-\frac{d^2}{8m^2s}}\right)(\partial_i\partial_k g_k^{\mathcal{S}})_{-}(y')dy'ds\\
    &\geq a\int_{0}^{\frac{1}{2}}\int_{|y'|\leq 1}e^{-\frac{x_n^2}{4s}}\frac{C_n}{s^{-a+\frac{3}{2}}}(4\pi)^{\frac{n-1}{2}}\left\{\left(\frac{m}{(m+1)d}\right)^{n-2}(\partial_i\partial_k g_k^{\mathcal{S}})_{-}(y')-\left(\frac{m}{(m-1)d}\right)^{n-2}(\partial_i\partial_k g_k^{\mathcal{S}})_{+}(y')\right\}dy'ds\\
    &-a\int_{0}^{\frac{1}{2}}\int_{|y'|\leq 1}e^{-\frac{x_n^2}{4s}}\frac{C}{s^{-a+\frac{3}{2}}}\frac{1}{d^{n-2}}e^{-\frac{d^2}{8m^2s}}\left|\partial_i\partial_kg_k^{\mathcal{S}}(y')\right|dy'ds.
\end{align*}
Let \(J\) be the second integral of the RHS above, then
\begin{align*}
    \frac{K_{281}(x,1)+J}{aC_n (4\pi)^{\frac{n-1}{2}}}
    &\geq \int_{0}^{\frac{1}{2}}e^{-\frac{x_n^2}{4s}}s^{a-\frac{3}{2}}ds\\
    &\times\int_{|y'|\leq 1}\left\{\left(\frac{m}{(m+1)d}\right)^{n-2}(\partial_i\partial_k g_k^{\mathcal{S}})_{-}(y')-\left(\frac{m}{(m-1)d}\right)^{n-2}(\partial_i\partial_k g_k^{\mathcal{S}})_{+}(y')\right\}dy'.
\end{align*}

Here since \(x_n\leq 1\), if \(a<\frac{1}{2}\),
\begin{align*}
    \int_{0}^{\frac{1}{2}}e^{-\frac{x_n^2}{4s}}s^{a-\frac{3}{2}}ds
    &=\left(\frac{x_n^2}{4}\right)^{a-\frac{1}{2}}\int_{\frac{x_n^2}{2}}^{\infty}e^{-u}u^{-a-\frac{1}{2}}du\\
    &\geq\left(\frac{x_n^2}{4}\right)^{a-\frac{1}{2}}\int_{\frac{1}{2}}^{\infty}e^{-u}u^{-a-\frac{1}{2}}du\sim \left(\frac{x_n^2}{4}\right)^{a-\frac{1}{2}},
\end{align*}
and if \(a=\frac{1}{2}\),
\begin{align*}
    \int_{0}^{\frac{1}{2}}e^{-\frac{x_n^2}{4s}}s^{a-\frac{3}{2}}ds&=\int_{\frac{x_n^2}{2}}^{\infty}e^{-u}u^{-1}du\gtrsim \log \frac{2}{x_n^2}\geq \log\frac{2}{x_n}.
\end{align*}\\
We now estimate the space integral which we denote as \(I\). From our construction of the boundary data, we find that \((\partial_i\partial_k g_k^{\mathcal{S}})_{-}(y')=0\) on \(\left\{y_i y_k\geq 0\right\}\) and \((\partial_i\partial_k g_k^{\mathcal{S}})_{+}(y')=0\) on \(\left\{y_i y_k\leq 0\right\}\) and thus
\begin{align*}
    I&=\int_{\substack{|y'|\leq 1 \\ y_iy_k<0}}\left(\frac{m}{(m+1)d}\right)^{n-2}(\partial_i\partial_k g_k^{\mathcal{S}})_{-}(y')dy'
    -\int_{\substack{|y'|\leq 1 \\ y_iy_k>0}}\left(\frac{m}{(m-1)d}\right)^{n-2}(\partial_i\partial_k g_k^{\mathcal{S}})_{+}(y')dy'\\
    &=\int_{\substack{|y'|\leq 1 \\ y_i>0 ,~ y_k<0}}\left(\frac{m}{(m+1)d}\right)^{n-2}(\partial_i\partial_k g_k^{\mathcal{S}})_{-}(y')dy'
    +\int_{\substack{|y'|\leq 1 \\ y_i<0,~ y_k>0}}\left(\frac{m}{(m+1)d}\right)^{n-2}(\partial_i\partial_k g_k^{\mathcal{S}})_{-}(y')dy'
    \\&-\int_{\substack{|y'|\leq 1 \\ y_i>0,~ y_k>0}}\left(\frac{m}{(m-1)d}\right)^{n-2}(\partial_i\partial_k g_k^{\mathcal{S}})_{+}(y')dy'
    -\int_{\substack{|y'|\leq 1 \\ y_i<0,~ y_k<0}}\left(\frac{m}{(m-1)d}\right)^{n-2}(\partial_i\partial_k g_k^{\mathcal{S}})_{+}(y')dy'\\
    &=:I_1+I_2-I_3-I_4.
\end{align*}

Now we denote \(y_{(j)}^{*}:=(y_1,\cdots, -y_j,\cdots, y_{n-1})\) and \(d_{j}:=|x'-y_{(j)}^{*}|\) for \(j=1,2,\cdots, n-1\) and \(y_{(i,k)}^{*}:=(y_1,\cdots, y_{i-1}, -y_i, y_{i+1}\cdots,y_{k-1}, -y_k, y_{k+1},\cdots, y_{n-1})\) with \(d_{ik}:=|x'-y_{(i,k)}^{*}|\). Then we find that
\begin{align*}
    I_1&=\int_{\substack{|y'|\leq 1 \\ y_i>0 ,~ y_k<0}}\left(\frac{m}{(m+1)d}\right)^{n-2}(\partial_i\partial_k g_k^{\mathcal{S}})_{-}(y')dy'\\
    &=\int_{\substack{|y'|\leq 1 \\ y_i>0 ,~ y_k>0}}\left(\frac{m}{(m+1)d_k}\right)^{n-2}(\partial_i\partial_k g_k^{\mathcal{S}})_{-}(y_{(k)}^{*})dy'\\
    &=\int_{\substack{|y'|\leq 1 \\ y_i>0 ,~ y_k>0}}\left(\frac{m}{(m+1)d_k}\right)^{n-2}(\partial_i\partial_k g_k^{\mathcal{S}})_{+}(y')dy',
\end{align*}
where the last equality follows since \(\partial_i\partial_k g_k^{\mathcal{S}}(y_{(k)}^{*})=-\partial_i\partial_k g_k^{\mathcal{S}}(y')\). Similarly we find that
\begin{align*}
    I_{2}&=\int_{\substack{|y'|\leq 1 \\ y_i>0,~ y_k>0}}\left(\frac{m}{(m+1)d_{i}}\right)^{n-2}(\partial_i\partial_k g_k^{\mathcal{S}})_{+}(y')dy'\\
    I_{3}&=\int_{\substack{|y'|\leq 1 \\ y_i>0,~ y_k>0}}\left(\frac{m}{(m-1)d}\right)^{n-2}(\partial_i\partial_k g_k^{\mathcal{S}})_{+}(y')dy'\\
    I_{4}&=\int_{\substack{|y'|\leq 1 \\ y_i>0,~ y_k>0}}\left(\frac{m}{(m-1)d_{ik}}\right)^{n-2}(\partial_i\partial_k g_k^{\mathcal{S}})_{+}(y')dy'.
\end{align*}
Thus we find that
\begin{align*}
    I=\int_{\substack{|y'|\leq 1 \\ y_i>0,~ y_k>0}}\left\{\left(\frac{m}{(m+1)d_{k}}\right)^{n-2}+\left(\frac{m}{(m+1)d_{i}}\right)^{n-2}-\left(\frac{m}{(m-1)d}\right)^{n-2}-\left(\frac{m}{(m-1)d_{ik}}\right)^{n-2} \right\}\left|\partial_i \partial_k g_k^{\mathcal{S}}(y')\right|dy'.
\end{align*}
Now consider the integral
\begin{equation*}
    I':=\int_{\substack{|y'|\leq 1 \\ y_i>0,~ y_k>0}}\left(\frac{1}{d_k^{n-2}}+\frac{1}{d_i^{n-2}}-\frac{1}{d^{n-2}}-\frac{1}{d_{ik}^{n-2}}\right)\left|\partial_i \partial_k g_k^{\mathcal{S}}(y')\right|dy'.
\end{equation*}
Then we find that
\begin{align*}
    |I-I'|&\leq \int_{\substack{|y'|\leq 1 \\ y_i>0,~ y_k>0}}\left[\left(1-\left(\frac{m}{m+1}\right)^{n-2}\right)\frac{1}{d_k^{n-2}}+\left(1-\left(\frac{m}{m+1}\right)^{n-2}\right)\frac{1}{d_i^{n-2}}\right.\\
    &\left.+\left(\left(\frac{m}{m-1}\right)^{n-2}-1\right)\frac{1}{d^{n-2}}+\left(\left(\frac{m}{m-1}\right)^{n-2}-1\right)\frac{1}{d_{ik}^{n-2}}\right]\left|\partial_i \partial_k g_k^{\mathcal{S}}(y')\right|dy'.
\end{align*}

Here using the inequalities
\begin{align*}
    1-\left(\frac{m}{m+1}\right)^{n-2}\leq \frac{n-2}{m+1},~~~~~~\left(\frac{m}{m-1}\right)^{n-2}-1\leq \frac{2^{n-2}}{m-1},
\end{align*}
and the fact that \(\min\left\{d, d_i, d_k, d_{ik}\right\}=d_i\geq \frac{2}{3}|x'|\) since \(x_i<0, x_k>0\), we find that
\begin{align*}
    \left|I-I'\right|\leq C_n\left(\frac{1}{m+1}+\frac{1}{m-1}\right)\frac{1}{|x'|^{n-2}}\int_{\substack{|y'|\leq 1 \\ y_i>0,~ y_k>0}}\left|\partial_i \partial_k g_k^{\mathcal{S}}(y')\right|dy',
\end{align*}
where \(C_n\) is a constant depending only on \(n\), which will vary line by line.

Now for any given \(\epsilon>0\), take \(m\geq \frac{C_m|x'|^4}{\epsilon}\) where \(C_m\) is a sufficiently large positive number independent of \(x'\) and \(\epsilon\), we then note that
\begin{align*}
    \frac{1}{m+1}+\frac{1}{m-1}\leq \frac{1}{m}+\frac{9C_m}{9C_m-1}\frac{1}{m}\leq \frac{\epsilon C_n}{C_m}\frac{1}{|x'|^4}
\end{align*}
(we will only need \(C_m\geq \frac{2}{9}\) to obtain the above estimate). Hence we obtain that
\begin{align*}
    |I-I'|\leq \frac{\epsilon C_n}{C_m}\frac{1}{|x'|^{n+2}}\int_{\substack{|y'|\leq 1 \\ y_i>0,~ y_k>0}}\left|\partial_i \partial_k g_k^{\mathcal{S}}(y')\right|dy'\leq \frac{\epsilon C_1}{C_m}\frac{1}{|x'|^{n+2}},
\end{align*}
where the above \(C_1\) depends only on \(n\) and \(g_k^\mathcal{S}\).

We now estimate \(I'\). For this we need a lemma which generalizes Lemma A.3 of \cite{KangTsai22FE}, for which we skip its proof.
\begin{lem}
    Let \(n\geq 3\), \(H(x,y)=(x^2+y^2)^{-\frac{n-2}{2}}\), \(0\leq t<u\) and \(0\leq a<b\). Then
    \begin{align*}
        H(t,a)-H(t,b)-H(u,a)+H(u,b)\geq \frac{n(n-2)}{4}(u^2-t^2)(b^2-a^2)H(u,b)^{\frac{n+2}{n-2}}.
    \end{align*}
\end{lem}

Let us define the variable \(x'':=(x_1,\cdots, x_{k-1}, x_{k+1},\cdots, x_{n-1})\in\mathbb{R}^{n-2}\). Then letting \(a=|x_i+y_i|\), \(b=|x_i-y_i|\), \(\displaystyle u=|x''-y''|+|x_k+y_k|\), \(\displaystyle t=|x''-y''|+|x_k-y_k|\), and \(R:=|x''-y''|\), we find that the above lemma implies
\begin{align*}
    \frac{1}{d_k^{n-2}}+\frac{1}{d_i^{n-2}}-\frac{1}{d^{n-2}}-\frac{1}{d_{ik}^{n-2}}&\geq \frac{n(n-2)}{4}(4|x_i|y_i)(4x_k y_k+4Rx_k)\frac{1}{d_k^{n+2}}\\
    &\geq 4n(n-2)|x_i|x_k y_i y_k \frac{1}{d_k^{n+2}}\geq \frac{C\alpha\beta}{|x'|^{n+2}},
\end{align*}
where in the last inequality we used that \(d_k=|x'-y_{(k)}^*|\geq |x'|-1\geq\frac{2}{3}|x'|\). We then conclude that
\begin{align*}
    I'\geq \frac{C_2\alpha\beta}{|x'|^{n+2}}.
\end{align*}
Thus, it follows that
\begin{align*}
    I\geq I'-|I-I'|\geq \frac{C_2\alpha\beta}{|x'|^{n+2}}-\frac{\epsilon C_1}{C_m}\frac{1}{|x'|^{n+2}},
\end{align*}
and by choosing \(\epsilon>0\) such that \(C_2\alpha\beta-\epsilon\frac{C_1}{C_m}>C\alpha\beta\) for some \(C>0\), we find that
\begin{align*}
    I\geq \frac{C\alpha\beta}{|x'|^{n+2}}.
\end{align*}
Thus we conclude that
\begin{align*}
    \frac{K_{281}(x,1)+J}{aC_n(4\pi)^\frac{n-1}{2}}\geq \left(\mathbbm{1}_{a<\frac{1}{2}}\left(\frac{x_n^2}{4}\right)^{a-\frac{1}{2}}+\mathbbm{1}_{a=\frac{1}{2}}\log \frac{2}{x_n}\right)\frac{C\alpha\beta}{|x'|^{n+2}}.
\end{align*}
We now estimate \(J\):
\begin{align*}
    \frac{J}{aC_n(4\pi)^{\frac{n-1}{2}}}&=\int_{0}^{\frac{1}{2}}\int_{|y'|\leq 1} e^{-\frac{x_n^2}{4s}}\frac{1}{s^{-a+\frac{3}{2}}}\frac{1}{d^{n-2}}e^{-\frac{d^2}{8m^2 s}}\left|\partial_i\partial_k g_k^{\mathcal{S}}(y')\right|dy'ds\\
    &\leq \int_{0}^{\frac{1}{2}}\int_{|y'|\leq 1} \frac{1}{s^{-a+\frac{3}{2}}}\frac{1}{d^{n-2}}e^{-\frac{d^2}{8m^2 s}}\left|\partial_i\partial_k g_k^{\mathcal{S}}(y')\right|dy'ds\\
    &= \int_{|y'|\leq 1}\frac{1}{d^{n-2}}\left|\partial_i\partial_k g_k^{\mathcal{S}}(y')\right|\int_{0}^{\frac{1}{2}} \frac{1}{s^{-a+\frac{3}{2}}}e^{-\frac{d^2}{8m^2 s}}dsdy'.
\end{align*}
Similarly as before, we have that for \(m\geq d\),
\begin{equation*}
    \int_{0}^{\frac{1}{2}} \frac{1}{s^{-a+\frac{3}{2}}}e^{-\frac{d^2}{8m^2 s}}ds=\mathbbm{1}_{a<\frac{1}{2}}\left(\frac{d^2}{8m^2}\right)^{a-\frac{1}{2}}+\mathbbm{1}_{a=\frac{1}{2}}\log \frac{4m^2}{d^2}.
\end{equation*}
Thus, by choosing \(m=2|x'|>d\), we get 
\begin{align*}
    J&\lesssim \int_{|y'|\leq 1}\frac{1}{d^{n-2}}  \left[\mathbbm{1}_{a<\frac{1}{2}}\left(\frac{d^2}{8m^2}\right)^{a-\frac{1}{2}}+\mathbbm{1}_{a=\frac{1}{2}}\log \frac{4m^2}{d^2}.\right]\left|\partial_i\partial_k g_k^{\mathcal{S}}(y')\right|dy'\\
    &\lesssim \mathbbm{1}_{a<\frac{1}{2}}\frac{m^{1-2a}}{|x'|^{n-1-2a}}+\mathbbm{1}_{a=\frac{1}{2}}\frac{1}{|x'|^{n-2}}\log\left(\frac{3m}{|x'|}\right)\lesssim \frac{1}{|x'|^{n-2}}.
\end{align*}
Thus we obtain that
\begin{align*}
    K_{281}(x,1)\gtrsim \left(\mathbbm{1}_{a<\frac{1}{2}}\frac{1}{x_n^{1-2a}}+\mathbbm{1}_{a=\frac{1}{2}}\log \frac{2}{x_n^2}\right)\frac{C\alpha\beta}{|x'|^{n+2}}-\frac{C}{|x'|^{n-2}}.
\end{align*}
Then finally we see that 
\begin{align*}
    |K_{28}(x,1)|&=|K_{281}(x,1)+K_{282}(x,1)|\\
    &\geq |K_{281}(x,1)|-|K_{282}(x,1)|\\
    &\gtrsim \left(\mathbbm{1}_{a<\frac{1}{2}}\frac{1}{x_n^{1-2a}}+\mathbbm{1}_{a=\frac{1}{2}}\log \frac{2}{x_n}\right)\frac{C\alpha\beta}{|x'|^{n+2}}-\frac{C}{|x'|^{n-2}}.
\end{align*}
Thus for \(1\leq i \leq n-1\), (\ref{w_iupper}) gives
\begin{align*}
    |\partial_n^2 w_i(x,1)|&\gtrsim\left(\mathbbm{1}_{a<\frac{1}{2}}\frac{1}{x_n^{1-2a}}+\mathbbm{1}_{a=\frac{1}{2}}\log \frac{2}{x_n}\right)\frac{C\alpha\beta}{|x'|^{n+2}}-\frac{C}{|x'|^{n-2}}-\frac{CN_3}{\left<x\right>^{n}}\\
    &\gtrsim\left(\mathbbm{1}_{a<\frac{1}{2}}\frac{1}{x_n^{1-2a}}+\mathbbm{1}_{a=\frac{1}{2}}\log \frac{2}{x_n}\right)\frac{C\alpha\beta}{|x'|^{n+2}}-\frac{1}{|x'|^{n-2}}.
\end{align*}\\
2) We now consider the case \(i=k\). In this case the boundary data \(g\) are symmetric in \(x_k\) and thus we will assume that \(x_k>0\). Also in this case the integral \(K_{281}\) is not necessarily positive and thus we obtain both a lower bound and upper bound for this integral. First for a lower bound (assuming \(K_{281}\) to be positive), as done in the case \(i\neq k\), \(K_{281}\) can be written as follows:
\begin{align*}
    K_{281}(x,1)&\geq a\int_{0}^{\frac{1}{2}}\int_{|y'|\leq 1}e^{-\frac{x_n^2}{4s}}\frac{C_n}{s^{-a+\frac{3}{2}}}(4\pi)^{\frac{n-1}{2}}\left\{\left(\frac{m}{(m+1)d}\right)^{n-2}(\partial_k^2 g_k^{\mathcal{S}})_{-}(y')-\left(\frac{m}{(m-1)d}\right)^{n-2}(\partial_k^2 g_k^{\mathcal{S}})_{+}(y')\right\}dy'ds\\
    &-J,
\end{align*}
where \(J\) is now given by
\begin{align}\label{J}
    J:=a\int_{0}^{\frac{1}{2}}\int_{|y'|\leq 1}e^{-\frac{x_n^2}{4s}}\frac{C}{s^{-a+\frac{3}{2}}}\frac{1}{d^{n-2}}e^{-\frac{d^2}{8m^2s}}\left|\partial_k^2 g_k^{\mathcal{S}}(y')\right|dy'ds.
\end{align}
Thus we obtain that
\begin{align*}
    \frac{K_{281}(x,1)+J}{aC_n (4\pi)^{\frac{n-1}{2}}}
    &\geq \int_{0}^{\frac{1}{2}}e^{-\frac{x_n^2}{4s}}s^{a-\frac{3}{2}}ds\\
    &\times\int_{|y'|\leq 1}\left\{\left(\frac{m}{(m+1)d}\right)^{n-2}(\partial_k^2 g_k^{\mathcal{S}})_{-}(y')-\left(\frac{m}{(m-1)d}\right)^{n-2}(\partial_k^2 g_k^{\mathcal{S}})_{+}(y')\right\}dy'.
\end{align*}
The time integral in the RHS above is already estimated in the \(i\neq k\) case and we denote the space integral of the RHS above as \(I\). Then using the symmetry of \(\partial_k^2 g_k^{\mathcal{S}}\) with respect to \(y_k=0\) and \(y_k=\pm q\) (see Appendix 8.3), \(I\) is given as
\begin{align*}
    I&=\int_{\substack{|y'|\leq 1 \\ p\leq|y_k|\leq q}}\left(\frac{m}{(m+1)d}\right)^{n-2}\left(\partial_k^2g_k^{\mathcal{S}}\right)_{-}(y')dy'-\int_{\substack{|y'|\leq 1 \\ q\leq|y_k|\leq r}}\left(\frac{m}{(m-1)d}\right)^{n-2}\left(\partial_k^2 g_k^{\mathcal{S}}\right)_{+}(y')dy'\\
    &=\int_{\substack{|y'|\leq 1 \\ p\leq y_k\leq q}}\left(\left(\frac{m}{(m+1)d}\right)^{n-2}+\left(\frac{m}{(m+1)d_k}\right)^{n-2}\right)\left(\partial_k^2g_k^{\mathcal{S}}\right)_{-}(y')dy'\\
    &-\int_{\substack{|y'|\leq 1 \\ p\leq y_k\leq q}}\left(\left(\frac{m}{(m-1)d_q}\right)^{n-2}+\left(\frac{m}{(m-1)d_{q k}}\right)^{n-2}\right)\left(\partial_k^2g_k^{\mathcal{S}}\right)_{+}(2q\mathbbm{e}_k-y')dy'\\
    &=\int_{\substack{|y'|\leq 1 \\ p\leq y_k\leq q}}\left(\left(\frac{m}{(m+1)d}\right)^{n-2}+\left(\frac{m}{(m+1)d_k}\right)^{n-2}-\left(\frac{m}{(m-1)d_q}\right)^{n-2}-\left(\frac{m}{(m-1)d_{q k}}\right)^{n-2}\right)\left|\partial_k^2g_k^{\mathcal{S}}(y')\right|dy',
\end{align*}
where the last equality follows since \(\left(\partial_k^2g_k^{\mathcal{S}}\right)_{+}(2q\mathbbm{e}_k -y')=\left(\partial_k^2g_k^{\mathcal{S}}\right)_{-}(y')\) for \(p\leq y_k\leq q\), and
\begin{align*}
    d_q:=|x'-y_{k}^{(*)}-2q \mathbbm{e}_k|,\qquad d_{q k}:=|x'-y'+2q\mathbbm{e}_k|.
\end{align*}
Now consider the integral
\begin{align*}
    I':=\int_{\substack{|y'|\leq 1 \\ p\leq y_k\leq q}}\left(\frac{1}{d^{n-2}}+\frac{1}{d_k^{n-2}}-\frac{1}{d_q^{n-2}}-\frac{1}{d_{q k}^{n-2}}\right)\left|\partial_k^2g_k^{\mathcal{S}}(y')\right|dy'.
\end{align*}
Then we find that
\begin{align*}
    |I-I'|&\leq \int_{\substack{|y'|\leq 1 \\p\leq y_k \leq q}}\left[\left(1-\left(\frac{m}{m+1}\right)^{n-2}\right)\frac{1}{d^{n-2}}+\left(1-\left(\frac{m}{m+1}\right)^{n-2}\right)\frac{1}{d_k^{n-2}}\right.\\
    &\left.+\left(\left(\frac{m}{m-1}\right)^{n-2}-1\right)\frac{1}{d_q^{n-2}}+\left(\left(\frac{m}{m-1}\right)^{n-2}-1\right)\frac{1}{d_{q k}^{n-2}}\right]\left|\partial_k^2 g_k^{\mathcal{S}}(y')\right|dy'.
\end{align*}
Here using the inequalities
\begin{align*}
    1-\left(\frac{m}{m+1}\right)^{n-2}\leq \frac{n-2}{m+1},~~~~~~~~~~\left(\frac{m}{m-1}\right)^{n-2}-1\leq \frac{2^{n-2}}{m-1},
\end{align*}
and the fact that \(\min\left\{d, d_k, d_q, d_{q k}\right\}=d_\alpha\geq \frac{2}{3}\left(1-\frac{13\alpha}{20\sqrt{n-1}}\right)|x'|\) since \(|x'|\geq 3\), we find that
\begin{align*}
    |I-I'|\leq C_n \left(\frac{1}{m+1}+\frac{1}{m-1}\right)\frac{1}{|x'|^{n-2}}\int_{\substack{|y'|\leq 1 \\0\leq y_k \leq \alpha}}\left|\partial_k^2 g_k^{\mathcal{S}}(y')\right|dy'.
\end{align*}
Now for any given \(\epsilon>0\), take \(m\geq \frac{C_m|x'|^2}{\epsilon}\) where \(C_m\) is a sufficiently large number independent of \(x'\) and \(\epsilon\), we note that 
\begin{align*}
    \frac{1}{m+1}+\frac{1}{m-1}\leq \frac{1}{m}+\frac{9C_m}{9C_m-1}\frac{1}{m}\leq \frac{\epsilon C_n}{C_m}\frac{1}{|x'|^2}
\end{align*}
(we will only need \(C_m\geq \frac{2}{9}\) to obtain the above estimate).
\\
Hence we obtain that
\begin{align*}
    |I-I'|\leq \frac{\epsilon C_n}{C_m}\frac{1}{|x'|^{n}}\int_{\substack{|y'|\leq 1 \\0\leq y_k \leq \alpha}}\left|\partial_k^2 g_k^{\mathcal{S}}(y')\right|dy'\leq \frac{\epsilon C'}{|x'|^{n}}.
\end{align*}

We now estimate \(I'\). 

Assume that \(x_k<p\).
Note the following chain of inequalities
\begin{equation*}
    0<y_k-x_k\leq\min\left\{x_k+y_k,~2q-x_k-y_k\right\}\leq\max\left\{x_k+y_k, ~2q-x_k-y_k\right\}\leq x_k-y_k+2q.
\end{equation*}
Denote \(t:=y_k-x_k, a:=2q-2y_k, b:=2x_k\). Then from the above chain of inequalities we have that \(t>0,~ a>2\delta, ~b\geq 0, ~t+a=2q-y_k-x_k<q<1<|x'|\). 

Now define the function \(\displaystyle f(x)=\frac{1}{(R^2+x^2)^{\frac{n-2}{2}}}\) (\(x>0\)), where \(R:=|x''-y''|\leq |x'|+1\leq \frac{4}{3}|x'|\). Then since \(\displaystyle f'(x)=-\frac{(n-2)x}{(R^2+x^2)^\frac{n}{2}}\), it follows that
\begin{align*}
    \frac{1}{d^{n-2}}+\frac{1}{d_k^{n-2}}-\frac{1}{d_q^{n-2}}-\frac{1}{d_{q k}^{n-2}}&=f(t)+f(t+b)-f(t+a)-f(t+a+b)\\
    &=(n-2)\left(\int_t^{t+a}\frac{1}{(R^2+s^2)^{\frac{n}{2}}}ds+\int_{t+b}^{t+a+b}\frac{1}{(R^2+s^2)^{\frac{n}{2}}}ds\right)\\
    &\geq \frac{(n-2)a}{(R^2+(t+a)^2)^{\frac{n}{2}}}\geq\frac{2\delta(n-2)}{\left(\frac{16}{9}|x'|^2+|x'|^2\right)^\frac{n}{2}}=\frac{C_n}{|x'|^n},
\end{align*}

and thus we see that \(\displaystyle I'\geq\frac{C_n}{|x'|^n}\) and hence by taking \(\epsilon>0\) sufficiently small,
\begin{align*}
    |I|\geq |I'|-|I-I'|\geq \frac{C}{|x'|^n}-\frac{C'\epsilon}{|x'|^{n}}\geq \frac{C}{|x'|^n}.
\end{align*}
We now consider an upper bound for \(K_{281}\) (assuming \(K_{281}\) to be negative),
\begin{align*}
    K_{281}(x,1)
    &=-a\int_{0}^{\frac{1}{2}}\int_{|y'|\leq 1}e^{-\frac{x_n^2}{4s}}\frac{C_n}{s^{\frac{n}{2}-a+1}}K(x'-y',s)\left((\partial_k^2 g_k^{\mathcal{S}})_{+}(y')-(\partial_k^2 g_k^{\mathcal{S}})_{-}(y')\right)dy'ds\\
    &\leq -a\int_{0}^{\frac{1}{2}}\int_{|y'|\leq 1}e^{-\frac{x_n^2}{4s}}\frac{C_n}{s^{\frac{n}{2}-a+1}}\left(\frac{m}{(m+1)d}\right)^{n-2}s^{\frac{n-1}{2}}(4\pi)^{\frac{n-1}{2}}\left(1-2^{\frac{n-1}{4}}e^{-\frac{d^2}{8m^2s}}\right)(\partial_k^2 g_k^{\mathcal{S}})_{+}(y')dy'ds\\
    &+a\int_{0}^{\frac{1}{2}}\int_{|y'|\leq 1}e^{-\frac{x_n^2}{4s}}\frac{C_n}{s^{\frac{n}{2}-a+1}}\left(\left(\frac{m}{(m-1)d}\right)^{n-2}s^{\frac{n-2}{2}}(4\pi)^{\frac{n-1}{2}}+\frac{C}{d^{n-2}}s^{\frac{n-1}{2}}e^{-\frac{d^2}{8m^2s}}\right)(\partial_k^2 g_k^{\mathcal{S}})_{-}(y')dy'ds\\
     &\leq -a\int_0^{\frac{1}{2}} e^{-\frac{x_n^2}{4s}}\frac{C_n}{s^{\frac{3}{2}-a}}(4\pi)^{\frac{n-1}{2}} ds\int_{|y'|\leq 1}
    \left\{\left(\frac{m}{(m+1)d}\right)^{n-2}(\partial_k^2g_k^{\mathcal{S}})_+(y')-\left(\frac{m}{(m-1)d}\right)^{n-2}(\partial_k^2g_k^{\mathcal{S}})_{-}(y')\right\}dy'\\
     &+J,
\end{align*}
where \(J\) is the integral previously defined in \eqref{J}. Let \(I\) be the space integral of RHS above. Then \(I\) is given as follows
\begin{align*}
    I=\int_{\substack{|y'|\leq 1 \\ q\leq y_k\leq r}}\left(\left(\frac{m}{(m+1)d}\right)^{n-2}+\left(\frac{m}{(m+1)d_k}\right)^{n-2}-\left(\frac{m}{(m-1)d_q}\right)^{n-2}-\left(\frac{m}{(m-1)d_{q k}}\right)^{n-2}\right)\left|\partial_k^2g_k^{\mathcal{S}}(y')\right|dy'.
\end{align*}
Now consider the integral
\begin{align*}
    I'=\int_{\substack{|y'|\leq 1 \\ q\leq y_k\leq r}}\left(\frac{1}{d^{n-2}}+\frac{1}{d_k^{n-2}}-\frac{1}{d_q^{n-2}}-\frac{1}{d_{qk}^{n-2}}\right)\left|\partial_k^2g_k^{\mathcal{S}}(y')\right|dy'.
\end{align*}
Similarly as in the previous case, we can show that for any \(\epsilon >0\), by taking \(m=m(\epsilon)\) sufficiently large, 
\begin{align*}
    |I-I'|\leq \frac{C'\epsilon}{|x'|^n}.
\end{align*}

We now estimate the integral \(I'\).

Assume that \(x_k>r\) and \(|x''|<\frac{3}{2}\). Note the following chain of inequalities
\begin{align*}
    0<x_k-y_k<x_k+y_k-2q<x_k-y_k+2q<x_k+y_k.
\end{align*}
Denote \(t:=x_k-y_k,~a:=2y_k-2q,~b=2q\). Then from the above chain of inequalities we have that \(t,\,a,\,b>0\). Then since \(\displaystyle f''(x)=\frac{(n-2)((n-1)x^2-R^2)}{(x^2+R^2)^\frac{n+2}{2}}\),
\begin{align*}
    \frac{1}{d^{n-2}}&+\frac{1}{d_k^{n-2}}-\frac{1}{d_q^{n-2}}-\frac{1}{d_{qk}^{n-2}}=f(t)+f(t+a+b)-f(t+a)-f(t+b)\\
    =&-\int_t^{t+a}f'(s)ds+\int_{t+b}^{t+a+b}f'(s)ds=\int_t^{t+a}(f'(s+b)-f'(s))ds\\
    =&\int_t^{t+a}\int_s^{s+b}f''(\tau)d\tau ds=(n-2)\int_t^{t+a}\int_s^{s+b}\frac{(n-1)\tau^2-R^2}{(\tau^2+R^2)^{\frac{n+2}{2}}}d\tau ds.
\end{align*}
Now  since \( R=|x''-y''|\leq |x''|+|y''|<\frac{3}{2}+\sqrt{n-2}r\) and \(\tau\geq s\geq t=x_k-y_k>\frac{3}{2}-r\), we have that
\begin{align}\label{eq1}
   (n-1)\tau^2-R^2\geq (n-1)\left(\frac{3}{2}-r\right)^2-\left(\frac{3}{2}+\sqrt{n-2}r)^2\right).
\end{align}
On the other hand since \( R\leq |x'|+1\leq \frac{4}{3}|x'|\) and \(\tau<t+a+b=x_k+y_k\leq \frac{4}{3}|x'|\), we have that
\begin{align}\label{eq2}
   (\tau^2+R^2)^{\frac{n+2}{2}}\leq 2^{\frac{n+2}{2}}\left(\frac{4}{3}\right)^{n+2}|x'|^{n+2}.
\end{align}
With these estimates \eqref{eq1} and \eqref{eq2}, we find that
\begin{align*}
    \int_t^{t+a}\int_s^{s+b}\frac{(n-1)\tau^2-R^2}{(\tau^2+R^2)^{\frac{n+2}{2}}}d\tau ds&\geq  C_n\int_t^{t+a}\int_s^{s+b} \frac{(n-1)\left(\frac{3}{2}-r\right)^2-\left(\frac{3}{2}+\sqrt{n-2}r)^2\right)}{|x'|^{n+2}}d\tau ds\\
    &\geq C_nab \frac{(n-1)\left(\frac{3}{2}-r\right)^2-\left(\frac{3}{2}+\sqrt{n-2}r)^2\right)}{|x'|^{n+2}}=\frac{C_n q\delta}{|x'|^{n+2}},
\end{align*}
and thus we see that \(I' \geq \frac{C_n}{|x'|^{n+2}}\) and hence by taking \(\epsilon>0\) sufficiently small,
\begin{align*}
    |I|\geq |I'|-|I-I'|\geq \frac{C}{|x'|^{n+2}}-\frac{C'\epsilon}{|x'|^{n+2}}\geq \frac{C}{|x'|^{n+2}}.
\end{align*}
Summarizing above, we get that
\begin{align*}
K_{281}(x,1)\begin{cases}
\geq  \left(\mathbbm{1}_{a<\frac{1}{2}}\frac{1}{x_n^{1-2a}}+\mathbbm{1}_{a=\frac{1}{2}}\log \frac{2}{x_n^2}\right)\frac{C}{|x'|^{n}}-\frac{C}{|x'|^{n-2}} & \text{ if } x_k <p, \\ 
 \leq -\left(\mathbbm{1}_{a<\frac{1}{2}}\frac{1}{x_n^{1-2a}}+\mathbbm{1}_{a=\frac{1}{2}}\log \frac{2}{x_n^2}\right)\frac{C}{|x'|^{n+2}}+\frac{C}{|x'|^{n-2}}& \text{ if } x_k>r \text{ and } |x''|<\alpha.
\end{cases}
\end{align*}
Then we finally see that 
\begin{align*}
    |K_{28}(x,1)|&=|K_{281}(x,1)+K_{282}(x,1)|\\
    &\geq |K_{281}(x,1)|-|K_{282}(x,1)|\\
    &\gtrsim \left(\mathbbm{1}_{a<\frac{1}{2}}\frac{1}{x_n^{1-2a}}+\mathbbm{1}_{a=\frac{1}{2}}\log \frac{2}{x_n}\right)\frac{C}{|x'|^{n+2}}-\frac{C}{|x'|^{n-2}},
\end{align*}
and thus for \(i=k\), (\ref{w_iupper}) gives
\begin{align*}
    |\partial_n^2 w_i(x,1)|&\gtrsim\left(\mathbbm{1}_{a<\frac{1}{2}}\frac{1}{x_n^{1-2a}}+\mathbbm{1}_{a=\frac{1}{2}}\log \frac{2}{x_n}\right)\frac{C}{|x'|^{n+2}}-\frac{C}{|x'|^{n-2}}-\frac{CN_3}{\left<x\right>^{n}}\\
    &\gtrsim\left(\mathbbm{1}_{a<\frac{1}{2}}\frac{1}{x_n^{1-2a}}+\mathbbm{1}_{a=\frac{1}{2}}\log \frac{2}{x_n}\right)\frac{C}{|x'|^{n+2}}-\frac{1}{|x'|^{n-2}}.
\end{align*}


\section{Proof of Theorem \ref{thm2} for the Navier-Stokes equations}\label{sec7}
Let \(w\) be a solution of the Stokes equations (\ref{stokeseq2})-(\ref{stokesbdry2}) with \(f=0\) and \(g\) given in Theorem 2.\\
Let \(\phi_1\in C_c^{\infty}(\mathbb{R}^n)\) be a cut-off function satisfying \(\phi_1 \geq 0\), \(\textrm{supp}(\phi_1)\subset B_2\) and \(\phi_1 \equiv 1\) in \(B_1\). Also let \(\phi_2\in C_c^{\infty}(-\infty, \infty)\) be a cut-off function satisfying \(\phi_2 \geq 0\), \(\textrm{supp}(\phi_2)\subset (-2,2)\) and \(\phi_2 \equiv 1\) in \((-1,1)\). Set \(\phi(x,t)=\phi_1(x)\phi_2(t)\) and define \(W=\phi w\). Then it is immediate that \(W=w\) in \(Q_1^+:=B_1^+\times(0,1)\) and \(W\bigm|_{x_n=0}=W\bigm|_{t=0}=0\). From the proof of Theorem \ref{thm2} for the Stokes equations, we find that
\begin{align*}
    \left\|W\right\|_{L^r(\mathbb{R}_+^n\times(0,1))}+ \left\|\nabla W\right\|_{L^r(\mathbb{R}_+^n\times(0,1))} &+\left\|\nabla\nabla' W\right\|_{L^{r}(\mathbb{R}_+^n\times (0,1))}\\
    &\leq C\left( \left\|w\right\|_{L^r(B_2^+\times(0,4))}+ \left\|\nabla w\right\|_{L^r(B_2^+\times(0,4))}+\left\|\nabla\nabla' w\right\|_{L^{r}(B_2^+\times (0,4))}\right)\\
    &\leq C\left\|g\right\|_{L^{\infty}((0,\infty);W^{2,1}(A))}\leq Ca
\end{align*}
for all \(1\leq r\leq \infty\), where \(\nabla':=(\partial_1,\cdots,\partial_{n-1})\) and \(a>0\) will be determined later.\\
We now consider the following perturbed Navier-Stokes equations in \(Q_+:=\mathbb{R}_+^n\times (0,1)\):
\begin{equation}\label{pertns}
\begin{dcases}
\begin{aligned}
  \partial_t v-\Delta v+\nabla q+\nabla \cdot (v\otimes v+v\otimes W+W\otimes v)&=-\nabla\cdot(W\otimes W),\\
  \text{div } v&= 0
\end{aligned}
\end{dcases}
\text{  in }  Q_+,
\end{equation}
with zero initial and boundary data
\begin{equation*}
    v\bigm|_{x_n=0}=v\bigm|_{t=0}=0.
\end{equation*}
We first show that the solution \(v\) for (\ref{pertns}) satisfies \(v, \nabla v\in L^r(Q_+)\cap L^\infty(Q_+)\) and \(\nabla^2 v\in L^r(Q_+)\) for all \(n+2<r<\infty\).\\
To show the existence we consider the following iteration: For a positive integer \(m\geq 1\),
\begin{equation}\label{iteration}
\begin{dcases}
\begin{aligned}
  \partial_t v^{m+1}-\Delta v^{m+1}+\nabla q^{m+1}&=-\nabla \cdot (v^m\otimes v^m+v^m\otimes W+W\otimes v^m+W\otimes W),\\
  \text{div } v^{m+1}&= 0
\end{aligned}
\end{dcases}
\text{  in }  Q_+,
\end{equation}
with zero initial and boundary data
\begin{equation*}
    v^{m+1}\bigm|_{x_n=0}=v^{m+1}\bigm|_{t=0}=0,
\end{equation*}
and define \(v^1\equiv 0\) and \(q^1\equiv 0\).\\

1) We first show the uniform-in-\(m\) estimate. 
By Proposition \ref{1}, we have
\begin{equation*}
    \left\|v^2(t)\right\|_{L^{r}(\mathbb{R}_+^n)}\leq C\int_{0}^{t}(t-s)^{-\frac{1}{2}}\left\|W(s)\otimes W(s)\right\|_{L^{r}(\mathbb{R}_+^n)}ds.
\end{equation*}
Using the above estimate and Young's inequality for convolution, we find that for \(r\in (1,\infty)\),
\begin{equation*}
    \left\|v^2\right\|_{L^r(Q_+)}\leq C_1\left\|W\otimes W\right\|_{L^{r}(Q_+)}\leq C_1\left\|W\right\|_{L^r(Q_+)}\left\|W\right\|_{L^\infty(Q_+)}.
\end{equation*}
Also by Proposition \ref{3}, we have for \(r>n+2\) that,
\begin{equation*}
     \left\|v^2\right\|_{L^\infty(Q_+)}\leq C_2\left\|W\otimes W\right\|_{L^{r}(Q_+)}\leq C_2\left\|W\right\|_{L^r(Q_+)}\left\|W\right\|_{L^\infty(Q_+)}.
\end{equation*}
Finally by Proposition \ref{2}, we have for \(r\in(1,\infty)\) that
\begin{equation*}
     \left\|\nabla v^2\right\|_{L^r(Q_+)}\leq C_3\left\|W\otimes W\right\|_{L^{r}(Q_+)}\leq C_3\left\|W\right\|_{L^r(Q_+)}\left\|W\right\|_{L^\infty(Q_+)}.
\end{equation*}
By the maximal regularity of the initial-boundary value problem for the Stokes equations in the half-space, we have 
\begin{equation*}
    \left\|\partial_t v^2\right\|_{L^r(Q_+)}+\left\|\nabla^2 v^2\right\|_{L^r(Q_+)}\leq \frac{1}{2n} C_4\left\|\nabla\cdot (W\otimes W)\right\|_{L^{r}(Q_+)}\leq C_4\left\|W\right\|_{L^{\infty}(Q_+)}\left\|\nabla W\right\|_{L^{r}(Q_+)}.
\end{equation*}
Then by the Sobolev embedding, it follows that
\begin{equation*}
    \left\|\nabla v^2\right\|_{L^\infty(Q_+)}\leq \frac{C_5}{C_4}\left(\left\|\nabla^2 v^2\right\|_{L^r(Q_+)}+\left\|\partial_t v^2\right\|_{L^r(Q_+)}\right)\leq C_5 \left\|W\right\|_{L^{\infty}(Q_+)}\left(\left\|W\right\|_{L^{r}(Q_+)}+\left\|\nabla W\right\|_{L^{r}(Q_+)}\right).
\end{equation*}
We now differentiate the equation (\ref{iteration}) in \(x_j\) (\(j=1,\,\cdots,\, n-1\)) to get
\begin{equation}
\begin{dcases}
\begin{aligned}
  \partial_t \partial_j v^{m+1}-\Delta \partial_j v^{m+1}+\nabla \partial_j q^{m+1}&=-\nabla \cdot (\partial_j v^m\otimes v^m+ v^m\otimes \partial_j v^m+\partial_j v^m\otimes W+v^m\otimes \partial_j W\\
  &~~~+\partial_j W\otimes v^m+W\otimes \partial_j v^m+\partial_j W\otimes W+W\otimes \partial_jW),\\
  \text{div } \partial_j v^{m+1}&= 0
\end{aligned}
\end{dcases}
\text{  in }  Q_+,
\end{equation}
with zero initial and boundary data
\begin{equation*}
    \partial_j v^{m+1}\mid_{x_n=0}=\partial_j v^{m+1}\mid_{t=0}=0.
\end{equation*}
By the maximal regularity for the Stokes equations, we find that
\begin{align*}
     \left\|\partial_t  v^2\right\|_{L^r(Q_+)}+\left\|\nabla^2 \partial_j v^2\right\|_{L^r(Q_+)}&\leq \frac{1}{2n}C_4\left\|\nabla\cdot(\partial_j W\otimes W+W\otimes \partial_j W)\right\|_{L^r(Q_+)}\\
    &\leq \frac{1}{2n} C_4\left(\left\|\nabla W\right\|_{L^{r}(Q_+)}\left\|\nabla W\right\|_{L^{\infty}(Q_+)}+\left\|\nabla\nabla' W\right\|_{L^{r}(Q_+)}\left\|W\right\|_{L^{\infty}(Q_+)}\right).
\end{align*}
We then find that since \(\nabla \cdot v^2=0\),
\begin{equation*}
    \left\|\nabla^2 \partial_n v_n^2\right\|_{L^{r}(Q_+)}\leq \sum_{j=1}^{n-1}\left\|\nabla^2 \partial_j v_j^2\right\|_{L^{r}(Q_+)}\leq  C_4\left(\left\|\nabla W\right\|_{L^{r}(Q_+)}\left\|\nabla W\right\|_{L^{\infty}(Q_+)}+\left\|\nabla\nabla' W\right\|_{L^{r}(Q_+)}\left\|W\right\|_{L^{\infty}(Q_+)}\right).
\end{equation*}
Then we have the following estimates:
 for \(j\neq n\) and \(1\leq k \leq n\) and \(r>n+2\),
\begin{align*}
    \left\|\partial_j\partial_k v^2\right\|_{L^{\infty}(Q_+)}&\leq \frac{C_5}{C_4}\left(\left\|\nabla^2\partial_j v^2\right\|_{L^{r}(Q_+)}+\left\|\partial_t\partial_j v^2\right\|_{L^{r}(Q_+)}\right)\\
    &\leq C_5\left(\left\|W\right\|_{L^{\infty}(Q_+)}+\left\|\nabla W\right\|_{L^{\infty}(Q_+)}\right)\left(\left\|\nabla W\right\|_{L^{r}(Q_+)}+\left\|\nabla\nabla'W\right\|_{L^{r}(Q_+)}\right)
\end{align*}
so that by letting \(C_6:=2n^2C_5\),
\begin{align*}
    \left\|\nabla\nabla'v^2\right\|_{L^{\infty}(Q_+)}\leq C_6 \left(\left\|W\right\|_{L^{\infty}(Q_+)}+\left\|\nabla W\right\|_{L^{\infty}(Q_+)}\right)\left(\left\|\nabla W\right\|_{L^{r}(Q_+)}+\left\|\nabla\nabla'W\right\|_{L^{r}(Q_+)}\right),
\end{align*}
and 
\begin{align*}
    \left\|\partial_n^2 v_n^2\right\|_{L^{\infty}(Q_+)}&\leq\sum_{k=1}^{n-1}\left\|\partial_n\partial_k v_k^2\right\|_{L^{\infty}(Q_+)}\leq C_6\left(\left\|W\right\|_{L^{\infty}(Q_+)}+\left\|\nabla W\right\|_{L^{\infty}(Q_+)}\right)\left(\left\|\nabla W\right\|_{L^{r}(Q_+)}+\left\|\nabla\nabla'W\right\|_{L^{r}(Q_+)}\right).
\end{align*}
Now let \(A:=\left\|W\right\|_{L^{r}(Q_+)}+\left\|W\right\|_{L^{\infty}(Q_+)}+\left\|\nabla W\right\|_{L^{r}(Q_+)}+\left\|\nabla W\right\|_{L^{\infty}(Q_+)}+\left\|\nabla\nabla'W\right\|_{L^{r}(Q_+)}\leq c^*a\)
for some fixed constant \(c^*>0\). 

Define \(\displaystyle C':=\max_{1\leq k\leq 6}C_k\). Then we find that by taking \(a>0\) sufficiently small such that \(C'c^*a\leq \frac{1}{100}\),
\begin{align*}
    \left\|v^2\right\|_{L^{r}(Q_+)}&+\left\|v^2\right\|_{L^{\infty}(Q_+)}+\left\|\nabla v^2\right\|_{L^{r}(Q_+)}+\left\|\nabla v^2\right\|_{L^{\infty}(Q_+)}+\left\|\nabla^2 v^2\right\|_{L^{r}(Q_+)}+\left\|\nabla\nabla'v^2\right\|_{L^{\infty}(Q_+)}+\left\|\partial_n^2 v_n^2\right\|_{L^{\infty}(Q_+)}\\
    &\leq 7C'A^2\leq 7C'c^*aA\leq \frac{7}{100}A<A.
\end{align*}
Suppose for \(m\geq 2\)
\begin{equation}\label{mthineq}
     \left\|v^m\right\|_{L^{r}(Q_+)}+\left\|v^m\right\|_{L^{\infty}(Q_+)}+\left\|\nabla v^m\right\|_{L^{r}(Q_+)}+\left\|\nabla v^m\right\|_{L^{\infty}(Q_+)}+\left\|\nabla^2 v^m\right\|_{L^{r}(Q_+)}+\left\|\nabla\nabla'v^m\right\|_{L^{\infty}(Q_+)}+\left\|\partial_n^2 v_n^m\right\|_{L^{\infty}(Q_+)}\leq A .
\end{equation}
We now estimate each term of the left hand side of the above inequality with \(v^m\) replaced by \(v^{m+1}\).
\begin{align*}
    \left\|v^{m+1}\right\|_{L^{r}(Q_+)} +\left\|v^{m+1}\right\|_{L^{\infty}(Q_+)}&+\left\|\nabla v^{m+1}\right\|_{L^{r}(Q_+)}\\
    &\leq (C_1+C_2+C_3)\left\|v^m\otimes v^m+v^m\otimes W+W\otimes v^m+W\otimes W\right\|_{L^{r}(Q_+)}\\
    &\leq (C_1+C_2+C_3)\left(\left\|v^m\right\|_{L^{r}(Q_+)}\left\|v^m\right\|_{L^{\infty}(Q_+)}+\left\|v^m\right\|_{L^{r}(Q_+)}\left\|W\right\|_{L^{\infty}(Q_+)}\right.\\
    &\left.+\left\|W\right\|_{L^{r}(Q_+)}\left\|v^m\right\|_{L^{\infty}(Q_+)}+\left\|W\right\|_{L^{r}(Q_+)}\left\|W\right\|_{L^{\infty}(Q_+)}\right)\\
    &\leq 4(C_1+C_2+C_3)A^2,
\end{align*}
\begin{align*}
    \left\|\partial_t v^{m+1}\right\|_{L^{r}(Q_+)}+\left\|\nabla^2 v^{m+1}\right\|_{L^{r}(Q_+)}&\leq\frac{C_4}{2n}\left\|\nabla\cdot(v^m\otimes v^m+v^m\otimes W+W\otimes v^m+W\otimes W)\right\|_{L^{r}(Q_+)}\\
    &\leq C_4\left(\left\|\nabla v^m\right\|_{L^{r}(Q_+)}\left\|v^m\right\|_{L^{\infty}(Q_+)}+\left\|\nabla v^m\right\|_{L^{r}(Q_+)}\left\|W\right\|_{L^{\infty}(Q_+)}\right.\\
    &\left.+\left\|\nabla W\right\|_{L^{r}(Q_+)}\left\|v^m\right\|_{L^{\infty}(Q_+)}+\left\|\nabla W\right\|_{L^{r}(Q_+)}\left\|W\right\|_{L^{\infty}(Q_+)}\right)\\
    &\leq 4C_4A^2,
\end{align*}
and
\begin{align*}
    \left\|\nabla v^{m+1}\right\|_{L^{\infty}(Q_+)}&\leq \frac{C_5}{C_4}\left(\left\|\nabla^2 v^{m+1}\right\|_{L^{r}(Q_+)}+\left\|\partial_t v^{m+1}\right\|_{L^{r}(Q_+)}\right)\leq 4\frac{C_5}{C_4}C_4A^2=4C_5A^2.
\end{align*}
By the maximal regularity for the Stokes equations, we find that for \(1\leq j \leq n-1\),
\begin{align*}
   \left\|\partial_t \partial_j v^{m+1} \right\|_{L^{r}(Q_+)}+\left\|\nabla^2 \partial_j v^{m+1} \right\|_{L^{r}(Q_+)}&\leq \frac{C_4}{2n}\left\|\nabla \cdot (\partial_j v^m\otimes v^m+ v^m\otimes \partial_j v^m+\partial_j v^m\otimes W+v^m\otimes \partial_j W\right.\\
    &\left.+\partial_j W\otimes v^m+W\otimes \partial_j v^m+\partial_j W\otimes W+W\otimes \partial_jW\right\|_{L^{r}(Q_+)}\\
    &\leq \frac{C_4}{2n}2n\left(\left\|v^m\right\|_{L^{r}(Q_+)}+\left\|\nabla v^m\right\|_{L^{r}(Q_+)}+\left\|\nabla\nabla'v^m\right\|_{L^{r}(Q_+)}\right.\\
    &\left.+\left\|W\right\|_{L^{r}(Q_+)}+\left\|\nabla W\right\|_{L^{r}(Q_+)}+\left\|\nabla\nabla'W\right\|_{L^{r}(Q_+)}\right)\\
    &\times\left(\left\|v^m\right\|_{L^{\infty}(Q_+)}+\left\|\nabla v^m\right\|_{L^{\infty}(Q_+)}+\left\|\nabla\nabla'v^m\right\|_{L^{\infty}(Q_+)}\right.\\
    &\left.+\left\|W\right\|_{L^{\infty}(Q_+)}+\left\|\nabla W\right\|_{L^{\infty}(Q_+)}+\left\|\nabla\nabla'W\right\|_{L^{\infty}(Q_+)}\right)\\
    &\leq 4C_4 A^2.
\end{align*}
We then see that
\begin{align*}
    \left\|\nabla^2\partial_n v_n^{m+1}\right\|_{L^{r}(Q_+)}\leq \sum_{j=1}^{n-1}\left\|\nabla^2 \partial_j v_j^{m+1}\right\|_{L^{r}(Q_+)}\leq 4nC_4A^2.
\end{align*}
Then we have that for \(j\neq n\) and \(1\leq k\leq n\),
\begin{align*}
    \left\|\partial_j\partial_k v^{m+1}\right\|_{L^{\infty}(Q_+)}&\leq \frac{C_5}{C_4}\left(\left\|\nabla^2\partial_jv^{m+1}\right\|_{L^r(Q_+)}+\left\|\partial_t\partial_jv^{m+1}\right\|_{L^r(Q_+)}\right)\leq \frac{4C_4 C_5}{C_4}A^2= 4C_5 A^2,
\end{align*}
and we see that
\begin{align*}
    \left\|\nabla\nabla'v^{m+1}\right\|_{L^{\infty}(Q_+)}\leq 4n^2 C_5A^2=2C_6 A^2.
\end{align*}
Finally we find that
\begin{align*}
    \left\|\partial_n^2 v_n^{m+1}\right\|_{L^\infty(Q_+)}&\leq \sum_{k=1}^{n-1}\left\|\partial_k^2 v_k^{m+1}\right\|_{L^\infty(Q_+)}\leq 2C_6 A^2.
\end{align*}
Hence summing all the estimates above gives
\begin{align*}
     \left\|v^m\right\|_{L^{r}(Q_+)}+\left\|v^m\right\|_{L^{\infty}(Q_+)}+\left\|\nabla v^m\right\|_{L^{r}(Q_+)}&+\left\|\nabla v^m\right\|_{L^{\infty}(Q_+)}+\left\|\nabla^2 v^m\right\|_{L^{r}(Q_+)}+\left\|\nabla\nabla'v^m\right\|_{L^{\infty}(Q_+)}+\left\|\partial_n^2 v_n^m\right\|_{L^{\infty}(Q_+)}\\
     &\leq 24C'A^2 \leq 24C'c^*aA\leq \frac{6}{25}A<A.
\end{align*}
Thus (\ref{mthineq}) holds for all \(m\geq 2\).\\

2. We now show the Cauchy estimate. Denote \(V^m:=v^{m+1}-v^{m}\) and \(Q^{m+1}:=q^{m+1}-q^m\) for \(m\geq 1\). Then \((V^{m+1}, Q^{m+1})\) solves
\begin{equation}\label{difiteration}
\begin{dcases}
\begin{aligned}
  \partial_t V^{m+1}-\Delta V^{m+1}+\nabla Q^{m+1}&=-\nabla \cdot (V^m\otimes v^m+v^{m-1}\otimes V^{m}+V^{m}\otimes W+W\otimes V^{m}),\\
  \nabla\cdot V^{m+1}&= 0
\end{aligned}
\end{dcases}
\text{  in }  Q_+,
\end{equation}
with zero initial and boundary data
\begin{equation*}
    V^{m+1}\mid_{x_n=0}=V^{m+1}\mid_{t=0}=0.
\end{equation*}
We then obtain 
\begin{align*}
    \left\|V^{m+1}\right\|_{L^{r}(Q_+)}+\left\|V^{m+1}\right\|_{L^{\infty}(Q_+)}&+\left\|\nabla V^{m+1}\right\|_{L^{r}(Q_+)}\\
    &\leq (C_1+C_2+C_3)\left\|V^{m}\otimes v^m+v_{m-1}\otimes V^m+V^m\otimes W+W\otimes V_m\right\|_{L^{r}(Q_+)}\\
    &\leq(C_1+C_2+C_3)\left(\left\|V^m\right\|_{L^{r}(Q_+)}\left(\left\|v^m\right\|_{L^{\infty}(Q_+)}+\left\|W\right\|_{L^{\infty}(Q_+)}\right)\right.\\
    &\left.+\left\|V^m\right\|_{L^{\infty}(Q_+)}\left(\left\|v^{m-1}\right\|_{L^{r}(Q_+)}+\left\|W\right\|_{L^{r}(Q_+)}\right)\right)\\
    &\leq 2(C_1+C_2+C_3)A\left(\left\|V^m\right\|_{L^{r}(Q_+)}+\left\|V^m\right\|_{L^{\infty}(Q_+)}\right),
\end{align*}
\begin{align*}
    \left\|\partial_t V^{m+1}\right\|_{L^{r}(Q_+)}+\left\|\nabla^2 V^{m+1}\right\|_{L^{r}(Q_+)}&\leq \frac{1}{2n}C_4\left\|\nabla\cdot(V^{m}\otimes v^m+v_{m-1}\otimes V^m+V^m\otimes W+W\otimes V_m)\right\|_{L^{r}(Q_+)}\\
    &\leq \frac{C_4}{2}\left(\left\|\nabla V^m\right\|_{L^{r}(Q_+)}\left\|v^m\right\|_{L^{\infty}(Q_+)}+\left\|V^m\right\|_{L^{\infty}(Q_+)}\left\|\nabla v^m\right\|_{L^{r}(Q_+)}\right.\\
    &\left.+\left\|\nabla v^{m-1}\right\|_{L^{r}(Q_+)}\left\|V^m\right\|_{L^{\infty}(Q_+)}+\left\|\nabla V^m\right\|_{L^{r}(Q_+)}\left\|v^{m-1}\right\|_{L^{\infty}(Q_+)}\right.\\
    &\left.+2\left\|\nabla W\right\|_{L^{r}(Q_+)}\left\|V^{m}\right\|_{L^{\infty}(Q_+)}+2\left\|W\right\|_{L^{\infty}(Q_+)}\left\|\nabla V^{m}\right\|_{L^{r}(Q_+)}\right)\\
    &\leq 2C_4A \left(\left\|\nabla V^{m}\right\|_{L^r(Q_+)}+\left\|V^{m}\right\|_{L^{\infty}(Q_+)}\right),
\end{align*}
and
\begin{align*}
    \left\|\nabla V^{m+1}\right\|_{L^{\infty}(Q_+)}&\leq \frac{C_5}{C_4}\left(\left\|\partial_t V^{m+1}\right\|_{L^{r}(Q_+)}+\left\|\nabla^2 V^{m+1}\right\|_{L^{r}(Q_+)}\right)\\
    &\leq \frac{C_5}{C_4}(2C_3A+2C_4A)\left(\left\|V^m\right\|_{L^{r}(Q_+)}+\left\|\nabla V^m\right\|_{L^{r}(Q_+)}+\left\|V^m\right\|_{L^{\infty}(Q_+)}\right)\\
    &=2C_5A\left(\left\|V^m\right\|_{L^{r}(Q_+)}+\left\|\nabla V^m\right\|_{L^{r}(Q_+)}+\left\|V^m\right\|_{L^{\infty}(Q_+)}\right).
\end{align*}
We now differentiate the equation (\ref{difiteration}) in \(x_j\) (\(j=1,\,\cdots,\, n-1\)) to get
\begin{equation}
\begin{dcases}
\begin{aligned}
  \partial_t \partial_j V^{m+1}-\Delta \partial_j V^{m+1}+\nabla \partial_j Q^{m+1}&=-\nabla \cdot (\partial_j V^m\otimes v^m+ V^m\otimes \partial_j v^m+\partial_j v^{m-1}\otimes V^m+v^{m-1}\otimes \partial_j V^m\\
  &~~~+\partial_j V^m\otimes W+V^m\otimes \partial_j W+\partial_j W\otimes V^m+W\otimes \partial_j V^m),\\
  \text{div } \partial_j V^{m+1}&= 0
\end{aligned}
\end{dcases}
\text{  in }  Q_+,
\end{equation}
with zero initial and boundary data
\begin{equation*}
    \partial_j V^{m+1}\mid_{x_n=0}=\partial_j V^{m+1}\mid_{t=0}=0.
\end{equation*}
By the maximal regularity for the Stokes equations, we find that,
\begin{align*}
    \left\|\partial_t \partial_j V^{m+1} \right\|_{L^{r}(Q_+)}&+\left\|\nabla^2 \partial_j V^{m+1} \right\|_{L^{r}(Q_+)}\\
    &\leq \frac{C_4}{2n}\left\|\nabla \cdot (\partial_j V^m\otimes v^m+ V^m\otimes \partial_j v^m+\partial_j v^{m-1}\otimes V^m+v^{m-1}\otimes \partial_j V^m\right.\\
  &\left.~~~+\partial_j V^m\otimes W+V^m\otimes \partial_j W+\partial_j W\otimes V^m+W\otimes \partial_j V^m)\right\|_{L^{r}(Q_+)}\\
  &\leq\frac{C_4}{2}\left[\left(\left\|\nabla\partial_j V^m\right\|_{L^{r}(Q_+)}+\left\|\nabla V^m\right\|_{L^{r}(Q_+)}\right)\left(\left\|v^m\right\|_{L^{\infty}(Q_+)}+\left\|v^{m-1}\right\|_{L^{\infty}(Q_+)}+2\left\|W\right\|_{L^{\infty}(Q_+)}\right.\right.\\
  &\left.+\left\|\nabla v^m\right\|_{L^{\infty}(Q_+)}+\left\|\nabla v^{m-1}\right\|_{L^{\infty}(Q_+)}+2\left\|\nabla W\right\|_{L^{\infty}(Q_+)}\right)\\
  &+\left(\left\|V^m\right\|_{L^{\infty}(Q_+)}+\left\|\nabla V^m\right\|_{L^{\infty}(Q_+)}\right)\left(\left\|\nabla v^m\right\|_{L^{r}(Q_+)}+\left\|\nabla v^{m-1}\right\|_{L^{r}(Q_+)}+2\left\|\nabla W\right\|_{L^{r}(Q_+)}\right.\\
  &\left.\left.+\left\|\nabla\partial_j v^m\right\|_{L^{r}(Q_+)}+\left\|\nabla \partial_j v^{m-1}\right\|_{L^{r}(Q_+)}+2\left\|\nabla\partial_j W\right\|_{L^{r}(Q_+)}\right)\right]\\
  &\leq 2C_4A\left(\left\|V^m\right\|_{L^{\infty}(Q_+)}+\left\|\nabla V^m\right\|_{L^{\infty}(Q_+)}+\left\|\nabla V^m\right\|_{L^{r}(Q_+)}+\left\|\nabla^2 V^m\right\|_{L^{r}(Q_+)}\right).
\end{align*}
Hence we find that for \(j\neq n\) and \(1\leq k\leq n\),
\begin{align*}
    \left\|\partial_j\partial_k V^{m+1}\right\|_{L^{\infty}(Q_+)}&\leq\frac{C_5}{C_4}\left(\left\|\nabla^2\partial_jV^{m+1}\right\|_{L^r(Q_+)}+\left\|\partial_t\partial_jV^{m+1}\right\|_{L^r(Q_+)}\right)\\
    &\leq 2C_5 A\left(\left\|V^m\right\|_{L^{\infty}(Q_+)}+\left\|\nabla V^m\right\|_{L^{\infty}(Q_+)}+\left\|\nabla V^m\right\|_{L^{r}(Q_+)}+\left\|\nabla^2 V^m\right\|_{L^{r}(Q_+)}\right),
\end{align*}
and thus
\begin{align*}
    \left\|\nabla\nabla' V^{m+1}\right\|_{L^{\infty}(Q_+)}&\leq 2n^2C_5A\left(\left\|V^m\right\|_{L^{\infty}(Q_+)}+\left\|\nabla V^m\right\|_{L^{\infty}(Q_+)}+\left\|\nabla V^m\right\|_{L^{r}(Q_+)}+\left\|\nabla \partial_j V^m\right\|_{L^{r}(Q_+)}\right)\\
    &\leq C_6A\left(\left\|V^m\right\|_{L^{\infty}(Q_+)}+\left\|\nabla V^m\right\|_{L^{\infty}(Q_+)}+\left\|\nabla V^m\right\|_{L^{r}(Q_+)}+\left\|\nabla^2 V^m\right\|_{L^{r}(Q_+)}\right).
\end{align*}
Finally we obtain that
\begin{align*}
    \left\|\partial_n^2 V_n^{m+1}\right\|_{L^{\infty}(Q_+)}&\leq \sum_{k=1}^{n-1}\left\|\partial_n\partial_k V_k^{m+1}\right\|_{L^{\infty}(Q_+)}\\
    &\leq C_6 A\left(\left\|V^m\right\|_{L^{\infty}(Q_+)}+\left\|\nabla V^m\right\|_{L^{\infty}(Q_+)}+\left\|\nabla V^m\right\|_{L^{r}(Q_+)}+\left\|\nabla^2 V^m\right\|_{L^{r}(Q_+)}\right).
\end{align*}
Thus summing all the above estimates gives
\begin{align*}
     \left\|V^{m+1}\right\|_{L^{r}(Q_+)}&+\left\|V^{m+1}\right\|_{L^{\infty}(Q_+)}+\left\|\nabla V^{m+1}\right\|_{L^{r}(Q_+)}+\left\|\nabla V^{m+1}\right\|_{L^{\infty}(Q_+)}\\
     &+\left\|\nabla^2 V^{m+1}\right\|_{L^{r}(Q_+)}+\left\|\nabla\nabla'V^{m+1}\right\|_{L^{\infty}(Q_+)}+\left\|\partial_n^2 V_n^{m+1}\right\|_{L^{\infty}(Q_+)}\\
     &\leq 12C'A\left( \left\|V^{m}\right\|_{L^{r}(Q_+)}+\left\|V^{m}\right\|_{L^{\infty}(Q_+)}+\left\|\nabla V^{m}\right\|_{L^{r}(Q_+)}+\left\|\nabla V^{m}\right\|_{L^{\infty}(Q_+)}+\left\|\nabla^2 V^m\right\|_{L^{r}(Q_+)}\right)\\
     &\leq 12C'c^*a\left( \left\|V^{m}\right\|_{L^{r}(Q_+)}+\left\|V^{m}\right\|_{L^{\infty}(Q_+)}+\left\|\nabla V^{m}\right\|_{L^{r}(Q_+)}+\left\|\nabla V^{m}\right\|_{L^{\infty}(Q_+)}+\left\|\nabla^2 V^m\right\|_{L^{r}(Q_+)}\right)\\
     &\leq \frac{1}{2}\left( \left\|V^{m}\right\|_{L^{r}(Q_+)}+\left\|V^{m}\right\|_{L^{\infty}(Q_+)}+\left\|\nabla V^{m}\right\|_{L^{r}(Q_+)}+\left\|\nabla V^{m}\right\|_{L^{\infty}(Q_+)}+\left\|\nabla^2 V^m\right\|_{L^{r}(Q_+)}\right).
\end{align*}
This implies that \((v^m, \nabla v^m, \nabla\nabla'v^m, \nabla^2 v^m, \partial_n^2 v^m)\) converges to \((v, \nabla v, \nabla\nabla'v,\nabla^2 v, \partial_n^2v)\) in
\begin{equation*}
    L^{\infty}(Q_+)\times L^{\infty}(Q_+)\times L^\infty(Q_+)\times L^r(Q_+)\times L^r(Q_+)
\end{equation*}
such that \(v\) solves (\ref{pertns})  with an appropriate distribution \(q\).
We now set \(u:=v+W\) and \(p:=q+\pi\). Then \(u\) becomes a weak solution of the Navier-Stokes equations (\ref{navierstokeseq-10}) in \(Q_+\) with the no-slip boundary condition (\ref{SS-20}). Also, then the claim that \(u\) satisfies \eqref{thm2estimate1} with \(Q_1^+\) in place of \(Q_2^+\) and \eqref{thm2estimate2} follows directly from the construction.


\section{Appendix}
\subsection{Proof of Lemma 5}\label{apdx81}
\begin{statement}{Lemma \ref{alg}}
If \(0<s<t<1\), \(0<a<1\) and \(c>0\), then
\begin{align*}
    \frac{e^{-\frac{x_n^2}{cs}}}{(1-t+s)^{1-a}}\lesssim \frac{1}{(x_n^2+1-t+s)^{1-a}}.
\end{align*}
\end{statement}
\begin{proof}
First note that the function \(f(t)=(t+1)e^{-\frac{t}{c}}\) (\(t>0\)), has its maximum at \(t=c-1\) and thus \(f(t)\leq ce^{-\frac{c-1}{c}}=:C\).
Then letting \(u=\frac{x_n^2}{s}\) gives
\begin{align*}
    e^{-\frac{x_n^2}{cs}}\leq C\frac{s}{x_n^2+s}\leq C\left(\frac{s}{x_n^2+s}\right)^{1-a},
\end{align*}
since \(0<a\leq \frac{1}{2}\).
Clearly \(e^{-\frac{x_n^2}{cs}}\lesssim C\), which implies that
\begin{align*}
    e^{-\frac{x_n^2}{cs}}\lesssim C\min \left\{1,\left(\frac{s}{x_n^2+s}\right)^{1-a}\right\}.
\end{align*}
Now we note that
\begin{align*}
    \min\left\{1,~ \frac{s}{x_n^2+s}\right\}\frac{1}{1-t+s}\leq \min \left\{\frac{1}{1-t+s},~ \frac{s}{1-t+s}\frac{1}{x_n^2+s}\right\}\leq \min\left\{\frac{1}{1-t+s},~\frac{1}{x_n^2+s}\right\},
\end{align*}
and 
\begin{align*}
    \min\left\{\frac{1}{1-t+s},~\frac{1}{x_n^2+s}\right\}\leq\frac{2}{x_n^2+1-t+s}.
\end{align*}
The above inequality follows since if \(x_n^2+s>1-t+s\), then 
\begin{align*}
    \frac{1}{x_n^2+s}=\frac{2}{x_n^2+x_n^2+2s}<\frac{2}{x_n^2+1-t+2s}<\frac{2}{x_n^2+1-t+s},
\end{align*}
while if \(x_n^2+s<1-t+s\), then \begin{align*}
    \frac{1}{1-t+s}=\frac{2}{(1-t)+(1-t)+2s}<\frac{2}{x_n^2+1-t+2s}<\frac{2}{x_n^2+1-t+s}.
\end{align*}
Thus we have that 
\begin{align*}
    \min\left\{1,~\frac{s}{x_n^2+s}\right\}\leq \frac{2(1-t+s)}{x_n^2+1-t+s},
\end{align*}
and hence 
\begin{align*}
    e^{-\frac{x_n^2}{cs}}\lesssim\frac{2^{1-a}(1-t+s)^{1-a}}{(x_n^2+1-t+s)^{1-a}}.
\end{align*}
We obtain the desired inequality.
\end{proof}

\subsection{Proof of Lemma 7}\label{apdx82}

\begin{statement}{Lemma \ref{integralidentity}}
For \(b,\, c>0\) and \(\alpha\in \left[\frac{1}{2},\,1\right)\), \(\beta \in \left[\frac{1}{2}\,,1\right]\),
\begin{equation*}
    \int_{0}^{1}\frac{1}{(u+c)^\alpha (u+b)^{\beta}}du\lesssim 
    \begin{cases}
\log\left(2+\frac{1}{\sqrt{b+c}}\right) & \textup{if } \alpha=\beta=\frac{1}{2},\\ 
\displaystyle\frac{1}{(c+b)^{\alpha+\beta-1}} & \textup{if } \frac{1}{2}\leq \alpha <1,\quad \frac{1}{2}\leq \beta <1,\\ 

\displaystyle\frac{1}{(b+c)^{\alpha}}\log\left(\frac{c}{b}+1\right) & \textup{if } \frac{1}{2}\leq \alpha <1,\quad \beta=1.
\end{cases}
\end{equation*}
\end{statement} 
\begin{proof}
We first show when \(b , \,c <1\).\\
1) If $\alpha=\beta=\frac{1}{2}$, it is direct that
\begin{equation}\label{logintegral}
     \int_{0}^{1}\frac{1}{(u+b)^{\beta}(u+c)^{\alpha}}du =2\log\frac{\sqrt{b+1}+\sqrt{c+1}}{\sqrt{b}+\sqrt{c}},
\end{equation}
which gives the last formula in the lemma.\\
    2) Next, we consider the case that $\alpha, \beta\in \left[\frac{1}{2}, 1\right)$. 
    Via the change of variable, we rewrite the integral as follows:
    \[
    \int_{0}^{1}\frac{1}{(u+b)^{\beta}(u+c)^{\alpha}}du 
    =\frac{1}{b^{\alpha+\beta-1}}\int_{0}^{1/b}\frac{1}{(u+\frac{c}{b})^{\alpha}(u+1)^{\beta}}du=:I.
    \]
    If \(c>\frac{1}{2}\), then using \(\frac{b+c}{4}<c\),
    \begin{equation*}
        I\leq \frac{1}{c^{\alpha}}\int_{0}^{1}\frac{1}{(u+b)^{\beta}}du\leq \frac{(1+b)^{1-\beta}}{c^{\alpha}}\lesssim\frac{1}{(b+c)^{\alpha+\beta-1}}.
    \end{equation*}
    We now treat the case $c<\frac{1}{2}$. We split the integral \(I\) as follows:
    \[
    I=\frac{1}{b^{\alpha+\beta-1}}\int_{0}^{c/b}\frac{1}{(u+\frac{c}{b})^{\alpha}(u+1)^{\beta}} du+\frac{1}{b^{\alpha+\beta-1}}\int_{c/b}^{1/b}\frac{1}{(u+\frac{c}{b})^{\alpha}(u+1)^{\beta}}du=:I_1+I_2.
    \]
 Estimating each term separately,  we have that
 \begin{align*}
      I_1 &\le \frac{1}{b^{\alpha+\beta-1}}\left(\frac{b}{c}\right)^{\alpha}\int_{0}^{c/b}\frac{1}{(u+1)^{\beta}}du\lesssim \frac{1}{b^{\beta-1}c^{\alpha}}\left[\left(\frac{c}{b}+1
 \right)^{1-\beta}-1 \right]\\
  &\lesssim \frac{(c+b)^{1-\beta}-b^{1-\beta}}{c^{\alpha}}
 \lesssim \frac{c^{1-\alpha}}{(c+b)^{\beta}}\lesssim \frac{1}{(c+b)^{\alpha+\beta-1}}.
 \end{align*}
On the other hand, we have that since \(c<\frac{1}{2}\),
\[
   I_2 \le \frac{1}{b^{\alpha-1}(c+b)^{\beta}}\int_{c/b}^{1/b}\frac{1}{\left(u+\frac{c}{b}\right)^{\alpha}} du\lesssim \frac{b^{1-\alpha}}{(c+b)^{\beta}}(1+c)^{1-\alpha}\lesssim \frac{1}{(c+b)^{\alpha+\beta-1}}.
\]  
\\
3) It remains to treat the case when $\alpha\in \left[\frac{1}{2}, 1\right)$ and $\beta=1$. Indeed, again considering \(c>\frac{1}{2}\), we have
\begin{equation*}
    I\leq\frac{1}{c^{\alpha}}\log\left(1+\frac{1}{b}\right)\lesssim\frac{1}{c^{\alpha}}\log\left(1+\frac{c}{b}\right)\lesssim\frac{1}{(b+c)^{\alpha}}\log\left(1+\frac{c}{b}\right),
\end{equation*}
and if $c<\frac{1}{2}$, we obtain that
\[
I_1 \le \frac{1}{b^{\alpha}}\int_{0}^{c/b}\frac{1}{(u+\frac{c}{b})^{\alpha}(u+1)}du.
\]
If $\frac{c}{b}<\frac{1}{2}$, then
\begin{equation*}
    I_1 \leq \frac{1}{b^{\alpha}}\int_{0}^{\infty}\frac{1}{(u+\frac{c}{b})^{\alpha}(u+1)}du\lesssim\frac{1}{b^{\alpha}}\lesssim\frac{1}{(b+c)^{\alpha}},
\end{equation*}
since the integral above is convergent.\\
If $\frac{c}{b}\ge \frac{1}{2}$, it is straight forward that
\begin{equation*}
    I_1 \le \frac{1}{c^{\alpha}}\log\left(1+\frac{c}{b}\right)\lesssim\frac{1}{(c+b)^{\alpha}} \log \left(\frac{c}{b}+1 \right).
\end{equation*}
For \(I_2\), we have that if \(\frac{c}{b}\geq\frac{1}{2}\), then
\begin{align*}
    I_2\lesssim \frac{1}{b^{\alpha}}\int_{\frac{c}{b}}^{\frac{1}{b}}\frac{1}{(u+1)^{\alpha+1}}du\lesssim \frac{1}{b^{\alpha}}\lesssim \frac{1}{(b+c)^{\alpha}},
\end{align*}
and if \(\frac{c}{b}<\frac{1}{2}\), then
\begin{align*}
    I_2\leq \frac{1}{b^{\alpha}}\int_{\frac{c}{b}}^{\frac{1}{b}}\frac{1}{\left(u+\frac{c}{b}\right)^{\alpha}(u+1)}du\leq  \frac{1}{b^{\alpha}}\int_{0}^{\infty}\frac{1}{\left(u+\frac{c}{b}\right)^{\alpha}(u+1)}du\lesssim \frac{1}{b^{\alpha}}\lesssim \frac{1}{(b+c)^{\alpha}},
\end{align*}
since the last integral is convergent.\\

We now show when \(b>1\) or \(c>1\).\\
1) If \(\frac{1}{2}\leq \alpha<1\) and \(\frac{1}{2}\leq \beta<1\), we find that if \(c\geq b\) (so that \(c>1\)), then
\begin{equation*}
    I\leq \frac{1}{c^{\alpha}}\int_{0}^{1}\frac{1}{(u+b)^\beta}du\lesssim \frac{(1+b)^{1-\beta}}{c^{\alpha}}\leq \frac{(c+b)^{1-\beta}}{c^{\alpha}}\lesssim \frac{1}{(c+b)^{\alpha+\beta-1}}.
\end{equation*}
If \(b>c\) (so that \(b>1\)), we may switch the role of \(b\) and \(\beta\) to \(c\) and \(\alpha\) to get the desired result.\\
2) If \(\alpha=\beta=\frac{1}{2}\), the result follows from the identity (\ref{logintegral}).\\
3) If \(\alpha\in \left[\frac{1}{2},1\right)\) and \(\beta=1\), we find that if \(c\geq b\), then
\begin{equation*}
    \int_{0}^{1}\frac{1}{(u+c)^{\alpha}(u+b)}du\leq \frac{1}{c^{\alpha}}\int_{0}^{1}\frac{1}{u+b}du=\frac{\log\left(1+\frac{1}{b}\right)}{c^{\alpha}}\lesssim \frac{\log\left(1+\frac{c}{b}\right)}{(c+b)^{\alpha}},
\end{equation*}
and finally if \(b>c\), then
\begin{equation*}
    \int_{0}^{1}\frac{1}{(u+c)^{\alpha}(u+b)}\leq \frac{1}{b}\int_{0}^{1}\frac{1}{(u+c)^{\alpha}}du\lesssim \frac{(1+c)^{1-\alpha}}{b}\lesssim \frac{1}{(c+b)^{\alpha}}.
\end{equation*}
This completes the proof.
\end{proof}
\newpage
\subsection{A figure of the spatial boundary data}\label{apdx83}
The following is a figure of a spatial boundary data \(\mathcal{G}\) satisfying all the assumptions given in Theorem 4.

\begin{figure}[b]
\includegraphics[width=14cm]{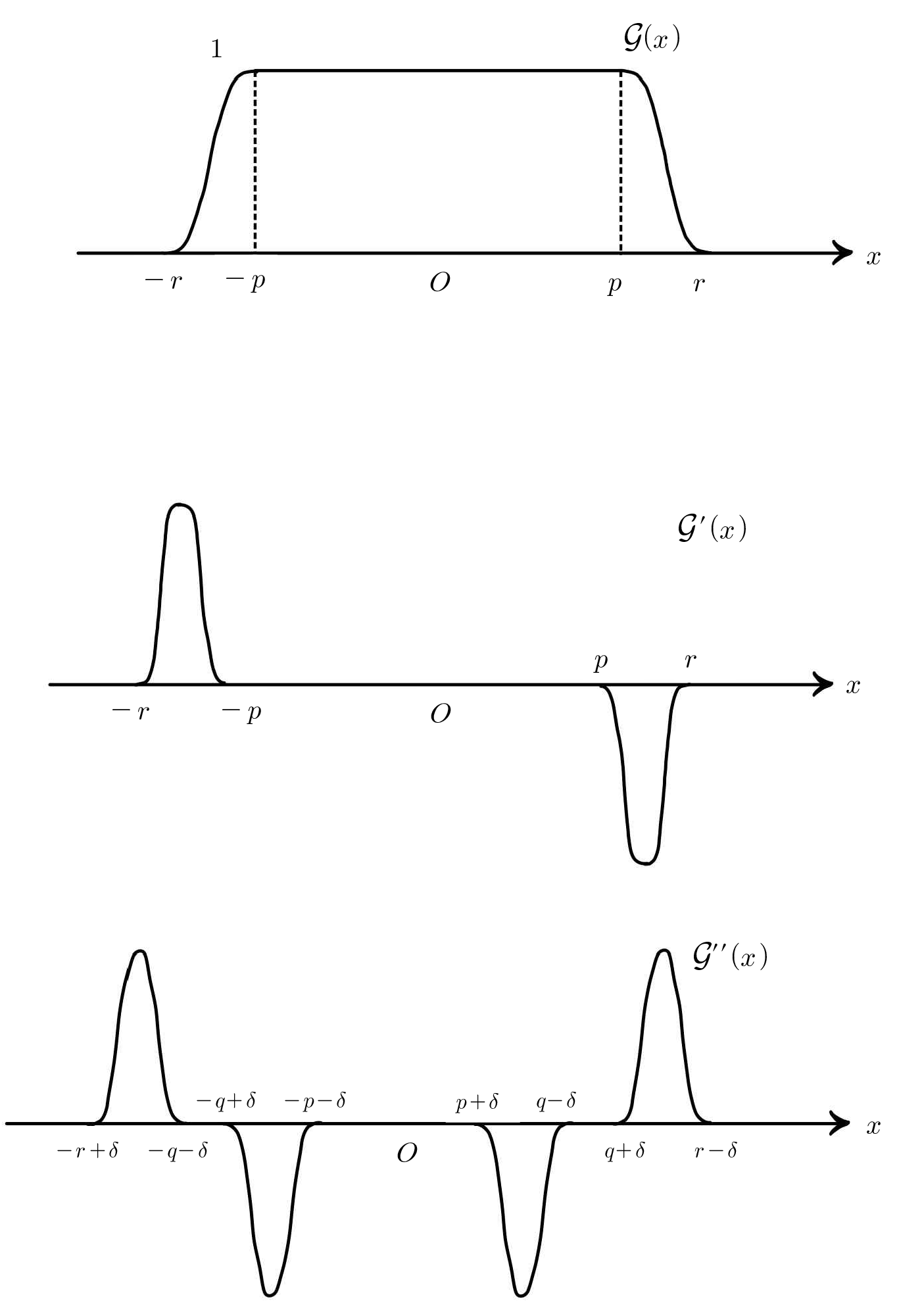}
\caption{Figure of \(\mathcal{G}, \mathcal{G}', \mathcal{G}''\),  \(\delta:=\frac{1}{20\sqrt{n-1}}\)}
\centering
\end{figure}

\subsection{A figure of the domain}\label{apdx84}
We also include the domain for \(u\) such that the conclusions of Theorem 5 hold when \(n=3\), \(k=1\) and \(i=2\) (\(i\neq k\) case), and \(n=3\) and \(k=1\) (\(i=k\) case). The "+" and "\(-\)" in the figures denote the respective signs of \(\partial_n^2 u(x,1)\).
\begin{figure}[!htb]
\includegraphics[width=12cm]{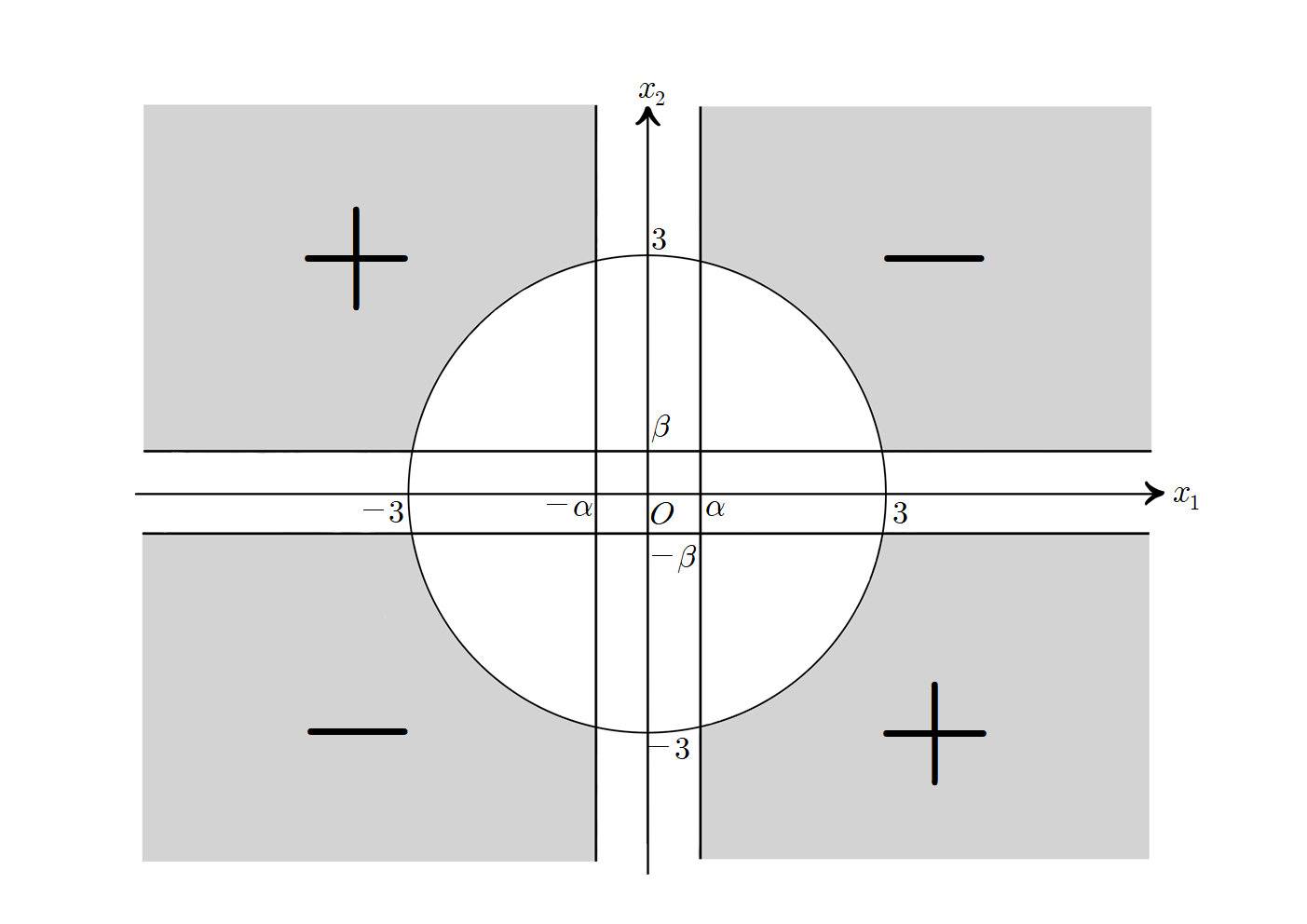}
\caption{(Case \(i\neq k\)) Graph of the domain for \(n=3\), \(i=1\) and \(k=2\).}

\includegraphics[width=15cm]{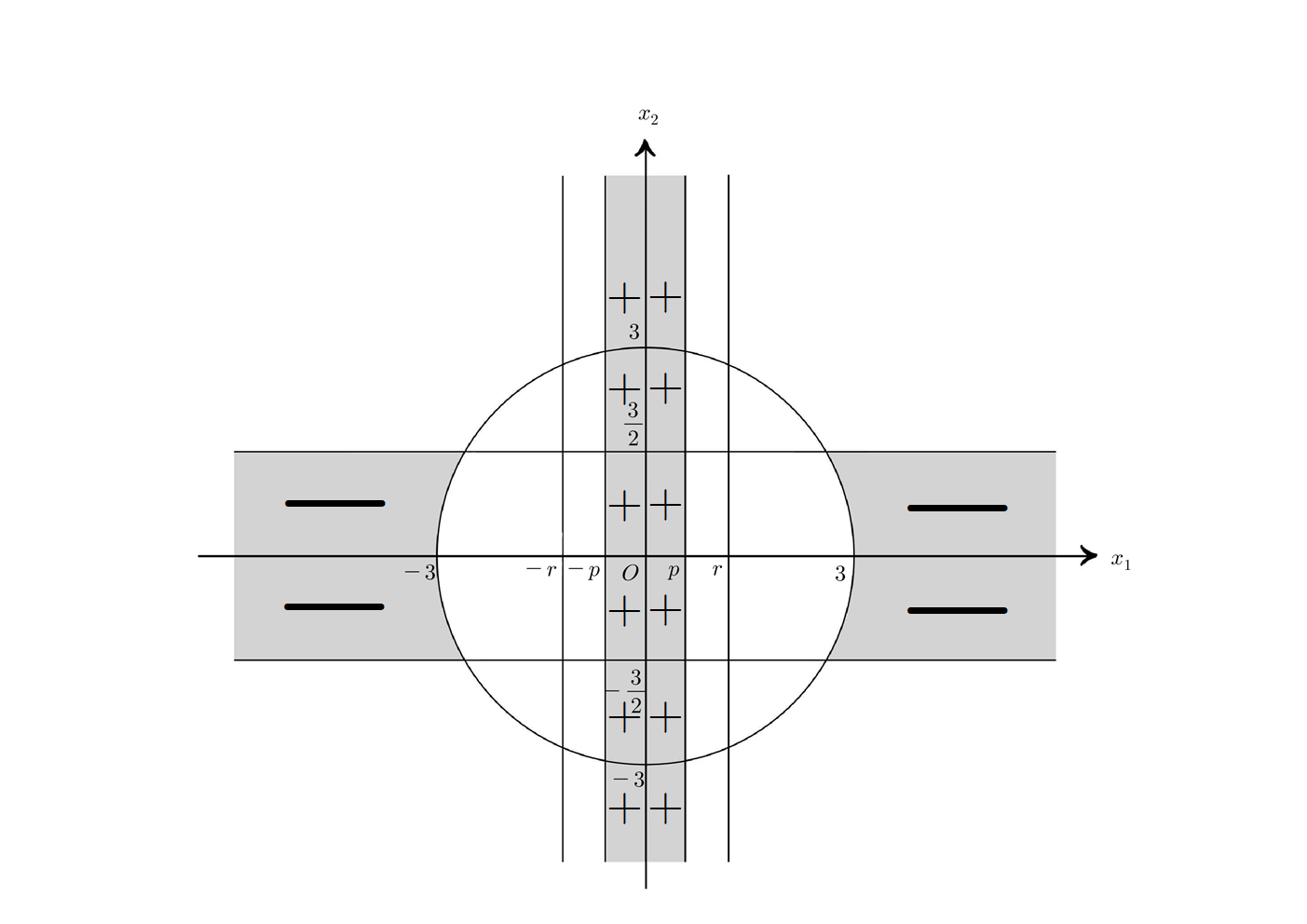}
\caption{(Case \(i=k\)) Graph of the domain for \(n=3\) and \(k=1\).}
\centering
\end{figure}

\newpage

\section*{Acknowledgement}
K. Kang is supported by NRF-2019R1A2C1084685. C. Min is supported by NRF-2019R1A2C1084685.



\begin{equation*}
\left.
\begin{array}{cc}
{\mbox{Kyungkeun Kang}}\qquad&\qquad {\mbox{Chanhong Min}}\\
{\mbox{Department of Mathematics }}\qquad&\qquad
 {\mbox{Department of Mathematics}} \\
{\mbox{Yonsei University
}}\qquad&\qquad{\mbox{Yonsei University}}\\
{\mbox{Seoul, Republic of Korea}}\qquad&\qquad{\mbox{Seoul, Republic of Korea}}\\
{\mbox{kkang@yonsei.ac.kr }}\qquad&\qquad
{\mbox{michelcamilo15@gmail.com }}
\end{array}\right.
\end{equation*}

\end{document}